\documentclass[11pt, leqno]{amsart} 
\usepackage{amssymb,amscd,amsfonts,amsbsy}
\usepackage{latexsym}
\usepackage{exscale}
\usepackage{amsmath,amsthm,amsfonts,amsxtra}
\usepackage{mathrsfs}
\usepackage{xcolor}
\usepackage[colorlinks=true, linkcolor=blue, citecolor=red, urlcolor=red, 
]{hyperref}
\usepackage{esint} 
\usepackage{stmaryrd}
\usepackage{pifont}
\usepackage{bm}

\usepackage[utf8]{inputenc}

\calclayout

\newcommand{\black}{\color{black}}

\allowdisplaybreaks

\makeatletter
\@namedef{subjclassname@2020}{%
	\textup{2020} Mathematics Subject Classification}
\makeatother

\theoremstyle{plain}
\newtheorem{theorem}[equation]{Theorem}
\newtheorem{proposition}[equation]{Proposition}
\newtheorem{lemma}[equation]{Lemma}
\newtheorem{corollary}[equation]{Corollary}

\theoremstyle{definition}
\newtheorem{definition}[equation]{Definition}

\numberwithin{equation}{section} 

\def\A{\mathcal{A}}
\def\B{\mathbf{B}}

\def\D{\mathcal{D}}

\def\G{\mathcal{G}}
\def\H{\mathcal{H}}
\def\K{\mathbf{K}}
\def\M{\mathcal{M}}
\def\N{\mathbb{N}}
\def\Q{\mathcal{Q}}
\def\R{\mathbb{R}}
\def\Rd{\mathbb{R}^d}
\def\S{\mathcal{S}}
\def\bS{\mathbb{S}}
\def\T{\mathbf{T}}
\def\bT{\mathbb{T}}
\def\W{\mathcal{W}}
\def\Z{\mathbb{Z}}
\def\c{\boldsymbol{c}}
\def\m{\boldsymbol{m}}
\def\q{\mathfrak{q}}
\def\w{\omega}

\def\Id{\operatorname{Id}}
\def\diam{\operatorname{diam}}

\def\dist{\operatorname{dist}}

\def\inter{\operatorname{int}}
\def\conv{\operatorname{conv}}

\def\supp{\operatorname{supp}}
\def\full{\operatorname{full}}
\def\lac{\operatorname{lac}}
\def\tri{\operatorname{tri}}
\def\rank{\operatorname{rank}}

\DeclareMathOperator*{\esssup}{ess\,sup}

\newcommand{\br}[1]{\langle #1\rangle}

\begin{document}

\title{Multi-scale analysis for multilinear operators}

\author{Mingming Cao}
\address{Mingming Cao\\
Institute for Advanced Study in Mathematics\\ 
Harbin Institute of Technology\\ 
Harbin 150001\\
People's Republic of China}
\email{mingming.cao@hit.edu.cn}

\author{Binyu Du}
\address{Binyu Du\\
Institute for Advanced Study in Mathematics\\ 
Harbin Institute of Technology\\ 
Harbin 150001\\
People's Republic of China} 
\email{bydu@stu.hit.edu.cn}

\author{K\^{o}z\^{o} Yabuta}
\address{K\^{o}z\^{o} Yabuta\\
Research Center for Mathematics and Data Science\\
Kwansei Gakuin University\\
Gakuen 2-1, Sanda 669-1337, Japan}
\email{kyabuta3@kwansei.ac.jp}

\thanks{The first author was supported by the National Natural Science Foundation of China (No. 12571101, 12526417, and 12671117).}

\year=2026 \month=10 \day=12

\date{\today}




\begin{abstract}
We develop a unified multi-scale framework for sparse domination of multilinear operators beyond the multilinear Calder\'{o}n--Zygmund theory. More precisely, under mild hypotheses on a support condition, weak type and restricted strong type uniform bounds for finite partial sums, and single-scale improving and regularity estimates, we establish $(r_1, \ldots, r_m, s')$-sparse form bound for a family of localized $m$-linear operators uniformly over all finite partial sums. This sparse bound yields weighted norm inequalities for the corresponding multilinear Muckenhoupt classes $A_{\vec{p}, (\vec{r}, s)}$, including the quasi-Banach range. As applications, we determine the optimal open sparse range for multilinear rough singular integrals, and characterize sparse domination for bilinear pseudo-differential operators. Further applications include sparse bounds for bilinear oscillatory Fourier multipliers, bilinear spherical maximal operators, and bilinear Bochner--Riesz means.
\end{abstract}

\maketitle

\section{Introduction}

\subsection{Motivation} 
Our starting point is the multi-scale structure shared by several classes of multilinear operators in harmonic analysis. Examples include multilinear rough singular integrals \cite{DS, GHH, GHHP23, GHHP24}, bilinear Bochner--Riesz means \cite{BGSY, BCS, JS}, bilinear spherical maximal operators \cite{BFO, BFOPZ, JL}, triangle averaging maximal operators \cite{BF, IPS20, IPS22, PS}, and bilinear pseudo-differential operators \cite{KMT, MT20, RS}. These operators are naturally studied through decompositions in the spatial or frequency variables, or through families of dilations organized into dyadic scales. Although the estimates for individual pieces depend on different features including angular integrability, curvature, oscillation, and symbol regularity, their multi-scale organization suggests a common approach. This observation motivates us to seek a unified framework to identify the estimates needed to pass from individual scales to the full operators.

Sparse domination provides a natural objective for such a framework. It controls an operator or its associated multilinear form by positive expressions built from local averages over sparse collections of cubes \cite{CDO, Ler, LN}. Such estimates contain more local information than a single operator-norm inequality and provide a common route to families of sharp or quantitative weighted inequalities \cite{BFP, CIRXY, CDO, LMO}. In particular, domination at the level of multilinear forms has made it possible to treat operators beyond classical  Calder\'{o}n--Zygmund theory \cite{BRS, BBL, CO, CCDO, Hu, KL, KraL, LMR, LS}. Rather than proving separate weighted estimates for each operator, one can first establish a sparse bound and then apply the weighted theory of the corresponding positive forms. Therefore, this naturally leads us to ask which common multi-scale properties of the operators above can be used to establish their sparse domination.

Important steps in this direction were taken by Lacey \cite{Lac} and Oberlin \cite{Obe}. Lacey used localized $L^p$-improving estimates and their continuity versions to obtain sparse bounds for the full and lacunary spherical maximal functions. Soon after, adapting this approach to a class of model singular and maximal Radon transforms, Oberlin combined single-scale improving estimates with regularization and Fourier decay to obtain sparse bounds with logarithmic factors in the local averages. Subsequently, Beltran--Roos--Seeger developed these ideas into an abstract theory in their monograph \cite{BRS}, which presented some useful applications include Fourier multipliers, maximal functions, square functions, and variation norm operators.

Motivated by these developments above, the purpose of this article is to establish a multi-scale framework of sparse domination for multilinear operators. We formulate common hypotheses on localized operator families that can be verified using the different analytic estimates available for the classes above. An additional difficulty in the multilinear setting is the interaction of good and bad functions in different input variables. Following the idea of using auxiliary exponents in the bilinear spherical analysis of \cite{BFOPZ}, we incorporate suitable auxiliary improving and regularity estimates into our abstract hypotheses. The resulting theorem yields sparse bounds uniform over finite partial sums, with logarithmic dependence on the regularity constant, and will prove highly useful in applications.

\subsection{Main theorem}
Before stating our main theorem, we introduce some notation and basic concepts. Let $\B_1, \ldots, \B_m, \B$ be Banach spaces. Let $\mathscr{S}_{\B_i}$ denote the collection of all $\B_i$-valued compactly supported simple functions. If an $m$-linear operator
\begin{align*}
T: \mathscr{S}_{\B_1} \times \cdots \times \mathscr{S}_{\B_m} \to \text{weakly measurable $\B$-valued functions}
\end{align*}
satisfies that for every $f_1 \in \mathscr{S}_{\B_1}, \ldots, f_m \in \mathscr{S}_{\B_m}, f_{m+1} \in \mathscr{S}_{\B^*}$ the $(m+1)$-linear pairing
\begin{align*}
\langle T(\vec{f}), f_{m+1} \rangle 
:= \int_{\R^d} \langle T(\vec{f})(x), f_{m+1}(x) \rangle_{(\B, \B^*)} \, dx
\end{align*}
is well defined, then we write $T \in \mathscr{L}_m$.

Define 
\begin{align*}
T \circ_i \Delta_h(\vec{f})
:= T(f_1, \ldots, f_{i-1}, \Delta_h f_i, f_{i+1}, \ldots, f_m), \quad i=1, \ldots, m,
\end{align*}
where $\Delta_h f(x) := f(x+h) - f(x)$. Throughout this paper, we assume
\begin{align*}
1 < r_1, \ldots, r_m, s < \infty
\quad \text{ and } \quad
\frac1s < \frac1r := \sum_{i=1}^m \frac{1}{r_i}.
\end{align*}
Let $\{T_j\}_{j\in\mathbb Z}$ be a family of $m$-linear operators in $\mathscr{L}_m$. To state the main theorem conveniently, we make the following assumptions:
\begin{list}{\rm (\theenumi)}{\usecounter{enumi}\leftmargin=1.3cm \labelwidth=1cm \itemsep=0.1cm \topsep=0.2cm \renewcommand{\theenumi}{H\arabic{enumi}}}

\item\label{H1} The support condition: 
\begin{align*}
\supp T_j(\vec{f})
\subset \big\{x \in \Rd: \dist(x, \supp(f_i)) \le A_1 \, 2^j, i=1, \ldots, m \big\},
\end{align*}
for every $j \in \Z$ and $f_i \in \mathscr{S}_{\B_i}$, $i=1, \ldots, m$.

\item\label{H2} The weak type estimate:
\begin{align*}
\sup_{N_1 \le N_2}
\bigg\|\sum_{j=N_1}^{N_2} T_j \bigg\|_{L^{r_1}(\B_1) \times \cdots \times L^{r_m}(\B_m) \to L^{r, \infty}(\B)}
\le A_2.
\end{align*}

\item\label{H3} The restricted strong estimate:
\begin{align*}
\sup_{N_1 \le N_2} \bigg\|\sum_{j=N_1}^{N_2} T_j \bigg\|_{L^{s_1, 1}(\B_1) \times \cdots \times L^{s_m, 1}(\B_m) \to L^s(\B)}
\le A_3,
\end{align*}
for some $\vec{s} \in (1, \infty)^m$ with $\frac1s = \sum_{i=1}^m \frac{1}{s_i}$.

\item\label{H4} The single-scale improving estimate:
\begin{align*}
\|T_j\|_{L^{r_1}(\B_1) \times \cdots \times L^{r_m}(\B_m) \to L^s(\B)}
& \le A_4 \, 2^{-jd(\frac1r - \frac1s)}, \quad \forall j \in \Z.
\end{align*}
In addition, there exist auxiliary exponents $r_1^*, \ldots, r_m^* \in [1, \infty]$ such that
\begin{align*}
\|T_j\|_{L^{r_1(I)}(\B_1) \times \cdots \times L^{r_m(I)}(\B_m) \to L^{s(I)}(\B)}
& \le A_4 \, 2^{-jd \sum_{i \in I} \frac{1}{r_i^*}}, \quad \forall j \in \Z,
\end{align*}
for every nonempty proper subset $I \subsetneq \{1, \ldots, m\}$, where $r_i(I) = r_i^*$ if $i \in I$, $r_i(I) = r_i$ if $i \notin I$, and $\frac{1}{s(I)} = \sum_{i \notin I} \frac{1}{r_i}$.

\item\label{H5} The single-scale regularity: each 
$T'_j \in \{\Delta_{2^j h} T_j, T_j \circ_1 \Delta_{2^j h}, \ldots, T_j \circ_m \Delta_{2^j h}\}$, 

\begin{align*}
\|T'_j\|_{L^{r_1}(\B_1) \times \cdots \times L^{r_m}(\B_m) \to L^s(\B)}
\le A_5 |h|^{\vartheta} 2^{-jd(\frac1r - \frac1s)}, \quad \forall j \in \Z,
\end{align*}
for all $|h| \le 1$ and for some $\vartheta \in (0, 1]$. With the same auxiliary exponents $r_1^*, \ldots, r_m^*$ as in \eqref{H4}, we further assume that
\begin{align*}
\|T'_j\|_{L^{r_1(I)}(\B_1) \times \cdots \times L^{r_m(I)}(\B_m) \to L^{s(I)}(\B)}
\le A_5 |h|^{\vartheta} \, 2^{-jd \sum_{i \in I} \frac{1}{r_i^*}}, \quad \forall j \in \Z,
\end{align*}
for all $|h| \le 1$ and for every nonempty proper subset $I \subsetneq \{1, \ldots, m\}$, where $r_i(I) = r_i^*$ if $i \in I$, $r_i(I) = r_i$ if $i \notin I$, and $\frac{1}{s(I)} = \sum_{i \notin I} \frac{1}{r_i}$.
\end{list}

A {\tt dyadic child} of a cube $Q \subset \Rd$ is any of the $2^d$ congruent subcubes obtained by bisecting $Q$ in each coordinate direction. Let $\D(Q)$ consist of $Q$ and all its dyadic descendants. A collection $\mathscr{D}$ of cubes is said to be a {\tt dyadic lattice} if it satisfies
\begin{enumerate}
\item[$(1)$] For every $Q \in \mathscr{D}$, each child of $Q$ is in $\mathscr{D}$ as well;

\item[$(2)$] For every cubes $Q_1, Q_2 \in \mathscr{D}$, there exists $Q \in \mathscr{D}$ such that $Q_1, Q_2 \in \D(Q)$;

\item[$(3)$] For every compact set $K \subset \Rd$, there exists a cube $Q \in \mathscr{D}$ containing $K$.
\end{enumerate}
Given $\gamma \in (0, 1)$, a subset $\S$ of a dyadic lattice is called {\tt $\gamma$-sparse} if for every $Q \in \S$ there exists a measurable set $E_Q \subset Q$ such that $|E_Q| \ge \gamma |Q|$ and the family $\{E_Q: Q \in \S\}$ is pairwise disjoint. Then define the maximal sparse form by
\begin{align*}
\Lambda_{\vec{r}, r_{m+1}}^*(\vec{f}, f_{m+1})
:= \sup_{\S: \gamma \text{-sparse}} 
\sum_{Q \in \S} |Q| \prod_{i=1}^{m+1} \langle |f_i|_{\B_i} \rangle_{Q, r_i},
\end{align*}
whenever $r_i \ge 1$ and $f_i \in \mathscr{S}_{\B_i}$, $i=1, \ldots, m+1$.

Our main theorem is formulated as follows.

\begin{theorem}\label{thm:main}
If the family $\{T_j\}_{j \in \Z}$ of $m$-linear operators in $\mathscr{L}_m$ satisfies the hypotheses \eqref{H1}--\eqref{H5} with $A_i \in (0, \infty)$, $i=1, \ldots, 5$, then
\begin{align*}
\sup_{\substack{N_1, N_2 \in \Z \\ N_1 \le N_2}} 
\bigg| \bigg\langle \sum_{j = N_1}^{N_2} T_j(\vec{f}), \, f_{m+1} \bigg\rangle_{(\B, \B^*)} \bigg|
\lesssim A_0 \, \Lambda_{\vec{r}, s'}^*(\vec{f}, f_{m+1}).
\end{align*}
for all $f_i \in L_c^{\infty}(\Rd; \B_i)$, $i=1, \ldots, m$, and $f_{m+1} \in L_c^{\infty}(\Rd; \B^*)$, where 
\begin{align*}
A_0 := A_2 + A_3 + A_4 \log_2 (2 + A_4^{-1} A_5).
\end{align*}
\end{theorem}

Our proof is a hybrid of ideas in \cite{BRS, Lac, Obe} and additional multilinear techniques. The core argument combines an induction on the number of scales with a common stopping-time decomposition. Unfortunately, in the quasi-Banach range $r \in (0, 1)$, the pairing cannot be estimated by the weak type estimate \eqref{H2}. To address this issue, our proof incorporates the maximal terms of both individual  functions and the partial sum operator into the stopping condition, and uses a Whitney decomposition to construct stopping cubes with controlled total volume. Inside each cube, it further selects a subset of large measure on which both the partial sum operator and its local version are controlled. We then restrict the average of the function $f_{m+1}$ to this subset, and guarantee cancellation of the bad part at the same time. In addition, the mixed good-bad terms are split into the full term and the all-good term, which can be bounded using weak type estimate and local pointwise estimates. Beyond above, another major multilinear difficulty is to control interactions among bad functions supported at different scales without accumulating losses in the recursive constants. To overcome this difficulty, we organize these interactions by the largest scale separation and use cancellation in a single selected variable. Moreover, auxiliary improving estimates  allow one to estimate each combination of good and bad functions by the total measure of the relevant bad cubes, while a near-far scale decomposition yields logarithmic dependence on the regularity constant.

\subsection{Applications of Theorem \ref{thm:main}}
Let us present several applications of Theorem \ref{thm:main}, which clearly illustrate the effectiveness and flexibility of our abstract sparse domination. Since weighted norm inequalities are a direct consequence of Theorems \ref{thm:main} and \ref{thm:form}, we omit their explicit statements. 

\subsubsection{\bf{Multilinear rough singular integrals}}
Given $\Omega \in L^q(\bS^{md-1})$ with $q \in [1, \infty]$ and $\int_{\bS^{md-1}} \Omega\, d\sigma=0$, we define the {\tt multilinear rough singular integral operator} $T_{\Omega}$ by 
\begin{equation*}
T_{\Omega}(f, g)(x) 
= \mathrm{p.v.} \int_{\R^{md}} K_{\Omega}(x-\vec{y}) \prod_{i=1}^m f_i(y_i) \, d\vec{y},  
\end{equation*}
where the rough kernel $K_{\Omega}$ is given by 
\begin{equation*}
K_{\Omega}(\vec{y}) 
:= \frac{\Omega(\vec{y}/|\vec{y}|)}{|\vec{y}|^{md}} 
\quad \text{with} \quad \vec{y} = (y_1, \ldots, y_m). 
\end{equation*}
Given $N \ge 2$ and $q \in [1, \infty]$, define the regions
\begin{align*}
\H_N^1(q) 
& := \bigg\{\vec{a} \in (0, 1)^N: 
\sum_{i \in I} a_i + \frac{|I| - 1}{q} < |I|, \, \, 
\forall \varnothing \ne I \subsetneq \{1, \ldots, N \}\bigg\},
\\
\H_N^d(q) 
& := \bigg\{\vec{a} \in (0, 1)^N: 
\sum_{i \in I} a_i + \frac{|I| - 1}{q} < |I|, \, \, 
\forall \varnothing \ne I \subseteq \{1, \ldots, N \}\bigg\},
\end{align*}
for every $d \ge 2$. Note that $\H_N^d(q) \subsetneq \H_N^1(q)$ for any $q \ne \infty$ and $d \ge 2$, and the only difference between them is the full-set condition.

We first establish sparse domination for $T_{\Omega}$ with an optimal open exponent range in all dimensions.

\begin{theorem}\label{thm:rough}
Let $d, m \ge 1$, $q \in (1, \infty]$, and $1 < r_1, \ldots, r_m, s < \infty$ with $\frac1s < \sum_{i=1}^m \frac{1}{r_i}$. If $\big(\frac{1}{r_1}, \ldots, \frac{1}{r_m}, \frac{1}{s'} \big) \in \H_{m+1}^d(q)$, then 
\begin{align*}
|\langle T_{\Omega}(\vec{f}), f_{m+1} \rangle|
\lesssim \|\Omega\|_{L^q} \, \Lambda_{\vec{r}, s'}^*(\vec{f}, f_{m+1}),
\end{align*}
for all $\Omega \in L^q(\bS^{md-1})$ with $\int_{\bS^{md-1}} \Omega\, d\sigma=0$, and 
for all $f_1, \ldots, f_{m+1} \in L_c^{\infty}(\Rd)$. Conversely, if the above sparse bound holds uniformly in $\Omega$ and $(\vec{f}, f_{m+1})$, then $\big(\frac{1}{r_1}, \ldots, \frac{1}{r_m}, \frac{1}{s'} \big) \in \overline{\H_{m+1}^d(q)}$.
\end{theorem}

To present the optimal open symmetric sparse range, we introduce the critical exponent by
\begin{equation*}
r_c = r_c(d, m, q) := 
\begin{cases}
\frac{m}{m-(m-1)/q}, & d = 1,
\\
\frac{m+1}{m+1-m/q}, & d \ge 2.
\end{cases}
\end{equation*}

The following result is a direct consequence of Theorem \ref{thm:rough}. 

\begin{corollary}\label{cor:rough}
Let $d, m \ge 1$ and $q \in (1, \infty]$. Then for any $r > r_c$, 
\begin{align*}
|\langle T_{\Omega}(\vec{f}), f_{m+1} \rangle|
\lesssim_r \|\Omega\|_{L^q} \, \Lambda_{r, \ldots, r}^*(\vec{f}, f_{m+1}),
\end{align*}
for all $\Omega \in L^q(\bS^{md-1})$ with $\int_{\bS^{md-1}} \Omega\, d\sigma=0$, and for all $f_1, \ldots, f_{m+1} \in L_c^{\infty}(\Rd)$. Conversely, for any $r \in (1, \infty)$, if the above sparse bound holds uniformly in $\Omega$ and $(\vec{f}, f_{m+1})$, then $r \ge r_c$. Thus, $(r_c, m+1)$ is the optimal open symmetric sparse range in the nontrivial regime.
\end{corollary}

Theorem \ref{thm:rough} and Corollary \ref{cor:rough} not only identify the optimal open sparse range, but also recover Barron’s bilinear result for $\Omega \in L^{\infty}(\bS^{2d-1})$ \cite{Bar} and improve the sparse range of Grafakos--Wang--Xue for $\Omega \in L^q(\bS^{2d-1})$ with $q>4/3$ \cite{GWX}. Our proof employs a different mechanism from those of the two works, both of which only obtained sparse bounds in terms of symmetric exponents. Beyond that, combined with sparse-to-weighted estimates (see Theorem \ref{thm:form}), our sparse bounds imply weighted estimates in terms of weights class $A_{\vec{p}, (\vec{r}, s)}$ with $\big(\frac{1}{r_1}, \ldots, \frac{1}{r_m}, \frac{1}{s'} \big) \in \H_{m+1}^d(q)$ and $(\vec{r}, s) \prec \vec{p}$. This indeed recovers unweighted boundedness in the optimal region $\H_m(q)$ given by Dosidis--Slav\'{i}kov\'{a} \cite{DS}. Although the sparse region $\H_{m+1}^1(q)$ is larger in dimension one, its projection onto the input coordinates remains $\H_m(q)$, so the optimal unweighted range is unchanged. The additional freedom in the testing exponent instead strengthens the sparse estimates and relaxes the corresponding weight conditions. 

The sufficiency proof builds on the frequency-decay estimates of \cite{DS, GHH}. The main difficulty is to preserve this decay under spatial localization while establishing the improving estimates required by our abstract Theorem \ref{thm:rough}. Global operator bounds alone do not supply these local estimates, particularly for quasi-Banach outputs and auxiliary exponents being $\infty$. A smooth kernel localization argument preserves the operator bounds even in the quasi-Banach range (see Lemma \ref{lem:TKfg} \eqref{TKfg-4}), while interpolation in both the kernel and functions determines the improving region. At low frequencies, cancellation provides summable decay; at high frequencies, the logarithmic dependence of sparse constants on kernel regularity allows the growth of the regularity bounds to be absorbed by the operator decay. On the other hand, the one-dimensional improvement rests on an additional geometric feature of the localized $(m+1)$-form (see Lemma \ref{lem:imp}). After passing to polar coordinates, pairing one factor with testing functions permits a nondegenerate change of spatial and radial variables. This gives an endpoint estimate for $\Omega \in L^1(\bS^{md-1})$, whose interpolation with the bounded for $\Omega \in L^{\infty}(\bS^{md-1})$ yields the larger improving region. The resulting gain removes the full-set restriction and accounts for the smaller symmetric threshold $r_c$ in dimension one.

The necessity argument adapts the concentration constructions of \cite{DGHST, GHHP24} to the $(m+1)$-linear form and detects conditions involving the exponent $s$ that are not supplied by the unweighted counterexamples alone. The functions $\Omega$ are concentrated on the directions determined by the geometry of suitably chosen supports; an odd construction ensures cancellation. Concentration on small balls/intervals yields the proper-subset condition. In dimensions at least two, thin tubes detect full-set condition through concentration in transverse directions. This obstruction disappears in dimension one, which is consistent with the stronger localized improving estimate. These constructions establish optimality uniformly for mean-zero functions $\Omega \in L^q(\bS^{md-1})$. Although the sufficient and necessary conditions match at the level of open regions, the general sparse boundary cases remain undecided.

A typical example of the rough bilinear operators is the Calder\'{o}n commutator defined in \cite{P.C} as 
\begin{align*}
\mathcal{C}(a, f)(x)
:= \text{p.v.} \int_{\R} \frac{A(x) - A(y)}{|x - y|^2} f(y) \, dy, 
\end{align*}
where $a$ is the derivative of $A$. The boundedness of $\mathcal{C}(a, f)$ in the full range of exponents $p_1, p_2 \in (1, \infty)$ was established in \cite{C.C}. It was shown in \cite{P.C} that the Calder\'{o}n 
commutator can be written as 
\begin{align*}
\mathcal{C}(a, f)(x) 
:= \text{p.v.} \int_{\R \times \R} K(x-y, x-z) f(y) a(z) \, dy \, dz,  
\end{align*}
with the kernel 
\begin{align*}
K(y, z) = \frac{e(z) - e(z-y)}{y^2} 
= \frac{\Omega((y, z)/|(y, z)|)}{|(y, z)|^2}, 
\end{align*}
where $e(t)=1$ if $t>0$ and $e(t)=0$ if $t<0$. Observe that $K(y, z)$ is odd and homogeneous of degree $-2$ whose restriction on $\bS^1$ is $\Omega(y,z)$. Since $\Omega$ is odd and bounded, Theorem \ref{thm:rough} and Corollary \ref{cor:rough} can be applied to $\mathcal{C}(a, f)$.

\subsubsection{\bf{Bilinear pseudo-differential operators}}
Given a function $\sigma$ on $\Rd \times \R^{2d}$, the {\tt bilinear pseudo-differential operator} $T_{\sigma}$ is defined by
\begin{align*}
T_{\sigma}(\vec{f})(x) := \int_{\R^{2d}} \sigma(x, \vec{\xi})
e^{2\pi i x \cdot (\xi_1 + \xi_2)} \widehat{f}_1(\xi_1) \widehat{f}_2(\xi_2) \, d\vec{\xi},
\end{align*}
for all $f_1, f_2 \in \S(\Rd)$, where $\widehat{f}$ is the Fourier transform of $f$. Given $m \in \R$, $\rho, \delta \in [0, 1]$, and $N \in \N$, let $BS_{\rho,\delta}^{m, N}$ consist of all $\sigma \in \mathscr{C}^N(\R^{3d})$ such that
\begin{align*}
\|\sigma\|_{BS_{\rho, \delta}^{m, N}} 
:= \max_{|\alpha| + |\beta| + |\gamma| \le N} \sup_{x, \xi, \eta}
\frac{|\partial_x^{\alpha} \partial_{\xi}^{\beta} \partial_{\eta}^{\gamma} \sigma(x, \xi, \eta)|}{(1 + |\xi| + |\eta|)^{m + \delta |\alpha| - \rho (|\beta| + |\gamma|)}} < \infty.
\end{align*}
Given $r_1, r_2, s \in (0, \infty)$, set
\begin{align*}
\Gamma(r_1, r_2, s)
:= \frac{1}{\min\{r_1, 2\}} + \frac{1}{\min\{r_2, 2\}} - \frac{1}{\max\{s, 2\}}.
\end{align*}

We present sparse bounds for bilinear pseudo-differential operators as follows.

\begin{theorem}\label{thm:pse}
Let $0 < \rho \le 1$, $0 \le \delta \le \rho$, and $\delta < 1$. Let $1 < r_1, r_2, s < \infty$ with $\frac1s < \frac{1}{r_1} + \frac{1}{r_2}$. Then $m \le - d(1-\rho) \Gamma(r_1, r_2, s)$ if and only if there exists some finite $N$ sufficiently large such that
\begin{align*}
|\langle T_{\sigma}(f_1, f_2), f_3 \rangle|
\lesssim \|\sigma\|_{BS_{\rho, \delta}^{m, N}} \, \Lambda_{r_1, r_2, s'}^*(f_1, f_2, f_3)
\end{align*}
for all $\sigma \in BS_{\rho, \delta}^{m, N}$ and all $f_1, f_2, f_3 \in L_c^{\infty}(\Rd)$. 
\end{theorem}

Weighted norm inequalities immediately follow from Theorems \ref{thm:pse} and \ref{thm:form} below. Theorem \ref{thm:pse} characterizes the exact symbol orders for sparse bounds for bilinear pseudo-differential operators, including the critical order. Cao--Xue--Yabuta \cite{CXY} obtained pointwise sparse bounds in the $m$-linear setting under stronger decay assumptions, while Beltr\'{a}n--Cladek \cite{BC} established sparse bounds in the linear case, sharp up to endpoints. Although an averaging exponent being greater than one leads to  larger symbol orders, our result determines the precise dependence on three exponents and provides a matching necessity statement. In particular, the $(2, 2, 2)$-sparse bound holds at $m = -d(1-\rho)/2$. The distinction between pointwise and integral-type domination is essential when comparing these results. Moreover, in the unweighted theory, Theorem \ref{thm:pse} improves the sufficient order $m < -d(1-\rho) \Gamma^*(p_1, p_2, p)$ of Rodr\'{i}guez-L\'{o}pez--Staubach \cite[Theorem 5.19]{RS} for $p_1 \ne 1$ or $p_2 \ne 1$ because 
\begin{align*}
\Gamma^*(p_1, p_2, p) - \Gamma(p_1, p_2, p) 
= \frac12 \min\bigg\{\bigg|\frac{1}{p_1} - \frac{1}{p_2} \bigg|, \,
\max \Big\{\frac{1}{p_1} + \frac{1}{p_2} - 1, 0 \Big\} \bigg\},
\end{align*}
where $\Gamma^*(p_1, p_2, p) := \frac{1}{\min\{p_1, p_2,  2\}} + \frac{1}{2 \min\{p, 1\}} - \frac{1}{\max\{p, 2\}}$ and $\frac1p = \frac{1}{p_1} + \frac{1}{p_2}$. The number $m_{\mathrm{crit}} = -d(1-\rho) \Gamma(p_1, p_2, p)$ is in fact the critical order by Miyachi–Tomita \cite{MT20}. The strict conditions $r_i < p_i$ and $p < s$ in the sparse-to-weighted estimates (see Theorem \ref{thm:form}) prevent the optimal unweighted order from being attained by an admissible choice of sparse exponents. Consequently, the sparse theorem alone does not recover all critical cases of the Miyachi--Tomita theorem.

The main difficulty in the sufficiency proof is summation over frequency scales at the critical order. When $m < - d(1-\rho) \Gamma(r_1, r_2, s)$, the frequency decay makes the localized estimates summable, but this decay disappears when $m = - d(1-\rho) \Gamma(r_1, r_2, s)$. The argument starts from the frequency localized estimates of Kato--Miyachi--Tomita \cite{KMT} and combines them with spatial localization at the scale dictated by the symbol class. To reach the critical order, we first remove a remainder of lower symbol order by a cutoff in the $x$-Fourier variable and the fact $\delta < 1$ (see \eqref{eq:srem}). The remaining frequency interactions admit annular localization in two of the three functions (see \eqref{eq:PEL}). Inserting projections before spatial localization allows the frequency scales to be summed using local square function estimates and a stopping time argument (see Lemma \ref{lem:pp}). This preserves the critical order without sacrificing the prescribed averaging exponents and avoids reliance on a symbolic calculus for partial adjoints. Beyond that, to show the necessity, we adapt the frequency-lattice constructions and randomization techniques used by Miyachi--Tomita \cite{MT13} and Kato--Shida \cite{KS} to the sparse setting. The additional ingredient is translation averaging: independent translations control the sparse supremum, while a compensating phase in the multiplier preserves an exactly computable trilinear pairing.

\subsubsection{\bf{Bilinear oscillatory Fourier multipliers}}
Let $\phi \in \mathscr{C}^{\infty}(\R^{2d})$ be a radial function satisfying $\phi(\vec{\xi}) = 0$ for $|\vec{\xi}| \le 1$, and $\phi(\vec{\xi}) = 1$ for $|\vec{\xi}| \ge 2$. Given $a > 0$, $a \ne 1$, and $b > 0$, denote
\begin{equation*}
\m_{a, b}(\vec{\xi})
:= \phi(\vec{\xi}) |\vec{\xi}|^{-b} e^{i |\vec{\xi}|^a}, \qquad \vec{\xi} = (\xi_1, \xi_2) \in \R^{2d}.
\end{equation*}
For a sequence $\c = (c_k)_{k \in \Z}$ with $\sup_k |c_k| \le 1$, the {\tt bilinear oscillatory Fourier multiplier} is defined by
\begin{align*}
T_{\m^{\c}_{a, b}}(\vec{f})(x)
&:= \int_{\R^{2d}} \m^{\c}_{a, b}(\vec{\xi}) 
\widehat{f}_1(\xi_1) \, \widehat{f}_2(\xi_2)
e^{2 \pi i x \cdot (\xi_1 + \xi_2)} \, d\vec{\xi},
\end{align*}
where
\begin{align*}
\m^{\c}_{a, b}(\vec{\xi})
:= \sum_{k \in \Z} c_k \, \m_{a, b} (2^k \vec{\xi}),
\end{align*}

We obtain the sparse bound for bilinear oscillatory Fourier multipliers below.
\begin{theorem}\label{thm:osc}
Let $a>0$, $a \ne 1$, $1 < r < 2$, and $b > a \, d/r$. Then
\begin{equation*}
\sup_{\substack{\c = (c_k)_{k \in \Z}\\ \sup_k |c_k| \le 1}} 
|\langle T_{\m^{\c}_{a, b}}(f_1, f_2), f_3 \rangle|
\lesssim \Lambda_{r, r, r}^*(f_1, f_2, f_3),
\end{equation*}
for all $f_1, f_2, f_3 \in L_c^{\infty}(\Rd)$.
\end{theorem}

Theorem \ref{thm:osc} provides a bilinear counterpart to the sparse bounds for linear oscillatory multipliers in \cite[proposition 7.10]{BRS}. The principal difficulty is to preserve the improved decay in arbitrary dyadic partial sums while simultaneously establishing the localized improving and regularity estimates required by the sparse framework. This is particularly delicate for the auxiliary bounds with quasi-Banach targets. The proof combines isotropic Sobolev control with integrability exponent four and quantitative Calder\'{o}n--Zygmund interpolation. Spatial localization then converts improving estimates into the required auxiliary estimates without imposing an additional decay condition, while compactly supported convolution factors provide the translation estimates. Finally, summation is performed at the level of scalar trilinear forms, which avoids the unavailable triangle inequality in the quasi-Banach target. These ingredients allow the improved decay to survive the passage from localized pieces to the full multiplier family.

\subsubsection{\bf{Bilinear spherical maximal operators}}
Let $d \ge 2$ and let $\sigma$ be the normalized surface measure on $\bS^{2d-1} \subset \R^{2d}$. Define the full and lacunary maximal operators by
\begin{align*}
\mathcal{M}_{\full}(f_1, f_2)(x)  := \sup_{t > 0} |\A_t (f_1, f_2)(x)|
\quad \text{and} \quad
\mathcal{M}_{\lac}(f_1, f_2)(x)  := \sup_{j \in \Z} |\A_{2^j} (f_1, f_2)(x)|,
\end{align*}
where
\begin{align*}
\A_t (f_1, f_2)(x) 
:= \int_{\bS^{2d-1}} f_1(x-ty) f_2(x-tz) \, d\sigma(y, z).
\end{align*}
Denote
\begin{align*}
\mathcal{R}_{\full}(d) 
& := \bigg\{ \bigg(\frac{1}{p_1}, \frac{1}{p_2}, \frac1q \bigg):
1 < p_1, p_2 < \infty, \, 0 < q < \infty, \, \frac1q < \frac{1}{p_1} + \frac{1}{p_2} < m(d, q) \bigg\},
\end{align*}
and
\begin{align*}
\mathcal{R}_{\lac}(d) 
& := \inter \conv \big(\mathcal{R}_{\full}(d) \cup\{P_0, \, P_1\}\big), 
\end{align*}
where
\begin{equation*}
m(d, q) =
\begin{cases}
\displaystyle
\min \bigg\{1 + \frac{1}{q}, \frac32 \bigg\}, & d = 2,
\\[1.2em]
\displaystyle
\min \bigg\{1 + \frac{d}{q}, \frac{2d-1}{d}, \frac1q + \frac{2(d-1)}{d} \bigg\}, & d \ge 3,
\end{cases}
\end{equation*}
$P_0 = (1, 1, 1)$, and $P_1 = \big(\frac{(2d-1)^2}{2 (2d^2-d+1)}, \, \frac{(2d-1)^2}{2 (2d^2-d+1)}, \, \frac{2d-3}{2d^2-d+1} \big)$.

We establish sparse bounds for $\M_{\full}$ and $\M_{\lac}$ as follows.

\begin{theorem}\label{thm:sph}
Let $d \ge 2$ and $1 < r_1, r_2, s < \infty$. Then
\begin{align*}
|\langle \M_{\full}(f_1, f_2), f_3 \rangle|
& \lesssim \Lambda_{r_1, r_2, s'}^*(f_1, f_2, f_3), 
\qquad \big(\tfrac{1}{r_1}, \tfrac{1}{r_2}, \tfrac1s \big) \in \mathcal{R}_{\full}(d),
\\
|\langle \M_{\lac}(f_1, f_2), f_3 \rangle|
& \lesssim \Lambda_{r_1, r_2, s'}^*(f_1, f_2, f_3),
\qquad \big(\tfrac{1}{r_1}, \tfrac{1}{r_2}, \tfrac1s \big) \in \mathcal{R}_{\lac}(d),
\end{align*}
for all $f_1, f_2, f_3 \in L_c^{\infty}(\Rd)$.
\end{theorem}

The full region $\mathcal{R}_{\full}(d)$ is the standard exponent region for the bilinear spherical maximal problem. Although it appears in the sparse domination results of \cite{BFOPZ}, we only use \cite[Proposition~4]{BFOPZ} to supply the localized regularity estimates; the full maximal boundedness and localized maximal estimates are taken from \cite{JL}.
However, the slicing method in \cite{BFOPZ} introduces linear spherical averages over a continuum of radii, thus loses the gain from lacunarity. This means that method does not yield a larger sparse range for $\M_{\lac}$ than $\mathcal{R}_{\full}(d)$. Our enlarged region $\mathcal{R}_{\lac}(d)$ refines $\mathcal{R}_{\full}(d)$ by adding the points $P_0$ and $P_1$, for which the stronger $L^p$-improving estimates are available for the fixed-scale average $\A_1$ (see \eqref{AA11} and \cite[Theorem~1.5]{BCSS}), since lacunary dilations avoid the additional loss associated with the continuous supremum $\sup_{t \in [1, 2]} |\A_t|$. The proof of Theorem \ref{thm:sph} uses a finite-dimensional $\ell^{\infty}$-valued linearization to verify \eqref{H1}--\eqref{H5} uniformly over truncations.

We also consider the lacunary triangle averaging operator
\begin{align*}
\mathcal{T}_{\mathrm{lac}}(f_1, f_2)(x) 
:= \sup_{j \in \Z} |\mathcal{T}_{2^j} (f_1, f_2)(x)|,
\end{align*}
where
\begin{align*}
\mathcal{T}_t (f_1, f_2)(x) := \int_{\mathfrak{D}} f_1(x - ty) f_2(x - tz) \, d\mu(y, z),
\end{align*}
$\mathfrak{D} := \{(y, z)\in\R^{2d}: |y|=|z|=|y-z|=1\}$, and $\mu$ is the natural surface measure on $\mathfrak{D}$. Let $\mathcal{P}_d$ be the convex hull of the following points
\begin{align*}
(0, 0, 0), & \, (0, 1, 1), \, (1, 0, 1), \,
\bigg(0, \frac{d}{d+1}, \frac{1}{d+1} \bigg),\, 
\bigg(\frac{d}{d+1}, 0, \frac{1}{d+1} \bigg),
\\
& \bigg(\frac{d}{d+1}, \frac{d}{d+1}, 1 \bigg),\
\bigg(\frac{d}{d+1}, \frac{d}{d+1}, \frac{2d}{d+1} \bigg).
\end{align*}
Note that $\mathcal{P}_d \subset \{(x, y, z): x+y \ge z\}$. Set
\begin{align*}
\mathcal{R}_{\tri}(d) := \big\{(x, y, z) \in \inter \mathcal{P}_d: 0 < z < 1 \big\}.
\end{align*}

\begin{theorem}\label{thm:tri}
Let $d \ge 5$ and $\big(\frac{1}{r_1}, \frac{1}{r_2}, \frac1s \big) \in \mathcal{R}_{\tri}(d)$ with $1 < r_1, r_2, s < \infty$. Then
\begin{align*}
|\langle \mathcal{T}_{\lac}(f_1, f_2), f_3 \rangle|
\lesssim \Lambda_{r_1, r_2, s'}^*(f_1, f_2, f_3),
\end{align*}
for all $f_1, f_2, f_3 \in L_c^{\infty}(\Rd)$. 
\end{theorem}

Theorem \ref{thm:tri} establishes sparse domination for $T_{\lac}$ throughout $\mathcal{R}_{\tri}(d)$ with $d \ge 5$, which matches the dimensional threshold in \cite{BFO} and improves the dimensional restriction $d \ge 13$ in \cite{PS}, while removes the restriction $r_1, r_2 \le s$ in both works. Our key improvement is the explicit verification of auxiliary improving and regularity estimates at $(r_1, r_2^*, r_1)$ and $(r_1^*, r_2, r_2)$, which permit removal of the exponent ordering. Moreover, Lemma~\ref{lem:Tlac} presents these global Lebesgue bounds directly for $d \ge 5$ by combining the
single-scale smoothing estimate of \cite{BFO} with spatial localization and the finite-scale bootstrap of \cite{BF}.
Even in the unweighted result, Theorem \ref{thm:tri} covers \cite[Theorem~5]{BF} with $d \ge 7$, where the dimensional restriction $d \ge 7$ arises from the available $L^2 \times L^2 \to L^1$ frequency decay, see \cite[Proposition 29]{BF}.

\subsubsection{\bf{Bilinear Bochner--Riesz means}}

Let $d \ge 2$ and $\delta > 0$. The bilinear Bochner--Riesz means is defined by
\begin{align*}
B^{\delta}(f, g)(x) 
:= \int_{\R^{2d}} (1 - |\xi|^2 - |\eta|^2)_+^{\delta} \,
\widehat{f}(\xi) \widehat{g}(\eta) \, e^{2 \pi i x \cdot (\xi + \eta)} \, d\xi \, d\eta,
\end{align*}
where $a_+ := \max\{a, 0\}$. Given $1 \le q_1, q_2 < \frac{2d}{d+1}$, denote
\begin{align*}
\mathcal{R}(\vec{q}) 
:= \bigg\{\bigg(\frac{1}{p_1}, \frac{1}{p_2} \bigg): q_1 \le p_1 \le \infty, q_2 \le p_2 \le \infty, 
\frac{1}{q_1} + \frac{1}{q_2} - \frac{1}{p_1} - \frac{1}{p_2} \le \alpha(q_1, q_2)\bigg\},
\end{align*}
where $\alpha(q_1, q_2)$ is defined in \eqref{def:aqq}.

The bilinear Bochner--Riesz means has sparse bounds as follows.
\begin{theorem}\label{thm:BR}
Let $d \ge 2$, $1 \le q_1, q_2 < \frac{2d}{d+1}$, and $\delta > d \alpha(q_1, q_2) -1$. Then for all $(\frac{1}{r_1}, \frac{1}{r_2}), (\frac{1}{s_1}, \frac{1}{s_1}) \in \mathcal{R}(\vec{q})$ satisfying $0 < \max\{\frac1r - \frac1s, \frac{1}{r_1}, \frac{1}{r_2}\} < \delta + 1 - d \alpha(q_1, q_2)$,
\begin{align*}
|\langle B^{\delta}(f_1, f_2), f_3 \rangle|
\lesssim \Lambda_{r_1, r_2, s'}^*(f_1, f_2, f_3),
\end{align*}
for all $f_1, f_2, f_3 \in L_c^{\infty}(\Rd)$, where $\frac1r = \frac{1}{r_1} + \frac{1}{r_2}$ and $\frac1s = \frac{1}{s_1} + \frac{1}{s_2}$.
\end{theorem}

Theorem \ref{thm:BR} provides a direct route to quantitative weighted inequalities, which can be viewed as a sparse form strengthening of $L^p$ bounds in \cite{BGSY} on an explicitly specified subregion although it does not recover the full theorems there. The sparse domination for bilinear Bochner--Riesz means obtained in this paper are subject to substantial restrictions. We will investigate this problem more thoroughly in subsequent papers, with the aim of overcoming these limitations.

\subsection{Structure of this paper}
The rest of the paper is organized as follows. Section \ref{sec:pre} contains some preliminaries including notation, definitions, and some basic estimates which will be used later. Section \ref{sec:red} is devoted to some reductions of Theorem \ref{thm:main} and single-scale regularity estimates in the dyadic setting. After that, the main part of the proof of Theorem \ref{thm:main} is presented in Section \ref{sec:main}. Finally, in Section \ref{sec:app}, we sequentially give the proof of Theorems \ref{thm:rough}--\ref{thm:BR}.

\section{Preliminaries}\label{sec:pre}

\subsection{Notation}\label{sec:notation}
Let us introduce some notation that will be used throughout this article.
\begin{itemize}
\item For convenience, we use the $\ell^{\infty}$ metric on $\Rd$.

\item Let $\N :=\{0, 1, 2, \ldots\}$ be the set of natural numbers. 

\item Let $\lfloor a \rfloor$ denote the largest integer not greater than $a$.

\item Let $\vec{v} = (v_1, \ldots, v_m)$ denote the vector of components $v_1, \ldots, v_m$.

\item For each $p \in (0, \infty]$ and weight $w$, let $\|f\|_{L^p(w^p)} = \|fw\|_{L^p}$.

\item The translation operator $\tau_v$ is defined by $\tau_v f (x) := f(x+v)$. 

\item Let $\mathscr{S}(\Rd)$ denote the class of Schwartz functions on $\Rd$. 

\item Let $\mathscr{C}^k(\Rd)$ be the space of $k$-times continuously differentiable functions on $\Rd$. 

\item Let $\mathscr{C}_c^{\infty}(\Rd)$ be the collection of smooth functions with compact support on $\Rd$. 

\item By a cube in $\Rd$ of sidelength $\ell$, we mean a half-open cube $Q = \prod_{i=1}^d [a_i, a_i + \ell)$, $\ell > 0$.

\item Let $\Q$ denote the collection of all cubes in $\Rd$.

\item Given a cube $Q \subset \Rd$, we denote its center by $c_Q$ and its sidelength by $\ell(Q)$. For any $\lambda>0$, we denote by $\lambda Q$ the cube with the center $c_Q$ and sidelength $\lambda \ell(Q)$. The measure of $Q$ is simply denoted by $|Q|$ no matter what dimension we are in.

\item For a bounded set $E \subset \Rd$ with $|E| > 0$, we write $[f]_E := \fint_E f \, dx = \frac{1}{|E|} \int_E f \, dx$, $\langle f \rangle_E := [|f|]_E$, and $\langle f \rangle_{E, p} := [|f|^p]_E^{1/p}$ for any $1 \le p < \infty$.

\item We shall use $a\lesssim b$ and $a \simeq b$ to mean, respectively, that $a\leq C b$ and $0 < c \leq a/b \leq C$, where the constants $c$ and $C$ are harmless positive constants, not necessarily the same at each occurrence, which depend only on dimension and the constants appearing in the hypotheses of theorems. Moreover, if $c$ and $C$ depend on some given parameter $\eta$, which is somehow relevant, we write $a \lesssim_\eta b$ and $a \simeq_\eta b$.
\end{itemize}

\subsection{Muckenhoupt weights}
A measurable function $w$ on $\Rd$ is called a weight if $0 < w(x) < \infty$ for a.e.~$x \in \Rd$. We are going to give the definition of the multilinear Muckenhoupt classes $A_{\vec{p}, (\vec{r}, s)}$ introduced in \cite{LMO}. Let $m \ge 1$. Given $\vec{p} = (p_1, \ldots, p_m) \in [1, \infty]^m$, $\vec{r} = (r_1, \ldots, r_m) \in [1, \infty)^m$, and $s \in (1, \infty]$, we say that $(\vec{r}, s) \preceq \vec{p}$ whenever
\begin{align*}
r_j \le p_j, \, j=1, \ldots, m, \text{ and } p \le s, \text{ where } \frac1p := \sum_{j=1}^m \frac{1}{p_j}.
\end{align*}
Analogously, we say that $(\vec{r}, s) \prec \vec{p}$ if $r_j < p_j$ for each $j=1, \ldots, m$, and $p<s$.

\begin{definition}\label{def:Apr}
Let $\vec{p} = (p_1, \ldots, p_m) \in [1, \infty]^m$, $\vec{r} = (r_1, \ldots, r_m) \in [1, \infty)^m$ and $s \in (1, \infty]$ such that $(\vec{r}, s) \preceq \vec{p}$. Suppose that $\vec{w} = (w_1, \ldots, w_m)$ and each $w_i$ is a weight on $\Rd$. We say that $\vec{w} \in A_{\vec{p}, (\vec{r}, s)}$ if 
\begin{align*}
[\vec{w}]_{A_{\vec{p}, (\vec{r}, s)}} 
:= \sup_{Q \in \Q} \Big\langle w^{\frac{sp}{s - p}} \Big\rangle_Q^{\frac1p - \frac1s} 
\prod_{j=1}^m \Big \langle w_j^{- \frac{p_j r_j}{p_j - r_j}} \Big \rangle_Q^{\frac{1}{r_j} - \frac{1}{p_j}} < \infty,
\end{align*} 
where the supremum is taken over all cubes $Q \subset \Rd$, $w = \prod_{j=1}^m w_j$, and $\frac1p = \sum_{j=1}^m \frac{1}{p_j}$. When $p=s$, the term corresponding to $w$ needs to be replaced by $\esssup_Q w$ and, analogously, when $p_j = r_j$, the term corresponding to $w_i$ should be $\esssup_Q w_i^{-1}$. When $s=\infty$, the term corresponding to $w$ needs to be replaced by $\langle w^p \rangle_Q^{1/p}$.
\end{definition}

Together with multilinear extrapolation, sparse domination immediately implies quantitative weighted estimates (\cite[Corollary 2.15]{LMO} and \cite[Corollary 4.2]{Nie}). 

\begin{theorem}\label{thm:form}
Let $r_1, \ldots, r_m \in [1, \infty)$ and $s \in (1, \infty]$. If an $m$-linear or a positive valued $m$-sublinear operator $T$ satisfies 
\begin{align*}
|\langle T(\vec{f}), \, f_{m+1} \rangle|
\lesssim \Lambda_{\vec{r}, s'}^*(\vec{f}, f_{m+1}).
\end{align*}
for all $f_1, \ldots, f_m \in L_c^{\infty}(\Rd)$, then 
\begin{align*}
\|T\|_{L^{p_1}(w_1^{p_1}) \times \cdots \times L^{p_m}(w_m^{p_m}) \to L^p(w^p)}
\lesssim [\vec{w}]_{A_{\vec{p}, (\vec{r}, s)}}^{\max\Big\{\frac{\frac{1}{r_1}}{\frac{1}{r_1} - \frac{1}{p_1}}, \ldots, \frac{\frac{1}{r_m}}{\frac{1}{r_m} - \frac{1}{p_m}}, \frac{\frac{1}{s'}}{\frac1p - \frac1s} \Big\}},
\end{align*}
for all $\vec{p} \in (0, \infty]$ with $(\vec{r}, s) \prec \vec{p}$, and for all $\vec{w} \in A_{\vec{p}, (\vec{r}, s)}$.
\end{theorem}

\subsection{Some auxiliary results}
We present some estimates that will be used in Application parts.

\begin{lemma}\label{lem:TKfg}
The $m$-linear convolution operator
\begin{align*}
T_K(\vec{f})(x)
:= \int_{\R^{md}} K(x - \vec{y}) \prod_{i=1}^m f_i(y_i) \, d\vec{y}
\end{align*}
satisfies 
\begin{list}{\rm (\theenumi)}{\usecounter{enumi}\leftmargin=1.3cm \labelwidth=1cm \itemsep=0.1cm \topsep=0.2cm \renewcommand{\theenumi}{\arabic{enumi}}}

\item\label{TKfg-1} For all $\frac1p = \sum_{i=1}^m \frac{1}{p_i}$ with $p_1, \ldots, p_m \in [1, \infty]$,
\begin{align*}
\|T_K\|_{L^{p_1} \times \cdots \times L^{p_m} \to L^p} 
\lesssim \|K\|_{L^1}.
\end{align*}

\item\label{TKfg-2} For all $N > 2d$ and $\frac{1}{p_0} \le \sum_{i=1}^m \frac{1}{p_i}$ with $p_1, \ldots, p_m \in [1, \infty]$, and , 
\begin{align*}
\|T_K\|_{L^{p_1} \times \cdots \times L^{p_2} \to L^{p_0}} 
\lesssim \|K\|_{\mathscr{S}_N},
\end{align*}
where $\|K\|_{\mathscr{S}_N} := \sup_{\vec{y} \in \R^{md}} |K(\vec{y})| (1 + |\vec{y}|)^N$.  

\item\label{TKfg-3} If $\sigma = \widehat{K}$, then 
\begin{align*}
\|T_K\|_{L^2 \times \cdots \times L^2 \to L^2} 
\lesssim \sup_{\zeta \in \Rd} \bigg[\int_{\R^{(m-1)d}} 
\bigg|\sigma \bigg(\xi', \zeta - \sum_{i=1}^{m-1} \xi_i \bigg) \bigg|^2 \, d\xi' \bigg]^{\frac12}.
\end{align*}

\item\label{TKfg-33} For all $p_0, p_1, \ldots, p_m \in [1, \infty)$,
\begin{align*}
\|T_{\psi_t*K}\|_{L^{p_1} \times \cdots \times L^{p_2} \to L^{p_0}} 
\le \|\psi\|_{L^1} \|T_K\|_{L^{p_1} \times \cdots \times L^{p_2} \to L^{p_0}},
\end{align*}
where $\psi \in L^1(\R^{md})$ and $\psi_t(x) = t^{-md} \psi(x/t)$.

\item\label{TKfg-4} 
For all $\frac{1}{p_0} \le \sum_{i=1}^m \frac{1}{p_i}$ with $p_1, \ldots, p_m \in [1, \infty)$,
\begin{align*}
\sup_{t > 0} \|T_{\phi (t \cdot ) K}\|_{L^{p_1} \times \cdots \times L^{p_m} \to L^{p_0}}
\lesssim C_{\phi} \|T_K\|_{L^{p_1} \times \cdots \times L^{p_m} \to L^{p_0}},
\end{align*}
where $\phi \in \mathscr{C}_c^{\infty}(\R^{md})$, and $C_{\phi} = \|\widehat{\phi}\|_{L^1}$ when $p_0 \ge 1$. 
\end{list}
\end{lemma}

\begin{proof}
Part \eqref{TKfg-1} follows from Minkowski's inequality in the case $p \ge 1$, and from Jensen's inequality and Tonelli's theorem in the case $0<p<1$. To prove part \eqref{TKfg-2}, note that
\begin{align*}
|T_K(\vec{f})(x)| 
\le \|K\|_{\mathscr{S}_N} \prod_{i=1}^m (\varphi*|f_i|)(x),
\end{align*}
where $\varphi(x) := (1 + |x|)^{-N/2}$. Choose $r_i \in [p_i, \infty]$ so that $\frac{1}{p_0} = \sum_{i=1}^m \frac{1}{r_i}$. Then take $q_i \in [1, \infty]$ such that $1 + \frac{1}{r_i} = \frac{1}{p_i} + \frac{1}{q_i}$, $i=1, \ldots, m$. Since $N > 2d \max\{\frac{1}{q_i}: i=1, \ldots, m\}$ so that $\varphi \in \cap_{i=1}^m L^{q_i}$, Young's inequality gives
\begin{align*}
\|T_K(\vec{f})\|_{L^{p_0}} 
& \le \|K\|_{\mathscr{S}_N} \prod_{i=1}^m \|\varphi*|f_i|\|_{L^{r_i}}
\le \|K\|_{\mathscr{S}_N} \prod_{i=1}^m \|f_i\|_{L^{p_i}} \|\varphi\|_{L^{q_i}}.
\end{align*}

Let $\xi' = (\xi_1, \ldots, \xi_{m-1})$ and $\xi_m = \zeta -\sum_{i=1}^{m-1} \xi_i$. By change of variables and the Cauchy--Schwarz inequality, 
\begin{align*}
|\widehat{T_K(\vec{f})}(\zeta)|
& = \bigg|\int_{\R^{(m-1)d}} \sigma(\vec{\xi}) \prod_{i=1}^m \widehat{f}_i(\xi_i) \, d\xi_1 \cdots d\xi_{m-1} \bigg|
\\
& \le \bigg[\int_{\R^{(m-1)d}} |\sigma(\vec{\xi})|^2 \, d\xi' \bigg]^{\frac12}
\bigg[\int_{\R^{(m-1)d}} \prod_{i=1}^m |\widehat{f}_i(\xi_i)|^2 \, d\xi'  \bigg]^{\frac12},
\end{align*}
which together with Plancherel's theorem yields the desired estimate. This shows part \eqref{TKfg-3}.

Part \eqref{TKfg-33} is a consequence of Minkowski's inequality and a change of variables:
\begin{align*}
T_{\psi_t*K}(\vec{f})
= \int_{\R^{md}} \psi_t(v) T_K(\tau_{-v_1}f_1, \ldots, \tau_{-v_m}f_m) \, dv.
\end{align*}

To prove part \eqref{TKfg-4}, let $t > 0$ and $\phi \in \mathscr{C}_c^{\infty}(\R^{md})$. Then $\supp \phi \subset [-\lambda, \lambda]^{md}$ for some integer $ \lambda \ge 1$. Partition $\Rd$ into cubes of side length $t^{-1}$, and denote the resulting family by $\mathcal{Q}_t$. Given $Q \in \mathcal{Q}_t$, choose $\eta_Q \in \mathscr{C}_c^{\infty}(\Rd)$ such that $0 \le \eta_Q \le 1$, $\eta_Q = 1$ on $(2\lambda +1)Q$, and $\supp \eta_Q \subset (2\lambda +3) =: Q^*$. Let $\lambda_0 = 4(\lambda + 2)$. Then the Fourier series expansion gives 
\begin{align}\label{eq:QTK-1}
\phi_0(\vec{y}) 
:= \sum_{\vec{\nu} \in \Z^{md}} \phi(\vec{y} + \lambda_0 \vec{\nu})
= \sum_{\vec{\xi} \in \Z^{md}} c_{\vec{\xi}} \, e^{2\pi i \vec{y} \cdot \vec{\xi}/\lambda_0},
\end{align}
where the coefficients satisfy $|c_{\vec{\xi}}| = |\lambda_0^{-md} \widehat{\phi}(\lambda_0^{-1} \vec{\xi})| \lesssim_N (1 + |\vec{\xi}|)^{-N}$. Note that $\phi_0 = \phi$ on $[-\lambda - 2, \lambda +2]^{md}$,  $x \in Q$ and $x-y_i \in (2\lambda +3)Q$ imply $t|y_i| \le \lambda +2$, which together with \eqref{eq:QTK-1} gives
\begin{align*}
& \mathbf{1}_Q T_{\phi(t \cdot) K}(\vec{f})
= \mathbf{1}_Q T_{\phi(t \cdot) K}(\vec{f} \eta_Q)
= \mathbf{1}_Q T_{\phi_0(t \cdot) K}(\vec{f} \eta_Q)
\\
& = \mathbf{1}_Q \sum_{\vec{\xi} \in \Z^{md}} c_{\vec{\xi}} 
\int_{\R^{md}} K(\vec{y}) \, e^{2\pi i \vec{y} \cdot \vec{\xi} \, t/\lambda_0} 
\prod_{i=1}^m (f_i \eta_Q)(x - y_i) \, d\vec{y}
\\
& = \mathbf{1}_Q \sum_{\vec{\xi} \in \Z^{md}} c_{\vec{\xi}} \, e^{2\pi i x \sum_{i=1}^m \xi_i \, t/\lambda_0}
T_K(M_{-\xi_1 t /\lambda_0} (f_1 \eta_Q), \ldots, M_{-\xi_m t /\lambda_0} (f_m \eta_Q))(x).
\end{align*}
where $M_{\xi} f(x) := e^{2\pi i x \cdot \xi}f(x)$ for $x, \xi \in \Rd$. Let $q = \min\{1, p_0\}$ and $\frac1p = \sum_{i=1}^m \frac{1}{p_i}$. Then 
\begin{align*}
\|T_{\phi(t \cdot) K}(\vec{f})\|_{L^{p_0}}^{p_0}
& = \sum_{Q \in \mathcal{Q}_t} \|\mathbf{1}_Q T_{\phi(t \cdot) K}(\vec{f})\|_{L^{p_0}}^{p_0}
= \sum_{Q \in \mathcal{Q}_t} \big\| |\mathbf{1}_Q T_{\phi(t \cdot) K}(\vec{f})|^{q} \big\|_{L^{p_0/q}}^{p_0/q}
\\
& \le \sum_{Q \in \mathcal{Q}_t} \bigg\| \sum_{\vec{\xi} \in \Z^{md}} |c_{\vec{\xi}}|^q
|T_K(M_{-\xi_1 t /\lambda_0} (f_1 \eta_Q), \ldots, M_{-\xi_m t /\lambda_0} (f_m \eta_Q))|^q \bigg\|_{L^{p_0/q}}^{p_0/q}
\\
& \le \sum_{Q \in \mathcal{Q}_t} \bigg[ \sum_{\vec{\xi} \in \Z^{md}} |c_{\vec{\xi}}|^q
\|T_K(M_{-\xi_1 t /\lambda_0} (f_1 \eta_Q), \ldots, M_{-\xi_m t /\lambda_0} (f_m \eta_Q)) \|_{L^{p_0}}^{q} \bigg]^{p_0/q}
\\
& \lesssim \|T_K\|_{L^{p_1} \times \cdots \times L^{p_m} \to L^{p_0}}^{p_0} 
\bigg[ \sum_{\vec{\xi} \in \Z^{md}} |c_{\vec{\xi}}|^q  \bigg]^{p_0/q}
\sum_{Q \in \mathcal{Q}_t} \prod_{i=1}^m \|f_i \eta_Q\|_{L^{p_i}}^{p_0}
\\
& \lesssim \|T_K\|_{L^{p_1} \times \cdots \times L^{p_m} \to L^{p_0}}^{p_0} 
\bigg[\sum_{Q \in \mathcal{Q}_t} \prod_{i=1}^m \|f_i \eta_Q\|_{L^{p_i}}^p \bigg]^{\frac{p_0}{p}}
\\
& \le \|T_K\|_{L^{p_1} \times \cdots \times L^{p_m} \to L^{p_0}}^{p_0} 
\bigg[\prod_{i=1}^m \bigg(\sum_{Q \in \mathcal{Q}_t} \|f_i \eta_Q\|_{L^{p_i}}^{p_i} \bigg)^{\frac{1}{p_i}} \bigg]^{p_0}
\\
& \lesssim \|T_K\|_{L^{p_1} \times \cdots \times L^{p_m} \to L^{p_0}}^{p_0}
\bigg[\prod_{i=1}^m \|f_i\|_{L^{p_i}} \bigg]^{p_0},
\end{align*}
where we have used the bounded overlap of the family $\{Q^* = (2\lambda + 3)Q: Q \in \mathcal{Q}_t\}$.
\end{proof}

Next, we present an annular improving estimate for rough singular integral operators.

\begin{lemma}\label{lem:imp}
Let $q \in [1, \infty]$ and $1<r_1, \ldots, r_m, s < \infty$ with $\frac1s < \frac1r = \sum_{i=1}^m \frac{1}{r_i}$. If $(\frac{1}{r_1}, \ldots, \frac{1}{r_m}, \frac{1}{s'}) \in \overline{\H_{m+1}^1(q)}$, then
\begin{align*}
\|T_{K_{\Omega}^{\lambda}}\|_{L^{r_1} \times \cdots \times L^{r_m} \to L^s}
\lesssim \|\Omega\|_{L^q} \lambda^{-(\frac1r - \frac1s)}, 
\end{align*}
for all $\lambda > 0$ and $\Omega \in L^q(\bS^{m-1})$, where $K_{\Omega}^{\lambda}(\vec{y}) := |\vec{y}|^{-m} \Omega(\vec{y}/|\vec{y}|) \mathbf{1}_{\{\lambda \le |\vec{y}| \le 2 \lambda\}}$.
\end{lemma}

\begin{proof}
Let $r_{m+1} = s'$. It suffices to show
\begin{align*}
|\langle T_{K_{\Omega}^{\lambda}}(\vec{f}), f_{m+1} \rangle|
\lesssim \|\Omega\|_{L^q} \lambda^{1-\sum_{i=1}^{m+1} 
\frac{1}{r_i}} \prod_{i=1}^{m+1} \|f_i\|_{L^{r_i}},
\end{align*}
for all $\Omega \in L^q(\bS^{m-1})$. By partition $\R$ into half-open intervals of length $\lambda$, we may assume that $\supp f_i \subset 5J$ for some interval $J$ with $|J|=\lambda$, $i=1, \ldots, m+1$.

Define
\begin{align*}
\frac{1}{p_j} := q \max \bigg\{\frac{1}{r_j} - \frac{1}{q'}, 0 \bigg\}
\quad \text{and} \quad
\frac{1}{q_j} := q' \min \bigg\{\frac{1}{r_j}, \frac1{q'} \bigg\}, \quad j=1, \ldots, m+1,
\end{align*}
which satisfies
\begin{align*}
\frac{1}{r_j} = \frac{1/q}{p_j} + \frac{1/q'}{q_j}, 
\quad 1/\vec{p} \in \overline{\H_{m+1}^1(1)}, \quad \text{and} \quad
1/\vec{q} \in \overline{\H_{m+1}^1(\infty)}.
\end{align*}
We only show $1/\vec{p} \in \overline{\H_{m+1}^1(1)}$ because others are trivial. Indeed, $0 \le 1/p_j \le q(1-1/q')=1$.
Given a nonempty subset $I \subsetneq \{1, \ldots, m+1\}$, set $I' = \{j \in I : r_j < q' \}$. If $I' = \varnothing$, then $\sum_{j \in I} \frac{1}{p_j} = 0$. Otherwise, $I'$ is also a nonempty proper subset of $\{1, \ldots, m+1\}$, which together with the definition of $\overline{\H_{m+1}^1(q)}$ implies
\begin{align*}
\sum_{j \in I} \frac{1}{p_j}
= q \bigg[ \sum_{j \in I'} \Big(\frac{1}{r_j} - \frac{1}{q'} \Big) \bigg]
\le q \bigg[|I'| - \frac{|I'|-1}{q} - \frac{|I'|}{q'} \bigg]
= 1.
\end{align*}
This gives $1/ \vec{p} \in \overline{\H_{m+1}^1(1)}$.

Note that $\frac1q = \frac{1/q}{1} + \frac{1/q'}{\infty}$. Regarding $\Omega$ as an additional input and interpolating the $(m+2)$-linear form on the fixed support sets, we are deduced to proving
\begin{align}
\label{eq:TKA-1}
|\langle T_{K_{\Omega}^{\lambda}}(\vec{f}), f_{m+1} \rangle|
& \lesssim \|\Omega\|_{L^1} \lambda^{1-\sum_{i=1}^{m+1} 
\frac{1}{p_i}} \prod_{i=1}^{m+1} \|f_i\|_{L^{p_i}}, 
\quad \forall 1/\vec{p} \in \overline{\H_{m+1}^1(1)},
\\
\label{eq:TKA-2}
|\langle T_{K_{\Omega}^{\lambda}}(\vec{f}), f_{m+1} \rangle|
& \lesssim \|\Omega\|_{L^{\infty}} \lambda^{1-\sum_{i=1}^{m+1} \frac{1}{q_i}} 
\prod_{i=1}^{m+1} \|f_i\|_{L^{q_i}},
\quad \forall 1/\vec{q} \in \overline{\H_{m+1}^1(\infty)}.
\end{align}
Since $\overline{\H_{m+1}^1(\infty)} = [0, 1]^{m+1}$ and $\supp f_i \subset 5J$ with $|J| = \lambda$, the inequality \eqref{eq:TKA-2} holds trivially:
\begin{align*}
|\langle T_{K_{\Omega}^{\lambda}}(\vec{f}), f_{m+1} \rangle|
\le \lambda^{-m} \|\Omega\|_{L^{\infty}} \prod_{i=1}^{m+1} \|f_i\|_{L^1}
\lesssim \lambda^{1-\sum_{i=1}^{m+1} \frac{1}{q_i}} \|\Omega\|_{L^{\infty}} 
\prod_{i=1}^{m+1} \|f_i\|_{L^{q_i}}.
\end{align*}

It remains to justify \eqref{eq:TKA-1}. Observe that
\begin{align*}
\overline{\H_{m+1}^1(1)} 
= \conv \{\vec{0}, \vec{e}_1, \ldots, \vec{e}_{m+1}, \vec{v}\},
\end{align*}
where $\{\vec{e}_1, \ldots, \vec{e}_{m+1}\}$ are the standard basis vectors of $\R^{m+1}$, and $\vec{v} = (1/m, \ldots, 1/m)$. Indeed, for each $\vec{u} \in \overline{\H_{m+1}^1(1)}$,
\begin{align*}
\vec{u} = m u_0 \, \vec{v} + \sum_{j=1}^{m+1} (u_j - u_0) \vec{e}_j 
+ \bigg(1 + u_0 - \sum_{j=1}^{m+1} u_j \bigg) \vec{0},
\end{align*}
where $u_0 = \min_{1 \le j \le m+1} u_j$. Hence, by multilinear interpolation, it is enough to show
\begin{align}
\label{eq:zero-tar}
|\langle T_{K_{\Omega}^{\lambda}}(\vec{f}), f_{m+1} \rangle|
& \lesssim \|\Omega\|_{L^1} \lambda \prod_{j=1}^{m+1} \|f_j\|_{L^{\infty}},
\\
\label{eq:ej-tar}
|\langle T_{K_{\Omega}^{\lambda}}(\vec{f}), f_{m+1} \rangle|
&\lesssim \|\Omega\|_{L^1} \|f_j\|_{L^1} \prod_{i \ne j} \|f_i\|_{L^{\infty}},
\\
\label{eq:v-tar}
|\langle T_{K_{\Omega}^{\lambda}}(\vec{f}), f_{m+1} \rangle|
&\lesssim \|\Omega\|_{L^1} \lambda^{-\frac1m} \prod_{j=1}^{m+1} \|f_j\|_{L^m},
\end{align}
for each $j=1, \ldots, m+1$.

A change of variables gives
\begin{align*}
|\langle T_{K_{\Omega}^{\lambda}}(\vec{f}), f_{m+1} \rangle|
\le \int_{\mathbb{S}^{m-1}} |\Omega(\theta)|
\int_{\lambda}^{2 \lambda} \int_{\R} |f_{m+1}(x)| 
\prod_{i=1}^m |f_i(x - t \theta_i)| \, dx \, \frac{dt}{t} \, d\sigma(\theta),
\end{align*}
which implies \eqref{eq:zero-tar} and \eqref{eq:ej-tar} immediately. Given $\theta \in \mathbb{S}^{m-1}$, there exists some $i_0 \in \{1, \ldots, m\}$ such that $|\theta_{i_0}| \ge m^{-1/2}$. Then by H\"older's inequality and a change of variables $y = x-t \theta_{i_0}$ and $z=x$, we arrive at
\begin{align*}
& \int_{\lambda}^{2 \lambda} \int_{\R} |f_{m+1}(x)| 
\prod_{i=1}^m |f_i(x - t \theta_i)| \, dx \, dt
\\
& \le \bigg[\iint_{\R^2} |f_{i_0}(x - t \theta_{i_0})|^m |f_{m+1}(x)|^m \, dx \, dt \bigg]^{\frac1m} 
\\
& \qquad \times \prod_{i \not\ne i_0} \bigg[\int_{\lambda}^{2 \lambda} \int_{\R} 
|f_i(x - t \theta_i)|^m \, dx \, dt \bigg]^{\frac1m} 
\\
& = |\theta_{i_0}|^{-1} \bigg[\iint_{\R^2} |f_{i_0}(y)|^m |f_{m+1}(z)|^m \, dy \, dz \bigg]^{\frac1m} 
\\
& \qquad \times \prod_{i \not\ne i_0} \bigg[\int_{\lambda}^{2 \lambda} \int_{\R} 
|f_i(x)|^m \, dx \, dt \bigg]^{\frac1m} 
\\
& \lesssim \lambda^{\frac{m-1}{m}} \prod_{j=1}^{m+1} \|f_j\|_{L^m}
\end{align*}
which indicates \eqref{eq:v-tar} as desired. 
\end{proof}

The following result transfers the translation to weighted estimates.

\begin{lemma}\label{lem:trans}
Let $\kappa> d$ and $1 < r_1, r_2, r_3 < \infty$ be such that $\frac{1}{r_1} + \frac{1}{r_2} + \frac{1}{r_3} > 1$. Then there exist $t_1, t_2, t_3 \in \bT^d = [0, 1]^d$ such that  
\begin{equation*}
\Lambda_{\vec r}^{*}
\big(f_1(\cdot - t_1), f_2(\cdot - t_2), f_3(\cdot - t_3)\big)
\lesssim \prod_{i=1}^3 \|\br{\cdot}^{\kappa} f_i\|_{L^{r_i}},
\end{equation*}
for all $f_i \in \mathscr{S}(\R^d)$, where $\langle x \rangle = (1+|x|^2)^{1/2}$.
\end{lemma}

\begin{proof}
We may assume $\prod_{i=1}^3 \|\br{\cdot}^{\kappa} f_i\|_{L^{r_i}} \ne 0$; otherwise, there is nothing to prove. 
Let $\S$ be a $\gamma$-sparse family. Then there exists a disjoint family $\{E_Q\}_{Q \in \S}$ such that $|E_Q| \ge \gamma |Q|$. Then for all $t_1, t_2, t_3 \in \bT^d$,
\begin{align}\label{eq:Qave}
\gamma \sum_{Q \in \S} |Q| \prod_{i=1}^3 \langle f_i(\cdot - t_i) \rangle_{Q, r_i}
& \le \sum_{Q \in \S} \int_{E_Q} \prod_{i=1}^3 M_{r_i} f_i(x - t_i) \, dx
\\ \nonumber
& \le \int_{\Rd} \prod_{i=1}^3 M_{r_i} f_i(x - t_i) \, dx
=: \mathcal{I}(\vec{t}).
\end{align}
We claim that
\begin{equation}\label{eq:ave}
\mathcal{B}_r f(x) 
:= \int_{\bT^d} M_r f(x-t) \, dt
\lesssim \langle x \rangle^{-d/r} \|\langle \cdot \rangle^{\kappa} f\|_{L^r}, \quad x \in \Rd, \, r>1.
\end{equation}
Since $\int_{\Rd} \langle x \rangle^{-d(\frac{1}{r_1} + \frac{1}{r_2} + \frac{1}{r_3})} dx < \infty$, we have
\begin{align*}
\int_{\bT^{3d}} \mathcal{I}(\vec{t}) \, d\vec{t}
= \int_{\Rd} \prod_{i=1}^3 \mathcal{B}_{r_i} f_i(x) \, dx
\le C_0 \prod_{i=1}^3 \|\langle \cdot \rangle^{\kappa} f_i\|_{L^{r_i}},
\end{align*}
which gives
\begin{align*}
\Big| \Big\{t \in \bT^{3d}: \mathcal{I}(\vec{t}) > 2 C_0 
\prod_{i=1}^3 \|\langle \cdot \rangle^{\kappa} f_i\|_{L^{r_i}} \Big\} \Big|
\le \frac12.
\end{align*}
Thus, there exists some $\vec{t} = (t_1, t_2, t_3) \in \bT^3$ such that $\mathcal{I}(\vec{t}) \le 2 C_0 \prod_{i=1}^3 \|\langle \cdot \rangle^{\kappa} f_i\|_{L^{r_i}}$, which along with \eqref{eq:Qave} implies the desired estimate. 

Let us turn to the proof of \eqref{eq:ave}. For any measurable $E$ of measure one, the weak $(1,1)$ bound for $M$ yields,
\begin{align}\label{eq:Kol}
\int_E M_r g \, dx
& = \int_0^{\infty} |\{x \in E: M_r g(x) > t\}| \, dt
\\ \nonumber
& \lesssim \int_0^{\|g\|_{L^r}} 1 \, dt + \int_{\|g\|_{L^r}}^{\infty} t^{-r} \|g\|_{L^r}^r \,dt
\lesssim \|g\|_{L^r}.
\end{align}
If $|x| \le 4$, then \eqref{eq:Kol} applied to $E=x-\bT^d$ gives
\begin{align*}
\mathcal{B}_r f(x) 
\lesssim \|f\|_{L^r}
\lesssim \langle x \rangle^{-d/r} \|\langle \cdot \rangle^{\kappa} f\|_{L^r}.
\end{align*}
If $|x| > 4$, then for any $t \in \bT^d$, $x-t \in Q$, and $y \in Q$ with $2|y| \le |x|$, 
\begin{align*}
\ell(Q) \ge |x-t-y| \ge |x| - |t| - |y| 
\ge |x| - |x|/4 - |x|/2 = |x|/4,
\end{align*}
which together with \eqref{eq:Kol} on $E=x-\bT^d$ leads to
\begin{align*}
\mathcal{B}_r f(x)
& \le \int_{\bT} M_r(f \mathbf{1}_{\{2|\cdot| \le x\}})(x-t) \, dt
+ \int_{\bT} M_r(f \mathbf{1}_{\{2|\cdot| > x\}})(x-t) \, dt
\\
& \lesssim |x|^{-d/r} \|f\|_{L^r} + \|f \mathbf{1}_{\{2|\cdot| > x\}}\|_{L^r}
\lesssim |x|^{-d/r} \|\langle \cdot \rangle^{\kappa} f\|_{L^r},
\end{align*}
provided $\kappa > d > d/r$.
\end{proof}

In addition, we establish an annular estimate for bilinear pseudo-differential operators as follows.
\begin{lemma}\label{lem:kmt}
Let $1\le p_1, p_2 \le \infty$ and $0< p_0 < \infty$ be such that $\frac{1}{p_0} \le \frac{1}{p_1} + \frac{1}{p_2}$. Suppose that $\supp_{\xi, \eta} a \subset \{|(\xi, \eta)| \le C_0 R\}$ for some $R \ge 1$. Then for some large integer $N \ge 1$,
\begin{equation*}
\|T_a\|_{L^{p_1} \times L^{p_2} \to L^{p_0}}
\lesssim \max_{|\alpha| + |\beta| + |\gamma| \le N} 
\|\partial_x^{\alpha} \partial_{\xi}^{\beta} \partial_{\eta}^{\gamma} a\|_{L^{\infty}} 
\, R^{d \Gamma(p_1, p_2, p_0)}.
\end{equation*}
In particular, for all $0 \le \delta \le \rho \le 1$ and $a_{\ell} \in BS_{\rho, \delta}^{m, N}$ with $\supp_{\xi, \eta} a_{\ell} \subset \{|(\xi, \eta)| \simeq 2^{\ell}\}$, 
\begin{equation*}
\|T_{a_{\ell}}\|_{L^{p_1} \times L^{p_2} \to L^{p_0}}
\lesssim \|a_{\ell}\|_{BS_{\rho, \delta}^{m, N}}
2^{\ell[ m + d(1-\rho) \Gamma(p_1, p_2, p_0) + \rho d (\frac{1}{p_1} + \frac{1}{p_2} - \frac{1}{p_0})]}.
\end{equation*}
\end{lemma}

\begin{proof}
It was shown in \cite[Theorem~1.4 and the proof of Proposition~4.1]{KMT} that given a moderate weight $W$ on $\R^{2d}$,
\begin{align*}
\|T_a\|_{L^{p_1}\times L^{p_2}\to L^{p_0}}
\lesssim \max_{|\alpha| + |\beta| + |\gamma| \le N} 
\|W^{-1} \partial_x^{\alpha} \partial_\xi^{\beta} \partial_\eta^{\gamma} a\|_{L^{\infty}}
\|W|_{\Z^{2d}}\|_{\mathcal{B}^{\rm{form}}_{q_1, q_2, q_3}},
\end{align*}
where $q_1 = \max\{2, p'_1\}$, $q_2 = \max\{2, p'_2\}$, $q_3 = \max\{2, p_0\}$, and
\begin{align*}
\|V\|_{\mathcal{B}^{\rm{form}}_{q_1, q_2, q_3}}
:= \sup_{F_i \ge 0, \|F_i\|_{\ell^{q_i}} \le 1}
\sum_{u, v \in \Z^d} V(u, v) F_1(u) F_2(v) F_3(u+v).
\end{align*}
Take
\begin{align*}
W(\xi, \eta) := \sum_{u, v \in \Z^d} V(u, v) (1 + |\xi - u| + |\eta - v|)^{-\kappa},
\end{align*}
where $V(u, v) := \mathbf{1}_{B(0, C_1 R)}(u) \mathbf{1}_{B(0, C_1 R)}(v)$, $C_1 > C_0 + 2\sqrt{d}$, and $\kappa > 2d$. For any $|(\xi, \eta)| \le C_0 R$, one can choose $u, v \in \Z^d$ such that $|\xi - u| \le \sqrt{d}$ and $|\eta - v| \le \sqrt{d}$. Thus, $|u| \le |\xi| + |\xi - u| \le C_0 R + \sqrt{d} < C_1 R$ and $|v| \le |\eta| + |\eta - v| \le C_0 R + \sqrt{d} < C_1 R$, which implies $W(\xi, \eta) \gtrsim 1$ on the set $\{|(\xi, \eta)| \le C_0 R\}$. Note that
\begin{align*}
W(\xi, \eta) 
= \sum_{u, v \in \Z^d} V(\xi - u, \eta - v) (1 + |u| + |v|)^{-\kappa}.
\end{align*}
Thus, it suffices to treat $\|V\|_{\mathcal{B}^{\rm{form}}_{q_1, q_2, q_3}}$. By H\"{o}lder's and Young's inequalities, we have
\begin{align*}
& \sum_{u, v \in \Z^d} V(u, v) F_1(u) F_2(v) F_3(u+v)
\\
& \le \sum_{w \in \Z^d} \sum_{u \in \Z^d}
(F_1 \mathbf{1}_{B(0, C_1 R)})(u) 
(F_2 \mathbf{1}_{B(0, C_1 R)})(w-u) 
(F_3 \mathbf{1}_{B(0, 2C_1 R)})(w) 
\\
& \le \|(F_1 \mathbf{1}_{B(0, C_1 R)})*(F_2 \mathbf{1}_{B(0, C_1 R)})\|_{\ell^3}
\|F_3 \mathbf{1}_{B(0, 2C_1 R)}\|_{\ell^{3/2}}
\\
& \le \prod_{i=1}^3 \|F_i \mathbf{1}_{B(0, 2C_1 R)}\|_{\ell^{3/2}}
 \lesssim R^{d(2 - \frac{1}{q_1} - \frac{1}{q_2} - \frac{1}{q_3})}
 \prod_{i=1}^3 \|F_i\|_{\ell^{q_i}}.
\end{align*}
Note that $2 - \frac{1}{q_1} - \frac{1}{q_2} - \frac{1}{q_3} = \Gamma(p_1, p_2, p_0)$. This gives $\|V\|_{\mathcal{B}^{\rm{form}}_{q_1, q_2, q_3}} \lesssim R^{d \Gamma(p_1, p_2, p_0)}$.

To proceed, define $a(x, \xi, \eta) := a_{\ell}(2^{-\rho\ell}x, 2^{\rho\ell}\xi, 2^{\rho\ell}\eta)$. Then for $|\alpha| + |\beta| + |\gamma| \le N$, Leibniz's rule gives
\begin{align*}
|\partial_x^{\alpha} \partial_{\xi}^{\beta} \partial_{\eta}^{\gamma} a(x,\xi,\eta)| 
& = 2^{\rho \ell(|\beta| + |\gamma| - |\alpha|)}
|(\partial_x^{\alpha} \partial_{\xi}^{\beta} \partial_{\eta}^{\gamma} a_{\ell}) 
(2^{-\rho \ell} x, 2^{\rho \ell} \xi, 2^{\rho \ell} \eta)| 
\\
& \lesssim \|a_{\ell}\|_{BS_{\rho, \delta}^{m, N}} \, 2^{m \ell}2^{-(\rho - \delta) |\alpha| \ell}
\le \|a_{\ell}\|_{BS_{\rho, \delta}^{m, N}} \, 2^{m \ell},
\end{align*}
Since $\supp_{\xi, \eta} a \subset \{(\xi,\eta): |(\xi, \eta)| \lesssim 2^{(1 - \rho) \ell} \}$, the above estimate yields
\begin{align*}
\|T_a\|_{L^{p_1} \times L^{p_2} \to L^{p_0}}
\lesssim \|a_{\ell}\|_{BS_{\rho, \delta}^{m, N}} \, 
2^{\ell[m + d(1 - \rho) \Gamma(p_1, p_2, p_0)]}.
\end{align*}
Thus, the desired inequality follows from this and the identity 
\begin{align*}
T_{a_{\ell}}(f_1, f_2)(x)
= T_a \big(f_1(2^{-\rho \ell} \, \cdot), f_2(2^{-\rho \ell} \, \cdot)\big)(2^{\rho \ell} x).
\end{align*} 
The proof is complete. 
\end{proof}

We then record a Sobolev transfer estimate.
\begin{lemma}\label{lem:transfer}
Let $\Psi, \Phi \in \mathscr{S}(\R^{2d})$ be such that $\supp \Psi \subset \{1/2 \le |\zeta| \le 2\}$ and 
$\partial^{\alpha} \Phi(0) = 0$ for all $|\alpha| < N$, where $N > s_1 + s_2 + 2d$ and $s_1, s_2 > 0$. Then there exists some $c_0 > 0$ such that
\begin{equation*}
\|\Psi \, \Phi(\lambda \cdot)^2\sigma(\lambda \cdot)\|_{W^{(s_1, s_2)}}
\lesssim \min\{\lambda^{c_0}, \lambda^{-c_0}\} \|\sigma\|_{W^{(s_1, s_2)}}, \quad \forall \lambda > 0.
\end{equation*}
Consequently, for any finite set $I \subset \Z$ and $\sup_{i \in \Z} |e_i| \le 1$,
\begin{equation*}
\sup_{t>0} \bigg\|\Psi \sum_{i \in I} e_i \Phi (2^i t \cdot)^2 \sigma_i(2^i t \, \cdot) \bigg\|_{W^{(s_1, s_2)}}
\lesssim \sup_{i \in \Z} \|\sigma_i\|_{W^{(s_1, s_2)}}.
\end{equation*}
\end{lemma}

\begin{proof}
It suffices to show the first inequality since the second one follows from the first one and that
\begin{equation*}
\sup_{t>0} \sum_{i \in \Z} \min\{(2^i t)^{c_0}, (2^i t)^{-c_0}\}
\lesssim 1.
\end{equation*}
By change of variables,
\begin{equation*}
\|\sigma(\lambda \cdot)\|_{W^{(s_1, s_2)}}^2
= \lambda^{-2d} \int_{\R^{2d}} |\widehat{\sigma}(y, z)|^2
\langle \lambda y \rangle^{2s_1} \langle \lambda z \rangle^{2s_2} \, dy \, dz,
\end{equation*}
which gives
\begin{equation}\label{eq:dW-1}
\|\sigma(\lambda \cdot)\|_{W^{(s_1, s_2)}} 
\le
\begin{cases} 
\lambda^{-d} \|\sigma\|_{W^{(s_1, s_2)}}, & 0< \lambda < 1, 
\\
\lambda^{s_1+s_2 - d}\|\sigma\|_{W^{(s_1, s_2)}}, & \lambda \ge 1.
\end{cases}
\end{equation}
By Leibniz's rule, for all $N_0 < N$,
\begin{equation}\label{eq:dW-2}
\max_{|\alpha| \le N_0} \|\partial^{\alpha} [\Psi \Phi(\lambda \cdot)^2]\|_{L^{\infty}} 
\lesssim_M
\begin{cases} 
\lambda^{2N}, & 0< \lambda < 1, 
\\
\lambda^{-M}, & \lambda \ge 1.
\end{cases}
\end{equation}
for any $M>0$. Then the desired inequality is a consequence of \eqref{eq:dW-1}, \eqref{eq:dW-2} and the following
\begin{align}\label{eq:dW-3}
\|\varphi_1 \varphi_2\|_{W^{(s_1, s_2)}} 
\lesssim \max_{|\alpha| \le 2N_1} \|\partial^{\alpha} \varphi_1\|_{L^{\infty}} 
\|\varphi_2\|_{W^{(s_1, s_2)}}, 
\end{align}
for all $\varphi_1 \in \mathscr{S}(\R^{2d})$ and $2N_1 > s_1 + s_2 +2d$. Denote $H(u, v) := \langle u \rangle^{s_1} \langle v \rangle^{s_2}$. By the fact $H(u_1 + u_2, v_1 + v_2) \lesssim H(u_1, v_1) H(u_2, v_2)$, 
\begin{align*}
\|\varphi_1 \varphi_2\|_{W^{(s_1, s_2)}} 
& = \|\widehat{\varphi_1 \varphi_2}  H\|_{L^2}
\lesssim \|(|\widehat{\varphi}_1| H)*(|\widehat{\varphi}_2| H)\|_{L^2}
\\
& \le \| |\widehat{\varphi}_1| H\|_{L^1} \| |\widehat{\varphi}_2| H\|_{L^2}
= \| |\widehat{\varphi}_1| H\|_{L^1} \|\varphi_2\|_{W^{(s_1, s_2)}}.
\end{align*}
Using integration by parts, we have
\begin{align*}
|\widehat{\varphi}_1(\omega)|
= \bigg| \int_{\R^{2d}} \frac{e^{-2\pi i \omega \cdot \zeta} (1-\Delta_{\zeta})^{N_1} \varphi_1(\zeta)}{(1 + |2\pi \omega|^2)^{N_1}} \, d\zeta \bigg|
\lesssim \frac{\max_{|\alpha| \le 2N_1} \|\partial^{\alpha} \varphi_1\|_{L^{\infty}}}{(1 + |\omega|)^{2N_1}}.
\end{align*}
Hence, \eqref{eq:dW-3} follows from two estimates above and that $2N_1 > s_1 + s_2 +2d$.
\end{proof}

\section{A local induction of the main Theorem}\label{sec:red}
We will show Theorem \ref{thm:main} by induction on the number of scales. By reindexing, we may assume that $A_1 = 1$ in \eqref{H1}. From now on, we fix $1 < r_1, \ldots, r_m, s < \infty$ with $\frac1s < \frac1r := \sum_{i=1}^m \frac{1}{r_i}$. For convenience, we simply write $r_{m+1} = s'$, $\B_{m+1} = \B^*$, and $\langle f_i \rangle_{Q, p} = \langle |f_i|_{\B_i} \rangle_{Q, p}$. Fix $\gamma \in (0, 1)$ and a dyadic lattice $\D$ of one of the following form
\begin{align*}
\mathscr{D}^{\alpha} 
:= \big\{2^{-k} \big([0, 1)^d + n) + (-1)^k \alpha \big): k \in \Z, n \in \Z^d\big\}, \quad \alpha \in \{1/3, 2/3\}^d.
\end{align*}
Let $\D_k := \{Q \in \D: \ell(Q) = 2^{-k}\}$ for each $k \in \Z$.

\subsection{Stopping time argument}
Given $\lambda \ge 1$ and a dyadic cube $Q_0 \in \D$, define the local sparse form
\begin{align*}
\mathbf{S}_{Q_0}^{\lambda}(\vec{f}, f_{m+1})
:= \sup_{\substack{\S \subset \D(Q_0) \\ \S: \gamma\text{-sparse}}}
\sum_{Q \in \S} |Q| \prod_{i=1}^{m+1} \langle f_i \rangle_{\lambda Q, r_i},
\end{align*}
where $\D(Q_0)$ consists of $Q_0$ and all its dyadic descendants.
\black

\begin{lemma}\label{lem:SL}
There holds
\begin{align*}
\sup_{\lambda \ge 1} \sup_{Q_0 \in \D} \mathbf{S}_{Q_0}^{\lambda}(\vec{f}, f_{m+1})
\lesssim \Lambda_{\vec{r}, r_{m+1}}^*(\vec{f}, f_{m+1}),
\end{align*}
for all $f_1, \ldots, f_{m+1} \in L_c^{\infty}(\Rd)$.
\end{lemma}

\begin{proof}
Fix $\lambda \ge 1$, a dyadic cube $Q_0 \in \D$, and a $\gamma$-sparse family $\S \subset  \D(Q_0)$. By the Three Lattice Theorem (see  \cite[Theorem 3.1]{LN}), there exist $3^d$ dyadic lattices $\{\mathscr{D}^k\}_{k=1}^{3^d}$ such that for every $Q \in \S$, there exists some $k \in \{1, \ldots, 3^d\}$ and $P(Q) \in \mathscr{D}^k$ such that $\lambda Q \subset P(Q)$ and $\ell(P(Q)) \le 3 \ell(\lambda Q)$. This implies
\begin{align*}
\sum_{Q \in \S} |Q| \prod_{i=1}^{m+1} \langle f_i \rangle_{\lambda Q, r_i}
\le 3^{(m+1) d} \sum_{k=1}^{3^d} \sum_{P \in \mathscr{D}^k} 
\sum_{Q \in \S: P(Q)=P} |Q| \prod_{i=1}^{m+1} \langle f_i \rangle_{P, r_i}
=: 3^{(m+1) d} \sum_{k=1}^{3^d} \Xi_{\mathscr{D}^k}.
\end{align*}
By the property of dyadic lattices, the finite collection $\{P \in \mathscr{D}^k: P \cap \supp f_{m+1} \ne \varnothing\}$ has common ancestor $P_0$. Thus, it suffices to show 
\begin{align}\label{eq:XDP}
\Xi_{\D(P_0)} 
\lesssim \Lambda_{\vec{r}, r_{m+1}}^*(\vec{f}, f_{m+1}).
\end{align}

Given $K_0 > 1$ and $Q \in \mathscr{D}^k$, we subdivide dyadically $Q$ and stop the first time that one of the following $(m+1)$ conditions occurs
\begin{align}\label{eq:stopcon}
\langle f_i \rangle_{Q', r_i} > K_0 \langle f_i \rangle_{Q, r_i}, \quad i=1, \ldots, m+1. 
\end{align}
Let $\mathrm{Stop}(Q)$ be the stopping time family. Choose $K_0$ sufficiently large so that 
\begin{align*}
|\Omega_Q| := \bigg|\bigcup_{Q' \in \mathrm{Stop}(Q)} Q' \bigg| \le (1-\gamma) |Q|,
\end{align*} 
which follows from
\begin{align*}
\sum_{\substack{Q' \in \mathrm{Stop}(Q) \\ \langle f_i \rangle_{Q', r_i} > K_0 \langle f_i \rangle_{Q, r_i}}} |Q'|
\le (K_0 \langle f_i \rangle_{Q, r_i})^{-r_i}
\sum_{Q' \in \mathrm{Stop}(Q)} \int_{Q'} |f_i|^{r_i}
\le K_0^{-1} |Q|.
\end{align*}
Thus, $Q \notin \mathrm{Stop}(Q)$. Define 
\begin{align}\label{def:Stop}
\mathrm{Stop}^{(0)}(Q) := \{Q\} \quad \text{ and } \quad
\mathrm{Stop}^{(k+1)}(Q) := \bigcup_{Q' \in \mathrm{Stop}^{(k)}(Q)} \mathrm{Stop}(Q'), \quad k \ge 0.
\end{align} 
If we set
\begin{align}\label{def:Good}
& \G(Q) := \{Q' \in \D(Q): Q' \not\subset P \text{ for all } P \in \mathrm{Stop}(Q)\},
\end{align} 
then there holds
\begin{align}\label{eq:DQSt}
\D(Q) = \G(Q) \sqcup \bigsqcup_{Q' \in \mathrm{Stop}(Q)} \D(Q').
\end{align}
Iterating this procedure we eventually obtain
\begin{align}\label{eq:DPG}
\D(P_0) = \sum_{P \in \S(P_0)} \G(P)
\quad \text{ and } \quad
\S(P_0) = \bigcup_{k=0}^{\infty} \mathrm{Stop}^{(k)}(P_0),
\end{align}
which gives
\begin{align}\label{eq:XDSG}
\Xi_{\D(P_0)}
\le \sum_{P \in \S(P_0)} \Xi_{\G(P)}.
\end{align}

We claim that
\begin{align}\label{eq:gas}
\text{the family $\S(P_0)$ is $\gamma$-sparse}.
\end{align}
Indeed, given $Q \in \mathcal{S}(P_0)$, if we denote $E_Q := Q \setminus \bigcup_{Q' \in \mathrm{Stop}(Q)} Q'$, then the construction of $\mathrm{Stop}(Q)$ gives $|E_Q| = |Q \setminus \Omega_Q| \ge \gamma |Q|$. Let $Q_1, Q_2 \in \S(P_0)$ with $Q_1 \ne Q_2$. Then $Q_1 \in \mathrm{Stop}^{(k_1)}(P_0)$ and $Q_2 \in \mathrm{Stop}^{(k_2)}(P_0)$ for some $k_1, k_2 \ge 0$. If $k_1 = k_2$, the maximality of $\mathrm{Stop}^{(k_1)}(Q_0)$ implies $Q_1 \cap Q_2 = \varnothing$, hence, $E_{Q_1} \cap E_{Q_2} = \varnothing$. If $k_1 \ne k_2$, we may assume $k_1 < k_2$. By construction, one can find some $Q'_1 \in \mathrm{Stop}(Q_1)$ so that $Q_2 \in Q'_1$, which gives $E_{Q_1} \cap Q_2 = \varnothing$, hence, $E_{Q_1} \cap E_{Q_2} = \varnothing$. This means $\{E_Q\}_{Q \in \S(P_0)}$ is pairwise disjoint. 

Note that by the construction of $\mathrm{Stop}(P)$ and $\G(P)$,
\begin{align*}
\langle f_i \rangle_{P', r_i} \le K_0 \langle f_i \rangle_{P, r_i}, 
\quad\forall P \in \S(P_0),\, P' \in \G(P), \, i=1, \ldots, m+1, 
\end{align*}
which implies
\begin{align*}
\Xi_{\D(P_0)}
& \le \sum_{P \in \S(P_0)} \Xi_{\G(P)}
\le K_0^{m+1} \sum_{P \in \S(P_0)} \sum_{Q \in \S: P(Q) \subset P} |Q| 
\prod_{i=1}^{m+1} \langle f_i \rangle_{P, r_i}
\\
& \le K_0^{m+1} \gamma^{-1} \sum_{P \in \S(P_0)} |P| 
\prod_{i=1}^{m+1} \langle f_i \rangle_{P, r_i}
\le K_0^{m+1} \gamma^{-1}
\Lambda_{\vec{r}, r_{m+1}}^*(\vec{f}, f_{m+1}),
\end{align*}
where we have used the $\gamma$-sparseness of $\S$ and $\S(P_0)$. This completes the proof of \eqref{eq:XDP}.
\end{proof}

Let us present a further result based on Lemma \ref{lem:SL} and its proof. Define
\begin{align*}
\widehat{P_{\ell} f}(\xi) 
:= \psi(2^{-\ell} \xi), \quad \ell \ge 1, 
\end{align*}
where $\psi \in \mathscr{C}_c^{\infty}(\Rd)$ with $\psi(\xi) = 1$ if $c_0 \le |\xi| \le C_0$ and $\supp \psi \subset \{c_0/2 \le |\xi| \le 2C_0\}$. Here $0<c_0 < C_0 < \infty$ are harmless constants.

\begin{lemma}\label{lem:pp}
Let $\mathcal D$ be a dyadic lattice. Then for every nonnegative bounded sequence $\{k_{\ell}\}_{\ell \ge 1}$,
\begin{align*}
\Xi_{\D} :=
\sum_{\ell \ge 1} \sum_{\substack{Q \in \mathcal D \\ \ell(Q) = 2^{k_{\ell}}}}
|Q| \langle P_{\ell} f_1 \rangle_{Q, p_1}
\langle P_{\ell} f_2 \rangle_{Q, p_2}
\langle f_3 \rangle_{Q, p_3}
\lesssim \Lambda^*_{p_1, p_2, p_3}(f_1, f_2, f_3),
\end{align*}
for all $f_1, f_2, f_3 \in L_c^{\infty}(\Rd)$, where $1 < p_1, p_2 \le 2$ and $1 \le p_3 < \infty$.
\end{lemma}

\begin{proof}
Define a dyadic square function by
\begin{equation*}
S_{\D}^p f
:= \bigg(\sum_{\ell \ge 1}
\sum_{\substack{Q \in \D \\ \ell(Q) = 2^{k_{\ell}}}}
\langle P_{\ell} f \rangle_{Q, p}^2 \, \mathbf{1}_Q \bigg)^{\frac12}.
\end{equation*}
We claim that for any $Q \in \D$,
\begin{equation}\label{eq:anm}
\|S_{\D(Q)}^{p_i} f_i\|_{L^{p_i, \infty}(Q)}
\lesssim_{\kappa} |Q|^{\frac{1}{p_i}} \sum_{j \ge 0} 2^{-j \kappa} \langle f_i \rangle_{3 \cdot 2^j Q, p_i}
=: |Q|^{\frac{1}{p_i}} G_{Q, p_i}, \quad i=1, 2,
\end{equation}
for all $\kappa > 0$. To show \eqref{eq:anm}, we split $f_i = f_i \mathbf{1}_{3Q} + f_i \mathbf{1}_{(3Q)^c}$. Since $\check{\psi}$ is Schwartz, the kernel of $P_{\ell}$ satisfies $|p_{\ell}(x)| \lesssim 2^{d \ell}(1 + 2^{\ell} |x|)^{-\kappa-d}$, which gives
\begin{align*}
|P_{\ell}(f_i \mathbf{1}_{(3 Q)^c})(x)|
\lesssim [2^{\ell} \ell(Q)]^{-A}
\sum_{j \ge 1} 2^{-j \kappa} \langle f \rangle_{3 \cdot 2^j Q}
\le [2^{\ell} \ell(Q)]^{- \kappa} G_{Q, p_i}, \quad\forall x \in Q,
\end{align*}
 and, hence,
\begin{align*}
\|S_{\D(Q)}^{p_i}(f_i \mathbf{1}_{(3Q)^c})\|_{L^{p_i, \infty}(Q)} 
\lesssim |Q|^{\frac{1}{p_i}} \|S_{\D(Q)}^{p_i}(f \mathbf{1}_{(3Q)^c})\|_{L^{\infty}(Q)} 
\lesssim G_{Q, p_i}.
\end{align*}
If $p=2$, we have 
\begin{align*}
\|S_{\D(Q)}^2 (f_i \mathbf{1}_{3Q})\|_{L^{2, \infty}(Q)} 
\le \|S_{\D}^2 (f_i \mathbf{1}_{3Q})\|_{L^2} 
\lesssim \|f_i \mathbf{1}_{3Q}\|_{L^2}
\simeq |Q|^{\frac12} \langle f \rangle_{3Q, 2}.
\end{align*} 
In the case $1<p_i<2$, by the weak $(1,1)$ inequality for vector-valued maximal functions \cite{FS} and the Littlewood--Paley theory,
\begin{align*}
\|S_{\D(Q)}^{p_i} (f_i \mathbf{1}_{3Q})\|_{L^{p_i, \infty}(Q)}
& \le \|S_{\D}^{p_i} (f_i \mathbf{1}_{3Q})\|_{L^{p_i, \infty}}
\le \bigg\|\bigg( \sum_{\ell \ge 1} [M(|P_{\ell} (f_i \mathbf{1}_{3Q})|^{p_i})]^{\frac{2}{p_i}} \bigg)^{\frac12} \bigg\|_{L^{p_i, \infty}}
\\
& \lesssim \bigg\|\bigg( \sum_{\ell \ge 1} |P_{\ell} (f_i \mathbf{1}_{3Q})|^2 \bigg)^{\frac12} \bigg\|_{L^{p_i}}
\lesssim \|f \mathbf{1}_{3Q}\|_{L^{p_i}}
\simeq |Q|^{\frac{1}{p_i}} \langle f \rangle_{3Q, p_i}.
\end{align*}
This justifies \eqref{eq:anm}.

Fix $Q \in \D$ and set
\begin{align*}
\Omega_Q 
:= \bigcup_{i=1}^2 \big\{S_{\D(Q)}^{p_i} f_i > K_0 G_{Q, p_i} \big\} 
\cup \{M_{Q, p_3} f_3 > K_0 \langle f_3 \rangle_{Q, p_3} \},
\end{align*}
where $M_{Q, p_3} f(x) := \sup_{x \in Q' \in \D(Q)} \langle f \rangle_{Q', p_3}$. By \eqref{eq:anm} and the weak type estimate for $M_{Q, p_3}$, one can choose $K_0$ large enough so that $|\Omega_Q| \le (1-\gamma) |Q|$. Let $\mathrm{Stop}(Q)$ be the collection of maximal subcubes of $Q$ such that $\Omega_Q = \bigcup_{Q' \in \mathrm{Stop}(Q)} Q'$. Let $Q_0$ be a common ancestor of the collection $\{Q \in \D: \ell(Q)= , Q \cap \supp f_3 \ne \varnothing\}$. As done in \eqref{def:Stop}--\eqref{eq:gas}, we have
\begin{align}\label{eq:SRf}
\Xi_{\D}
\le \Xi_{\D(Q_0)}
\le \sum_{Q \in \S(Q_0)} \Xi_{\G(Q)}
\quad \text{and} \quad
\text{$\S(Q_0)$ is $\gamma$-sparse}. 
\end{align}

We are going to estimate $\Xi_{\G(Q)}$ for $Q \in \D$. The maximality of $\mathrm{Stop}(Q)$ implies
\begin{align}\label{eq:mma-1}
\langle f_3 \rangle_{Q', p_3} \le K_0 \langle f_3 \rangle_{Q, p_3}, \quad \forall Q' \in \G(Q).
\end{align}
Moreover,
\begin{align}\label{eq:mma-2}
S_{\G(Q)}^{p_i} f_i(x)
\le K_0 G_{Q, p_i}, \quad \forall x \in Q, \, i = 1, 2.
\end{align}
Indeed, if $x \notin \Omega_Q$, the above follows from $S_{\D(Q)}^{p_i}f_i(x) \le K_0 $. Let $x \in \Omega_Q$. Then $x \in P$ for some $P \in \mathrm{Stop}(Q)$. Let $\widehat{P} \in \D$ be the dyadic father of $Q$. If $\widehat{P} \setminus \Omega_Q = \varnothing$, then $P \subsetneq \widehat{P} \subset \Omega_Q$, which contradicts the maximality of $P$. Hence, there exists some $x_P \in \widehat{P} \setminus \Omega_Q$. Note that for every $Q' \in \G(Q)$ containing $x$, there must $P \subsetneq Q'$, which implies $x_P \in \widehat{P} \subset Q'$. Then \eqref{eq:mma-2} follows from $S_{\D(Q)}^{p_i}f_i(x_P) \le K_0 $.

As a consequence of \eqref{eq:SRf}--\eqref{eq:mma-2} and the Cauchy--Schwarz inequality, we have 
\begin{align*}
\Xi_{\D}
&\le \sum_{Q \in \S(Q_0)} \Xi_{\G(Q)}
\lesssim \sum_{Q \in \S(Q_0)} |Q| G_{Q, p_1} G_{Q, p_2}
\langle f_3 \rangle_{Q, p_3}
\\
&\lesssim \sum_{k, j \ge 0} 2^{- \kappa (k + j)} 
2^{d (\frac{\lambda - k}{p_1} + \frac{\lambda - j}{p_2} + \frac{\lambda}{p_3})}
\sum_{Q \in \S(Q_0)} |Q| \prod_{i=1}^3 \langle f_i \rangle_{3 \cdot 2^{\lambda} Q, p_i}
\\
& \lesssim \Lambda^*_{p_1, p_2, p_3}(f_1, f_2, f_3),
\end{align*}
provided that $\lambda = \max\{k, j\}$, $\kappa > 2d$, and Lemma \ref{lem:SL}. 
\end{proof}

\subsection{A recursive proposition}
We will demonstrate Theorem~\ref{thm:main} by induction. 
\begin{definition}\label{def:Un}
Given $n = 0, 1, \ldots$, let $U_n$ be the smallest constant such that for all pairs of integers $(N_1, N_2)$ with $0 \leq N_2 - N_1 \leq n$, and for all $Q_0 \in \D$ with $\ell(Q_0) = 2^{N_2}$, there holds
\begin{align*}
\bigg| \bigg\langle \sum_{j=N_1}^{N_2} T_j(\vec{f}), \, f_{m+1} \bigg\rangle \bigg|
\le U_n \, \mathbf{S}_{Q_0}^3(\vec{f}, f_{m+1}),
\end{align*}
for all $f_i \in \mathscr{S}_{\B_i}$, $i=1, \ldots, m+1$, with $\bigcup_{i=1}^m\supp(f_i) \subset 3Q_0$ and $\supp(f_{m+1}) \subset Q_0$.
\end{definition}

The core of the proof of Theorem~\ref{thm:main} is the following recursive proposition.

\begin{proposition}\label{prop:Un}
There exists a constant $C_0 \in (0, \infty)$ such that
\begin{align*}
U_n \le \max \{U_{n-1}, C_0 A_0 \}, \quad \forall n \ge 1.
\end{align*}
\end{proposition}

\begin{proof}[\bf Proof of Theorem~\ref{thm:main}]
We first treat the situation of simple functions. By Lemma \ref{lem:SL} and Proposition \ref{prop:Un}, it suffices to prove 
\begin{align*}
U_0 \lesssim A_0,
\end{align*}
which further implies $U_n \le (n+1) U_0 < \infty$ for every $n \ge 0$. To this end, let $Q \in \D$ with $\ell(Q) = 2^j$, and take $f_i \in \mathscr{S}_{\B_i}$, $i=1, \ldots, m+1$, with $\bigcup_{i=1}^m \supp(f_i) \subset 3Q$ and $\supp f_{m+1} \subset Q$. It follows from H\"older's inequality and the assumption \eqref{H4} that
\begin{align*}
& |\langle T_j(\vec{f}), f_{m+1} \rangle|
\le \|T_j(\vec{f})\|_{L^s} \|f_{m+1} \mathbf{1}_{3Q}\|_{L^{s'}}  
\\
& \le A_4 \, 2^{-jd(\frac1r - \frac1s)}
\prod_{i=1}^{m+1} \|f_i \mathbf{1}_{3Q}\|_{L^{r_i}}
\lesssim A_4 |Q| \prod_{i=1}^{m+1} \langle f_i \rangle_{3Q, r_i},
\end{align*}
where the implicit constant is independent of $j$. This shows $U_0 \lesssim A_4 \le A_0$.

Next, let us extend the result above to the case of compactly supported bounded functions. Fix $N_1, N_2 \in \Z$ with $N_1 \le N_2$, and write $T = \sum_{j=N_1}^{N_2} T_j$. Fix $f_i \in L_c^{\infty}(\Rd; \B_i)$, $i=1, \ldots, m+1$. There exists $h_i^{(k)} \in \mathscr{S}_{\B_i}$ such that $h_i^{(k)}(x) \to f_i(x)$ as $k \to \infty$, a.e. $x \in \Rd$. Define 
\begin{align*}
f_i^{(k)}(x) := h_i^k(x) \mathbf{1}_{\{|h_i^k(x)|_{\B_i} \le 2 |f_i(x)|_{\B_i}\}}, \quad i=1, \ldots, m+1.
\end{align*}
By the sparse bound for simple functions, 
\begin{align}\label{eq:ext}
|\langle T(\vec{f}^{(k)}), f_{m+1}^{(k)} \rangle|
\lesssim A_0 \Lambda_{\vec{r}, r_{m+1}}^*(\vec{f}^{(k)}, f_{m+1}^{(k)})
\lesssim A_0 \Lambda_{\vec{r}, r_{m+1}}^*(\vec{f}, f_{m+1}),
\end{align}
where the implicit constants are independent of $k$. On the other hand, Lebesgue’s dominated convergence theorem gives $\lim_{k \to \infty} \|f_i^{(k)} - f_i\|_{L^{r_i}(\B_i)} = 0$.  The hypothesis \eqref{H4} implies
\begin{align*}
& \big| \langle T(\vec{f}^{(k)}), f_{m+1}^{(k)} \rangle 
- \langle T(\vec{f}), f_{m+1} \rangle \big|
\\
& \le \|T(\vec{f}^{(k)}) - T(\vec{f})\|_{L^s} \|f_{m+1}^{(k)}\|_{L^{s'}} 
+ \|T(\vec{f})\|_{L^s} \|f_{m+1}^{(k)} - f_{m+1}\|_{L^{s'}} 
\\
& \lesssim A_4 \sum_{j=1}^m \prod_{i<j} \|f_i^{(k)}\|_{L^{r_i}}  \|f_j^{(k)} - f_j\|_{L^{r_j}}  
\prod_{i>j} \|f_i^{(k)}\|_{L^{r_i}} \|f_{m+1}\|_{L^{s'}} 
\\
&\quad + A_4 \prod_{i=1}^m \|f_i^{(k)}\|_{L^{r_i}} \, \|f_{m+1}^{(k)} - f_{m+1}\|_{L^{s'}}.
\end{align*}
Thus, 
\begin{align*}
\lim_{k \to \infty} \big| \langle T(\vec{f}^{(k)}), f_{m+1}^{(k)} \rangle 
- \langle T(\vec{f}), f_{m+1} \rangle \big| = 0,
\end{align*}
which along with \eqref{eq:ext} yields 
\begin{align*}
|\langle T(\vec{f}), f_{m+1} \rangle|
\lesssim A_0 \Lambda_{\vec{r}, r_{m+1}}^*(\vec{f}, f_{m+1}).
\end{align*}
This completes the proof of Theorem \ref{thm:main}.
\end{proof}

Consequently, it remains to prove Proposition~\ref{prop:Un}. For this purpose, we present a regularity estimate in the dyadic setting.

\begin{lemma}\label{lem:reg}
For all $k, j \in \Z$, there holds
\begin{align*}
\|T_j \circ_i (\Id - \mathbb{E}_{k-j})\|_{L^{r_1} \times \cdots \times L^{r_m} \to L^s}
& \lesssim A_5 \, 2^{-k \vartheta} 2^{-jd(\frac1r - \frac1s)}, \quad i=1, \ldots, m,
\end{align*}
where
\begin{align*}
\mathbb{E}_j f 
:= \sum_{Q \in \D: \ell(Q)=2^{-j}} [f]_Q \mathbf{1}_Q.
\end{align*}
\end{lemma}

\begin{proof}
It suffices to prove the following inequality
\begin{align*}
|\langle T_j \circ_1 (\Id - \mathbb{E}_{k-j}) (\vec{f}), f_{m+1} \rangle|
\lesssim A_5 \, 2^{-k \eta} 2^{-jd(\frac1r - \frac1s)},
\end{align*}
where by homogeneity we have assumed $\|f_i\|_{L^{r_i}} \le 1$, $i=1, \ldots, m+1$. 

Let us first handle the case $k \le 0$. Let $Q \in \D_{-j}$. Recall that we always use $\ell^{\infty}$ in $\Rd$. Let $e_1 = (1, 0, \ldots, 0)$. Note that $\supp (\tau_{2^{j+3} e_1}(f_1\mathbf1_{5Q}) ) \subset \overline{5Q} - 2^{j+3} e_1$ and $\dist(Q, \overline{5Q} - 2^{j+3} e_1) = 5 \ell(Q) > 2^j$. Thus, by \eqref{H1} and the multilinearity of $T_j$,
\begin{align*}
\mathbf1_Q T_j \big(\tau_{2^{j+3} e_1}(f_1 \mathbf{1}_{5Q}), f_2 \mathbf{1}_{5Q}, \ldots, f_m \mathbf{1}_{5Q} \big) = 0,
\end{align*}
which gives
\begin{align*}
\mathbf{1}_Q T_j (\vec{f}) 
= \mathbf{1}_Q T_j (\vec{f} \mathbf{1}_{5Q}) 
= - \mathbf{1}_Q T_j \big(\Delta_{2^{j+3} e_1}(f_1 \mathbf{1}_{5Q}), f_2 \mathbf{1}_{5Q}, \ldots, f_m \mathbf{1}_{5Q}\bigr).
\end{align*}
Together with the fact $\Delta_{2^{j+3} e_1} = \sum_{k=0}^7 \Delta_{2^j e_1} \tau_{2^j k e_1}$, the above and \eqref{H5} imply 
\begin{align}\label{Tjkee-1}
\|T_j(\vec{f})\|_{L^s}
& = \bigg[\sum_{Q \in \D_{-j}} \|\mathbf{1}_Q T_j(\vec{f})\|_{L^s}^s \bigg]^{\frac1s}
\\ \notag
& \lesssim A_5 2^{-jd(\frac1r - \frac1s)} 
\bigg[\sum_Q \prod_{i=1}^m \|f_i \mathbf{1}_{5Q}\|_{L^{r_i}}^s \bigg]^{\frac1s}
\\ \notag
& \le A_5 2^{-jd(\frac1r - \frac1s)} 
\bigg[\sum_Q \prod_{i=1}^m \|f_i \mathbf{1}_{5Q}\|_{L^{r_i}}^r \bigg]^{\frac1r}
\\ \notag
& \le A_5 2^{-jd(\frac1r - \frac1s)} 
\prod_{i=1}^m \bigg[\sum_Q \|f_i \mathbf{1}_{5Q}\|_{L^{r_i}}^{r_i} \bigg]^{\frac{1}{r_i}}
\\ \notag
& \lesssim A_5 \, 2^{-k \eta} 2^{-jd(\frac1r - \frac1s)}  \prod_{i=1}^m \|f_i\|_{L^{r_i}},
\end{align}
provided that the collection $\{5Q: Q \in \D_{-j}\}$ is bounded overlapping. In view of the fact $\|\mathbb{E}_{k-j} f_1\|_{L^{r_1}} \le \|f_1\|_{L^{r_1}}$, the inequality \eqref{Tjkee-1} gives
\begin{align*}
\big\|T_j \circ_1 (\Id - \mathbb{E}_{k-j})\big\|_{L^{r_1} \times \cdots \times L^{r_m} \to L^s}
\le 2 \|T_j\|_{L^{r_1} \times \cdots \times L^{r_m} \to L^s}
\lesssim A_5 \, 2^{-k \eta} 2^{-jd(\frac1r - \frac1s)}.
\end{align*}

Next, we treat the case $k>0$. If we define a linear functional
\begin{align*}
\mathscr{L}(\varphi) 
:= \langle T_j(\varphi, f_2, \ldots, f_m), f_{m+1} \rangle,
\end{align*}
then the hypothesis \eqref{H4} gives
\begin{align*}
|\mathscr{L}(\varphi)|
\le A_4 \, 2^{-jd(\frac1r - \frac1s)} \|\varphi\|_{L^{r_1}}.
\end{align*}
This means $\mathscr{L}$ is a bounded linear functional on $L^{r_1}$. By duality, there exists $G \in L^{r_1'}$ so that
\begin{align*}
\mathscr{L}(\varphi) = \langle \varphi, G \rangle,
\quad \text{where} \quad
\|G\|_{L^{r'_1}} \le A_4 \, 2^{-jd(\frac1r - \frac1s)}.
\end{align*}
For all $ v \in [-2^{j-k}, 2^{j-k}]^d$, there holds $|v| \le 2^{j-k} \le 2^j$, which along with \eqref{H5} implies
\begin{align*}
|\langle \varphi, \Delta_v G \rangle|
= |\langle \Delta_{-v} \varphi, G \rangle|
= |\mathscr{L}(\Delta_{-v} \varphi)|
\le A_5 \, (2^{-j} |v|)^{\vartheta} \, 2^{-jd(\frac1r - \frac1s)} \|\varphi\|_{L^{r_1}}. 
\end{align*}
Thus, 
\begin{align*}
\|\Delta_v G \|_{L^{r'_1}}
\le A_5 \, (2^{-j} |v|)^{\vartheta} \, 2^{-jd(\frac1r - \frac1s)},
\end{align*}
which leads to
\begin{align*}
\|(\Id - \mathbb{E}_{k-j}) G\|_{L^{r'_1}}^{r'_1}
& = \sum_{Q \in \D_{k-j}} \int_Q \bigg|\frac{1}{|Q|} \int_Q [G(x) - G(y)] \,dy \bigg|^{r_1'} \, dx
\\
& \le 2^{-(j-k)d} \sum_{Q \in \D_{k-j}} \int_Q \int_Q |G(x) - G(y)|^{r'_1} \, dx \, dy
\\
& = 2^{-(j-k)d} \int_{[-2^{j-k}, 2^{j-k}]^d} 
\int_{\Rd} \sum_{Q \in \D_{k-j}} \mathbf{1}_Q(x) \mathbf{1}_Q(x+v)
|\Delta_v G(x)|^{r'_1} \, dx \, dv
\\
& \le 2^{-(j-k)d} \int_{[-2^{j-k}, 2^{j-k}]^d} \|\Delta_v G\|_{L^{r'_1}}^{r'_1} \, dv
\\
& \lesssim \Big[A_5 \, 2^{-k \vartheta} \, 2^{-jd(\frac1r - \frac1s)} \Big]^{r'_1}.
\end{align*}
Therefore, we have
\begin{align*}
& |\langle T_j \circ_1 (\Id - \mathbb{E}_{k-j})(\vec{f}), f_{m+1} \rangle|
= |\mathscr{L}((\Id - \mathbb{E}_{k-j}) f_1)|
\\
& = |\langle (\Id - \mathbb{E}_{k-j}) f_1, G \rangle|
= |\langle f_1, (\Id - \mathbb{E}_{k-j}) G \rangle|
\\
& \le \|f_1\|_{L^{r_1}} \|(\Id - \mathbb{E}_{k-j}) G\|_{L^{r'_1}}
\lesssim A_5 \, 2^{-k\vartheta} \, 2^{-jd(\frac1r - \frac1s)}.
\end{align*}
This completes the proof. 
\end{proof}

\section{The inductive step}\label{sec:main}
Now let us prove Proposition \ref{prop:Un}. By reindexing, we may assume that $A_1 = 1$ in \eqref{H1}. Fix $n \ge 1$, and assume that $U_{n-1}$ is known. Let $N_1\le N_2$ with $0 \le N_2 - N_1 \le n$, and let $Q_0 \in \D$ with $\ell(Q_0) = 2^{N_2}$. Simply write
\begin{align*}
T := \sum_{j=N_1}^{N_2} T_j.
\end{align*}
Given $p \ge 1$, write $M_p f := \big(M(|f|^p)\big)^{\frac1p}$, where $M$ is the Hardy--Littlewood maximal operator. Pick $0 < \tau < \min\{r, 1\}$. Given a constant $K_0 > 1$, define
\begin{align*}
\Omega := \bigcup_{i=1}^{m+1} \big\{M_{r_i} f_i > K_0 \langle f_i \rangle_{3Q_0, r_i} \big\}
\cup \Big\{M_{\tau} (T(\vec{f})) > K_0 A_2 \prod_{i=1}^m \langle f_i \rangle_{3Q_0, r_i} \Big\}.
\end{align*}
By Theorem 2.1.6 and Exercises 2.1.13 in \cite{Gra1}, 
\begin{align}\label{eq:Mweak}
M: L^1(\Rd) \to L^{1, \infty}(\Rd)
\quad \text{ and } \quad
M: L^{p, \infty}(\Rd) \to L^{p, \infty}(\Rd), \quad\forall p \in (1, \infty).
\end{align}
Then in view of \eqref{eq:Mweak} and \eqref{H2}, one can choose $K_0$ sufficiently large so that
\begin{equation}\label{eq:EQQ}
|\Omega| \le (1 - \gamma) |Q_0|
\quad \text{and} \quad 
|E_{Q_0}| := |Q_0 \setminus \Omega| \ge \gamma |Q_0|.
\end{equation}
Performing a Whitney decomposition of $\Omega$ (cf. \cite[p. 167]{Stein} for $c=5\sqrt{d}$), we obtain a pairwise disjoint family $\W \subset \D$ such that
\begin{align}\label{eq:Whd}
\Omega = \bigcup_{Q \in \W} Q \quad \text{and} \quad
4 \diam(Q) \le \dist(Q, \Omega^c) \le 10 \diam(Q), \quad\forall Q \in \W.
\end{align}
This gives the Calder\'{o}n--Zygmund decomposition
\begin{equation}\label{eq:CZ}
f_i = g_i + b_i, \quad i=1, \ldots, m+1,
\end{equation}
where
\begin{align*}
g_i & := f_i \mathbf{1}_{\Omega^c} + \sum_{Q \in \W} [f_i]_Q \mathbf{1}_Q,
\\
b_i & := \sum_{Q \in \W} b_{i, Q}
:= \sum_{Q \in \W} (f_i - [f_i]_Q) \mathbf{1}_Q.
\end{align*}
Recall that ${\displaystyle [f_i]_Q = \fint_Q f_i \, dx}$. Observe that by \eqref{eq:EQQ},
\begin{align}\label{eq:QQin}
Q \in \W_0 :=\{Q \in \W: Q \cap Q_0 \ne \varnothing\} 
\quad \Longrightarrow \quad Q \subsetneq Q_0,
\end{align}
which further gives
\begin{align}\label{eq:gbsupp}
\supp(g_{m+1}) \cup \supp(b_{m+1}) \subset 3 Q_0 
\quad \text{and} \quad
\supp(g_i) \cup \supp(b_i) \subset Q_0, \, i=1, \ldots, m.
\end{align}
In view of \eqref{eq:Whd}, there exists a constant $\kappa = \kappa(d) > 1$ such that 
\begin{align}\label{eq:kQ}
\kappa Q \cap \Omega^c \ne \varnothing \quad \text{for all } Q \in \W.
\end{align} 
This together with the definition of $\Omega$ implies
\begin{align}\label{eq:ffqp}
\langle f_i \rangle_{Q, r_i} \lesssim \langle f_i \rangle_{3Q_0, r_i}, 
\quad \forall Q \supset Q' \in \W, \, i=1, \ldots, m+1,
\end{align}
which immediately indicates 
\begin{align}\label{eq:bLp}
\langle b_i \rangle_{Q, r_i} 
\lesssim \langle f_i \rangle_{3Q_0, r_i}, \quad \forall Q \supset Q' \in \W, \, i=1, \ldots, m+1.
\end{align}
By the disjointness of $\W$, one has
\begin{align}\label{eq:bLr}
\bigg\|\sum_{Q \subset P} b_{i, Q} \bigg\|_{L^p} 
\lesssim |P|^{\frac1p} \langle f_i \rangle_{P, p}, \quad \forall p \in [1, \infty), \, i=1, \ldots, m+1.
\end{align}
In addition, by definition and \eqref{eq:kQ}, 
\begin{align}\label{eq:ggqp}
\|g_i\|_{L^{\infty}} \lesssim \langle f_i \rangle_{3Q_0, r_i}, \qquad i=1, \ldots, m+1,
\end{align}
which along with \eqref{eq:gbsupp} yields
\begin{align}\label{eq:ggss}
\|g_i\|_{L^{q, 1}} \lesssim |Q_0|^{\frac1q} \langle f_i \rangle_{3Q_0, r_i}, 
\quad \forall q \in (0, \infty), \, i=1, \ldots, m+1.
\end{align}

For each $Q \in \W$, define
\begin{align*}
T_Q^{\le} := \sum_{N_1 \le j \le \log_2 \ell(Q)} T_j
\quad \text{ and } \quad
T_Q^{>} := \sum_{\log_2 \ell(Q) < j \le N_2} T_j.
\end{align*}
The assumption \eqref{H1} gives
\begin{align*}
T_j(h_1, \ldots, h_i \mathbf{1}_{(3Q)^c}, \ldots, h_m)(x) = 0, \quad \forall x \in Q, \, j \le \log_2 \ell(Q),
\end{align*}
which implies
\begin{align}\label{eq:TQQ}
T_Q^{\le}(\vec{h})(x) = T_Q^{\le}(\vec{h} \mathbf{1}_{3Q})(x), \quad \forall x \in Q.
\end{align}
Denote
\begin{align*}
F_Q := \Big\{x \in Q: \max \big\{|T(\vec{f})(x)|, \, |T_Q^{\le}(\vec{f})(x)| \big\} 
\le K_1 A_2 \prod_{i=1}^m \langle f_i \rangle_{3Q_0, r_i} \Big\}.
\end{align*}
Then we invoke \eqref{H2} and \eqref{eq:TQQ} to obtain $|F_Q| \ge \frac12 |Q|$ provided $K_1 > 1$ sufficiently large. Set $\phi_Q := \frac{|Q|}{|F_Q|} \mathbf{1}_{F_Q}$, which satisfies $\supp \phi_Q \subset Q$, $0 \le \phi_Q \le 2$, and $\langle \phi_Q \rangle_Q =1$. We now modify functions $g_{m+1}$ and $b_{m+1}$ such that
\begin{equation}\label{eq:CZm}
f_{m+1} = \widetilde{g}_{m+1} + \widetilde{b}_{m+1},
\end{equation}
where
\begin{align*}
\widetilde{g}_{m+1} & := f_{m+1} \mathbf{1}_{\Omega^c} + \sum_{Q \in \W_0} [f_{m+1}]_Q \, \phi_Q,
\\
\widetilde{b}_{m+1} & := \sum_{Q \in \W_0} \widetilde{b}_{m+1, Q}
:= \sum_{Q \in \W_0} (f_{m+1} \mathbf{1}_Q - [f_{m+1}]_Q \, \phi_Q).
\end{align*}
Set $\widetilde{b}_{m+1, Q} = 0$ if $Q \in \W \setminus \W_0$. Then by construction of $\Omega$ and $F_Q$, 
\begin{align}
\label{eq:TA21}
|T(\vec{f})(x)| 
& \lesssim A_2 \prod_{i=1}^m \langle f_i \rangle_{3Q_0, r_i}, \quad \text{if } \widetilde{g}_{m+1}(x) \ne 0,
\\
\label{eq:TA22}
|T_Q^{\le}(\vec{f})(x)| 
& \lesssim A_2 \prod_{i=1}^m \langle f_i \rangle_{3Q_0, r_i}, \quad \forall x \in F_Q.
\end{align}
Moreover, 
\begin{align}
\label{CZm-1}
& \|\widetilde{g}_{m+1}\|_{L^{\infty}} \lesssim \langle f_{m+1} \rangle_{3Q_0, r_{m+1}}, \quad 
\supp(\widetilde{g}_{m+1}) \subset Q_0, 
\\
\label{CZm-2}
& \|\widetilde{g}_{m+1}\|_{L^{q, 1}} \lesssim |Q_0|^{\frac1q} \langle f_{m+1} \rangle_{3Q_0, r_{m+1}}, 
\quad \forall q \in (0, \infty), 
\\
\label{CZm-3}
& \supp(\widetilde{b}_{m+1}) \subset Q_0, \, \, 
\supp(\widetilde{b}_{m+1, Q}) \subset Q, \, \,
\int_{\Rd} \widetilde{b}_{m+1, Q} = 0,
\\
\label{CZm-4}
& \text{and} \quad \bigg\|\sum_{Q \subset P} \widetilde{b}_{m+1, Q} \bigg\|_{L^p} 
\lesssim |P|^{\frac1p} \langle f_{m+1} \rangle_{P, p}, \quad \forall p \in [1, \infty).
\end{align}

The Calder\'{o}n--Zygmund decomposition \eqref{eq:CZ} leads to 
\begin{align}\label{eq:ggbb}
\langle T(\vec{f}), f_{m+1} \rangle
= \sum_{\sigma \subset \{1, \ldots, m\}} \langle T(h_1^{\sigma}, \ldots, h_m^{\sigma}), f_{m+1} \rangle
=: \sum_{\sigma \subset \{1, \ldots, m\}} \mathscr{I}_{\sigma},
\end{align}
where $h_i^{\sigma} := b_i$ if $i \in \sigma$, and $h_i^{\sigma} := g_i$ if $i \notin \sigma$. It follows from \eqref{H3} and \eqref{eq:ggss} that
\begin{align}\label{eq:I00}
|\mathscr{I}_{\varnothing}|
\le A_3 \prod_{i=1}^m \|g_i\|_{L^{s_i, 1}} \|f_{m+1}\|_{L^{s'}}
\lesssim A_3 |Q_0| \prod_{i=1}^{m+1} \langle f_i \rangle_{3Q_0, r_i}.
\end{align}
Given $\sigma \ne \varnothing$, we use \eqref{eq:CZm} to split
\begin{align}\label{eq:bg}
\mathscr{I}_{\sigma}
= \langle T(h_1^{\sigma}, \ldots, h_m^{\sigma}), \widetilde{g}_{m+1} \rangle
+ \langle T(h_1^{\sigma}, \ldots, h_m^{\sigma}), \widetilde{b}_{m+1} \rangle
=: \mathscr{I}_{\sigma}^1 + \mathscr{I}_{\sigma}^2.
\end{align}
By \eqref{eq:TA21}, \eqref{H3}, \eqref{eq:ggss}, and \eqref{CZm-1}, there holds
\begin{align}\label{eq:bg1}
\bigg|\sum_{\sigma \ne \varnothing} \mathscr{I}_{\sigma}^1 \bigg|
& \le |\langle T(\vec{f}), \widetilde{g}_{m+1} \rangle|
+ |\langle T(\vec{g}), \widetilde{g}_{m+1} \rangle|
\\ \notag
& \lesssim A_2 \prod_{i=1}^m \langle f_i \rangle_{3Q_0, r_i} \|\widetilde{g}_{m+1}\|_{L^1}
+ A_3 \prod_{i=1}^m \|g_i\|_{L^{s_i, 1}} \|\widetilde{g}_{m+1}\|_{L^{s'}}
\\ \notag
& \lesssim (A_2 + A_3) |Q_0| \prod_{i=1}^{m+1} \langle f_i \rangle_{3Q_0, r_i}.
\end{align}
For the term $\mathscr{I}_{\sigma}^2$, we write
\begin{align}\label{eq:Is2}
\mathscr{I}_{\sigma}^2
= \mathscr{I}_{\sigma}^{2, 1} - \mathscr{I}_{\sigma}^{2, 2} + \mathscr{I}_{\sigma}^{2, 3},
\end{align}
where
\begin{align*}
\mathscr{I}_{\sigma}^{2, 1}
& := \sum_{Q \in \W_0} \langle T_Q^{\le} (h_1^{\sigma}, \ldots, h_m^{\sigma}), f_{m+1} \mathbf{1}_Q \rangle, 
\\
\mathscr{I}_{\sigma}^{2, 2}
&:= \sum_{Q \in \W_0} \langle T_Q^{\le} (h_1^{\sigma}, \ldots, h_m^{\sigma}), [f_{m+1}]_Q \phi_Q \rangle,
\\
\mathscr{I}_{\sigma}^{2, 3} 
& := \sum_{Q \in \W_0} \langle T_Q^{>} (h_1^{\sigma}, \ldots, h_m^{\sigma}), \widetilde{b}_{m+1, Q} \rangle.
\end{align*}
In light of \eqref{eq:TQQ}, \eqref{eq:TA22}, \eqref{H3}, \eqref{eq:ffqp}, and \eqref{eq:ggqp}, one has
\begin{align}\label{eq:bg22}
\bigg| \sum_{\sigma \ne \varnothing} \mathscr{I}_{\sigma}^{2, 2} \bigg|
& \le \sum_{Q \in \W_0} \big[ |\langle T_Q^{\le} (\vec{f}), \phi_Q \rangle|
+ |\langle T_Q^{\le} (\vec{g} \mathbf{1}_{3Q}), \phi_Q \rangle| \big] \langle f_{m+1} \rangle_Q 
\\ \nonumber
& \lesssim \sum_{Q \in \W_0} \bigg[A_2 \prod_{i=1}^m \langle f_i \rangle_{3Q_0, r_i} \|\phi_Q\|_{L^1}
+ A_3 \prod_{i=1}^m \|g_i \mathbf{1}_{3Q}\|_{L^{s_i, 1}} 
\|\phi_Q\|_{L^{s'}} \bigg] \langle f_{m+1} \rangle_Q  
\\ \nonumber
& \lesssim (A_2 + A_3) \sum_{Q \in \W_0} |Q| \prod_{i=1}^{m+1} \langle f_i \rangle_{3Q_0, r_i}
\lesssim (A_2 + A_3) |Q_0| \prod_{i=1}^{m+1} \langle f_i \rangle_{3Q_0, r_i}.
\end{align}
Observe that
\begin{align}\label{eq:ff21}
\sum_{\sigma \ne \varnothing} \mathscr{I}_{\sigma}^{2, 1}
= \sum_{Q \in \W_0} \langle T_Q^{\le} (\vec{f}), f_{m+1} \mathbf{1}_Q \rangle 
- \sum_{Q \in \W_0} \langle T_Q^{\le} (\vec{g}), f_{m+1} \mathbf{1}_Q \rangle.
\end{align}
It follows from \eqref{eq:ffqp}, \eqref{eq:ggqp} and \eqref{eq:TQQ} that
\begin{align}\label{eq:ggf3}
& \bigg|\sum_{Q \in \W_0} \langle T_Q^{\le} (\vec{g}), f_{m+1} \mathbf{1}_Q \rangle \bigg|
\le A_3 \sum_{Q \in \W_0} \prod_{i=1}^m \|g_i \mathbf{1}_{3Q}\|_{L^{s_i, 1}} \|f_{m+1} \mathbf{1}_Q\|_{L^{s'}}
\\ \nonumber
& \lesssim A_3 \sum_{Q \in \W_0} |Q| \prod_{i=1}^m \|g_i\|_{L^{\infty}} \langle f_{m+1} \rangle_{Q, r_{m+1}}
\lesssim A_3 |Q_0| \prod_{i=1}^{m+1} \langle f_i \rangle_{3Q_0, r_i}.
\end{align}
On the other hand, the property \eqref{eq:QQin} gives $\log_2 \ell(Q) - N_1 \le N_2 -1 - N_1 \le n-1$. Then by \eqref{eq:TQQ} and the induction hypothesis, 
\begin{align}\label{eq:Un1}
|\langle T_Q^{\le} (\vec{f}), f_{m+1} \mathbf{1}_Q \rangle|
= |\langle T_Q^{\le} (\vec{f} \mathbf{1}_{3Q}), f_{m+1} \mathbf{1}_Q \rangle|
\le U_{n-1} \mathbf{S}_Q^3 (\vec{f}, f_{m+1}).
\end{align}
For the remaining terms, we claim that
\begin{align}\label{bad23}
|\mathscr{I}_{\sigma}^{2, 3}|
\lesssim \big(j_0 A_4 + 2^{-\vartheta j_0/2} A_5 \big)
|Q_0| \prod_{i=1}^{m+1} \langle f_i \rangle_{3Q_0, r_i}, \quad \forall \sigma \ne \varnothing,
\end{align} 
where $j_0 := \lfloor 2 \vartheta^{-1} \log_2(2 + A_4^{-1} A_5) \rfloor + 1 \ge 3$.

Assuming \eqref{bad23} holds momentarily, let us conclude the proof of Proposition \ref{prop:Un} as follows. The maximal nature of $\mathbf{S}_Q^3$ prevents us from directly obtaining the desired sparse domination for $\sum_{Q \in \W_0} \langle T_Q^{\le} (\vec{f}), f_{m+1} \mathbf{1}_Q \rangle$ from \eqref{eq:Un1}. To overcome this issue, we use a limiting argument. Given $\varepsilon > 0$ and $Q \in \W_0$, by \eqref{eq:Un1} and the definition of $\mathbf{S}_Q^3$, there exists a $\gamma$-sparse family $\S_Q^{\varepsilon} \subset \D(Q)$ so that
\begin{align*}
|\langle T_Q^{\le} (\vec{f}), f_{m+1} \mathbf{1}_Q \rangle|
\le (U_{n-1} + \varepsilon) \sum_{Q' \in \S_Q^{\varepsilon}} 
|Q'| \prod_{i=1}^{m+1} \langle f_i \rangle_{3Q', r_i},
\end{align*}
which along with \eqref{eq:ff21} and \eqref{eq:ggf3} gives 
\begin{align}\label{eq:21}
\bigg| \sum_{\sigma \ne \varnothing} \mathscr{I}_{\sigma}^{2, 1} \bigg|
& \le (U_{n-1} + \varepsilon) \sum_{Q \in \W} \sum_{Q' \in \S_Q^{\varepsilon}} 
|Q'| \prod_{i=1}^3 \langle f_i \rangle_{3Q', r_i} 
+ C_0 A_3 |Q_0| \prod_{i=1}^{m+1} \langle f_i \rangle_{3Q_0, r_i}.
\end{align}
Collecting \eqref{eq:ggbb}--\eqref{eq:bg22} and \eqref{bad23}--\eqref{eq:21}, we arrive at
\begin{align}\label{eq:Un1e}
|\langle T(\vec{f}), f_{m+1} \rangle|
& \le |\mathscr{I}_{\varnothing}| 
+ \bigg| \sum_{\sigma \ne \varnothing} \mathscr{I}_{\sigma}^1 \bigg|
+ \sum_{k=1}^3 \bigg| \sum_{\sigma \ne \varnothing} \mathscr{I}_{\sigma}^{2, k} \bigg|
\\ \nonumber
& \le (U_{n-1} + \varepsilon) \sum_{Q \in \W} \sum_{Q' \in \S_Q^{\varepsilon}} 
|Q'| \prod_{i=1}^{m+1} \langle f_i \rangle_{3Q', r_i} 
\\ \nonumber
& \quad + C_0 \big(A_2 + A_3 + j_0 A_4 + 2^{-\vartheta j_0/2} A_5 \big) 
|Q_0| \prod_{i=1}^{m+1} \langle f_i \rangle_{3Q_0, r_i},
\end{align}
where the constant $C_0$ is independent of $\varepsilon$, $j_0$, and $Q_0$. To unify two sparse forms in \eqref{eq:Un1e}, denote 
\begin{align*}
\S^{\varepsilon} := \{Q_0\} \cup \bigcup_{Q \in \D(Q_0) \cap \W} \S_Q^{\varepsilon} \subset \D(Q_0).
\end{align*}
Recall that $E_{Q_0} = Q_0 \setminus \Omega$ and $|E_{Q_0}| \ge \gamma |Q_0|$. Moreover, for each $Q \in \D(Q_0) \cap \W$ and $Q' \in \S_Q^{\varepsilon}$, by the $\gamma$-sparseness of $\S_Q^{\varepsilon}$, there exists a subset $E_{Q'} \subset Q'$ such that $|E_{Q'}| \ge \gamma |Q'|$ and the family $\{E_{Q'}: Q' \in \S_Q^{\varepsilon}\}$ is pairwise disjoint. By the fact $\Omega = \bigcup_{Q \in \W} Q$ and the disjointness of $\W$, we assert 
\begin{align*}
\text{the collection $\{E_{Q_0}\} \cup \{E_{Q'}: Q' \in \S_Q^{\varepsilon}, Q \in \D(Q_0) \cap \W\}$ is pairwise disjoint}.
\end{align*}
This means 
\begin{align}\label{eq:SS}
\text{$\S^{\varepsilon}$ is a $\gamma$-sparse collection}. 
\end{align}
Note that the choice of $j_0$ indicates
\begin{align}\label{eq:AAjj}
A_2 + A_3 + j_0 A_4 + 2^{-\vartheta j_0/2} A_5
\lesssim_{\vartheta} A_2 + A_3 + A_4 \log_2(2 + A_4^{-1} A_5) = A_0.
\end{align}
Consequently, \eqref{eq:Un1e}--\eqref{eq:AAjj} imply that for any $\varepsilon > 0$,
\begin{align*}
|\langle T(\vec{f}), f_{m+1} \rangle|
& \le \max \big\{U_{n-1} + \varepsilon, \, C_0 A_0 \big\} 
\sum_{Q \in \S^{\varepsilon}} |Q| \prod_{i=1}^{m+1} \langle f_i \rangle_{3Q, r_i} 
\\
& \le \max \big\{U_{n-1} + \varepsilon, \, C_0 A_0 \big\} \mathbf{S}_{Q_0}^3(\vec{f}, f_{m+1}).
\end{align*}
By the arbitrariness of $\varepsilon$, we conclude 
\begin{align*}
U_n \le \max \big\{U_{n-1}, \, C_0 A_0 \big\},
\end{align*}
which is exactly the assertion of Proposition \ref{prop:Un}.

It remains to demonstrate \eqref{bad23}. For simplicity, from now on, let $b_{m+1, I_{m+1}} = \widetilde{b}_{m+1, I_{m+1}}$ with the convention that $b_{m+1, I_{m+1}} = 0$ whenever $I_{m+1} \notin \W_0$. Fix $\sigma \ne \varnothing$. Expanding the bad functions, we obtain
\begin{align*}
\mathscr{I}_{\sigma}^{2, 3} 
= \sum_{N_1 \le j \le N_2} 
\sum_{\substack{I_i \in \W, i \in \sigma \\ I_{m+1} \in \W \\ \ell(I_{m+1}) < 2^j}} 
\langle T_j (u_1, \ldots, u_m), b_{m+1, I_{m+1}} \rangle, 
\end{align*}
where $u_i := b_{i, I_i}$ if $i \in \sigma$, and $u_i := g_i$ if $i \notin \sigma$. We split 
\begin{align}\label{bg23}
\mathscr{I}_{\sigma}^{2, 3} 
= \mathscr{I}_{\sigma}^{2, 3, 1} + \mathscr{I}_{\sigma}^{2, 3, 2},
\end{align}
where
\begin{align*} 
\mathscr{I}_{\sigma}^{2, 3, 1} 
& := \sum_{N_1 \le j \le N_2} 
\sum_{\substack{I_i \in \W, i \in \sigma, I_{m+1} \in \W \\ \ell(I_{m+1}) < 2^j \le \max_{i \in \sigma} \ell(I_i)}} 
\langle T_j (u_1, \ldots, u_m), b_{m+1, I_{m+1}} \rangle, 
\\
\mathscr{I}_{\sigma}^{2, 3, 2} 
& := \sum_{N_1 \le j \le N_2} 
\sum_{\substack{I_i \in \W, i \in \sigma, I_{m+1} \in \W \\ \max_{i \in \sigma} \{\ell(I_i), \ell(I_{m+1})\} < 2^j}}
\langle T_j (u_1, \ldots, u_m), b_{m+1, I_{m+1}} \rangle.
\end{align*}

\noindent{\bf $\bullet$ Estimates for $\mathscr{I}_{\sigma}^{2, 3, 1}$.} 
Let $a \in \sigma$ be such that $\ell(I_a) = \max_{i \in \sigma} \ell(I_i)$. Fix $I_i \in \W$, $i \in \sigma$, and $I_{m+1} \in \W$ with $\ell(I_{m+1}) < 2^j \le \max_{i \in \sigma} \ell(I_i)$ such that $\langle T_j (u_1, \ldots, u_m), b_{m+1, I_{m+1}} \rangle \ne 0$. We claim that 
\begin{align}\label{eq:JPQ}
I_{m+1} \subset 3I_a \quad \text{ and } \quad 
\ell(I_a) \le 8 \ell(I_{m+1}).
\end{align}
Indeed, by the assumption \eqref{H1}, $\supp T_j (u_1, \ldots, u_m) \subset 3I_a$. Then $I_{m+1} \cap 3I_a \ne \varnothing$, which along with the fact $\ell(I_{m+1}) \le \ell(I_a)/2$ gives $I_{m+1} \subset 3I_a$. Additionally, by \eqref{eq:Whd},
\begin{align*}
4 \diam(I_a) 
\le \dist(I_a, \Omega^c) 
& \le \dist(I_{m+1}, \Omega^c) + \diam(I_{m+1}) + \diam(3I_a)
\\
& \le 10 \diam(I_{m+1}) + \diam(I_{m+1}) + 3 \diam(I_a),
\end{align*}
which indicates $\ell(I_a) \le 11 \ell(I_{m+1})$. Hence, $\ell(I_a) \le 8 \ell(I_{m+1})$.

If we set 
\begin{align*}
\bm{b}_i 
:= \sum_{\substack{I \in \W \\ \ell(I) \le \ell(I_a)}} b_{i, I} \mathbf{1}_{5I_a}, \quad i \in \sigma,
\end{align*}
then the disjointness of $\W$ and \eqref{eq:ffqp} imply
\begin{align}\label{eq:bbi}
\|\bm{b}_i\|_{L^{r_i}}^{r_i}
\le \sum_{\substack{I \in \W \\ I \subset 5I_a}} \|b_{i, I}\|_{L^{r_i}}^{r_i}
\lesssim |I_a| \langle f_i \rangle_{5I_a, r_i}^{r_i}
\lesssim |I_a| \langle f_i \rangle_{3Q_0, r_i}^{r_i}.
\end{align}
Thus, it follows from \eqref{H1}, \eqref{eq:JPQ}, \eqref{eq:bbi}, \eqref{H4}, \eqref{eq:ffqp}, and \eqref{eq:ggqp} that
\begin{align}\label{bg231}
|\mathscr{I}_{\sigma}^{2, 3, 1}|
& = \bigg|\sum_{N_1 \le j \le N_2} 
\sum_{\substack{I_i \in \W, i \in \sigma, I_{m+1} \subset 3I_a \\ \ell(I_{m+1}) < 2^j \le \ell(I_a) \le 8 \ell(I_{m+1})}} 
\langle T_j (\vec{u} \mathbf{1}_{5I_a}), b_{m+1, I_{m+1}} \rangle \bigg|
\\ \nonumber
& \le A_4 \sum_{N_1 \le j \le N_2} 2^{-jd(\frac1r - \frac1s)}
\sum_{\substack{I_a, I_{m+1} \in \W, I_{m+1} \subset 3I_a \\ \ell(I_{m+1}) < 2^j \le \ell(I_a) \le 8 \ell(I_{m+1})}} 
\|b_{a, I_a}\|_{L^{r_a}} 
\\ \nonumber
&\qquad \times \prod_{i \in \sigma \setminus \{a\}} \|\bm{b}_i\|_{L^{r_i}} 
\prod_{i \notin \sigma} \|g_i \mathbf{1}_{5I_a}\|_{L^{r_i}} 
\|b_{m+1, I_{m+1}}\|_{L^{s'}}
\\ \nonumber
& \lesssim A_4 \sum_{N_1 \le j \le N_2} 
\sum_{\substack{I_a \in \W \\ 2^j \le \ell(I_a) \le 2^{j+2}}} 
\sum_{\substack{I_{m+1} \in \W, I_{m+1} \subset 3I_a \\ \ell(I_a)/8 \le \ell(I_{m+1}) \le \ell(I_a)}} |I_a| 
\prod_{i=1}^{m+1} \langle f_i \rangle_{3Q_0, r_i}
\\ \nonumber
& \lesssim A_4 \sum_{I_a \in \W} |I_a| \prod_{i=1}^{m+1} \langle f_i \rangle_{3Q_0, r_i}
\lesssim A_4 |Q_0| \prod_{i=1}^{m+1} \langle f_i \rangle_{3Q_0, r_i}.
\end{align}

\noindent{\bf $\bullet$ Estimates for $\mathscr{I}_{\sigma}^{2, 3, 2}$.}
Define $i_0 := \min \sigma$ and $\sigma^* := \sigma \cup \{m+1\}$. Let $I_{i_0} \subset Q \in \D_j := \{Q \in \D(Q_0): \ell(Q) = 2^j\}$ and $\max_{i \in \sigma^*} \ell(I_i) < 2^j$ such that $\langle T_j (u_1, \ldots, u_m), b_{m+1, I_{m+1}} \rangle \ne 0$. Then
\begin{align}\label{5Qaa-1}
I_{m+1} \subset 3Q \quad \text{ and } \quad 
I_i \subset 5Q, \quad \forall i \in \sigma \setminus \{i_0\}.
\end{align}
Indeed, by \eqref{H1}, $\supp T_j (u_1, \ldots, u_m) \subset \{x \in \Rd: \dist(x, I_{i_0}) \le 2^j\} \subset 3Q$, which indicates $I_{m+1} \cap 3Q \ne \varnothing$. Then the fact $\ell(I_{m+1}) \le \ell(Q)/2$ implies $I_{m+1} \subset 3Q$. Moreover, there exists some $x \in \supp T_j (u_1, \ldots, u_m) \cap I_{m+1}$, which gives $x \in 3Q$ and $ \dist(x, I_i) \le 2^j = \ell(Q)$ for each $i \in \sigma$. Since $\{y \in \Rd: \dist(y, 3Q) \le \ell(Q)\} \subset 5Q$, we have $I_i \cap 5Q \ne \varnothing$, which together with $\ell(I_i) \le \ell(Q)$ leads to $I_i \subset 5Q$ for any $i \in \sigma$.

Note that
\begin{align}\label{5Qaa-2}
T_j(u_1, \ldots, u_{i-1}, g_i 1_{(5Q)^c}, u_{i+1}, \ldots, u_m)(x) = 0, \quad\forall x \in 3Q, \, i \notin \sigma.
\end{align}
For convenience, define 
\begin{align*}
B_{Q, k}^i
:=\sum_{\substack{I \in \W, I \subset P_i(Q) \\ \ell(I) = 2^{j-k}}} b_{i, I}
\quad \, \text{ and } \, \quad
B_{Q, \leq L}^i
:= \sum_{k=1}^L B_{Q, k}^i,
\end{align*}
where $P_i(Q) := Q$ if $i=i_0$, and $P_i(Q) := 5Q$ if $i \in \sigma^* \setminus \{i_0\}$. Then \eqref{5Qaa-1} and \eqref{5Qaa-2} allow us to rewrite 
\begin{align*}
\mathscr{I}_{\sigma}^{2,3,2}
= \sum_{j=N_1}^{N_2}\sum_{Q \in \D_j} \sum_{\bm{k} = (k_i)_{i \in \sigma^*}: k_i \ge 1}
\big\langle T_j \big(U_{Q, \bm{k}}^1, \ldots, U_{Q, \bm{k}}^m \bigr), B_{Q, k_{m+1}}^{m+1} \big\rangle,
\end{align*}
where $U_{Q, \bm{k}}^i := B_{Q, k_i}^i$ if $i \in \sigma$, and $U_{Q, \bm{k}}^i := g_i \mathbf{1}_{5Q}$ if $i \notin \sigma$. We further split
\begin{equation}\label{bg232}
\mathscr{I}_{\sigma}^{2, 3, 2}
= \mathscr{I}_{\sigma, \mathrm{near}}^{2, 3, 2} 
+ \mathscr{I}_{\sigma, \mathrm{far}}^{2, 3, 2},
\end{equation}
where
\begin{align*}
\mathscr{I}_{\sigma, \mathrm{near}}^{2,3,2}
& := \sum_{j=N_1}^{N_2}\sum_{Q \in \D_j} 
\sum_{\substack{1 \leq k_i \leq  j_0 \\ i \in \sigma^*}}
\big\langle T_j \big(U_{Q, \bm{k}}^1, \ldots, U_{Q, \bm{k}}^m \bigr), B_{Q, k_{m+1}}^{m+1} \big\rangle,
\\
\mathscr{I}_{\sigma, \mathrm{far}}^{2,3,2}
& := \sum_{j=N_1}^{N_2}\sum_{Q \in \D_j} 
\sum_{\max\limits_{i \in \sigma^*} k_i >  j_0}
\big\langle T_j \big(U_{Q, \bm{k}}^1, \ldots, U_{Q, \bm{k}}^m \bigr), B_{Q, k_{m+1}}^{m+1} \big\rangle.
\end{align*}
Observe that
\begin{align}
\label{BQki-1}
\|B_{Q, k}^i\|_{L^{r_i}}
& \lesssim (\beta_{Q, k})^{\frac{1}{r_i}} \langle f_i \rangle_{3Q_0, r_i},
\\
\label{BQki-2}
\|B_{Q, \leq L}^i\|_{L^{r_i}}
& \lesssim (\beta_{Q, \leq L})^{\frac{1}{r_i}} \langle f_i \rangle_{3Q_0, r_i},
\end{align}
where
\begin{align*}
\beta_{Q, k} := \sum_{\substack{I \in \W, I \subset 5Q \\ \ell(I) = 2^{j-k}}} |I|
\quad \, \text{ and } \, \quad
\beta_{Q, \leq L} := \sum_{k=1}^{L} \beta_{Q, k}
\lesssim |Q|.
\end{align*}
For example, by \eqref{eq:ffqp},
\begin{align*}
\|B_{Q, \leq L}^i\|_{L^{r_i}}^{r_i}
= \sum_{\substack{I \in\W, I \subset P_i(Q) \\ 2^{j-L}\leq \ell(I) < 2^j}}
\|b_{i, I}\|_{L^{r_i}}^{r_i} 
\lesssim \sum_{\substack{I \in \W, I \subset P_i(Q) \\ 2^{j-L} \leq \ell(I) < 2^j}}
|I| \langle f_i \rangle_{I, r_i}^{r_i}
\lesssim \beta_{Q, \leq L} \, \langle f_i \rangle_{3Q_0, r_i}^{r_i}.
\end{align*}

To bound $\mathscr{I}_{\sigma, \mathrm{near}}^{2,3,2}$, assume first that $\sum_{i \in \sigma} \frac{1}{r_i} \ge \frac1s$. By improving estimates in \eqref{H4}, \eqref{BQki-2}, and \eqref{eq:ggqp}, we obtain
\begin{align}\label{Near-1}
& \bigg| \sum_{\substack{1 \leq k_i \leq  j_0 \\ i \in \sigma^*}}
\big\langle T_j \big(U_{Q, \bm{k}}^1, \ldots, U_{Q, \bm{k}}^m \bigr), 
B_{Q, k_{m+1}}^{m+1} \big\rangle \bigg|
\\ \nonumber
& \leq A_4 \, 2^{-jd(\frac1r - \frac1s)} 
\prod_{i \in \sigma} \|B_{Q, \leq j_0}^i\|_{L^{r_i}}
\prod_{i \notin \sigma} \|g_i \mathbf{1}_{5Q}\|_{L^{r_i}}
\|B_{Q, \leq j_0}^{m+1}\|_{L^{s'}}
\\ \nonumber
&\lesssim A_4 \, |Q| \prod_{i \in \sigma^*}
\bigg(\frac{\beta_{Q, \leq j_0}}{|Q|}\bigg)^{\frac{1}{r_i}}
\prod_{i=1}^{m+1} \langle f_i \rangle_{3Q_0, r_i}
\\ \nonumber
& \lesssim A_4 \, \beta_{Q, \leq j_0}
\prod_{i=1}^{m+1} \langle f_i \rangle_{3Q_0, r_i},
\end{align}
provided
\begin{align}\label{eq:rsrs}
- \bigg(\frac1r - \frac1s \bigg) + \sum_{i \in \sigma^*} \frac{1}{r_i} 
+ \sum_{i \notin \sigma} \frac{1}{r_i} = 1
\quad \text{ and } \quad
\sum_{i \in \sigma^*} \frac{1}{r_i}
\geq \frac{1}{s} + \frac{1}{s'} 
\geq 1.
\end{align}
Suppose now that $\frac{1}{r_{\sigma}} := \sum_{i \in \sigma} \frac{1}{r_i} < \frac1s$. Then $1 < r'_{\sigma} < s' = r_{m+1}$ and
\begin{align}\label{eq:bqj}
\|B_{Q, \leq L}^{m+1}\|_{L^{r'_{\sigma}}}
\lesssim (\beta_{Q, \leq L})^{\frac{1}{r'_{\sigma}}} 
\langle f_{m+1} \rangle_{3Q_0, r_{m+1}},
\end{align}
which, along with auxiliary estimates in \eqref{H4} applied to $I = \{1, \ldots, m\} \setminus \sigma$, \eqref{BQki-2}, and \eqref{eq:ggqp}, gives
\begin{align}\label{Near-2}
& \bigg| \sum_{\substack{1 \leq k_i \leq  j_0 \\ i \in \sigma^*}}
\big\langle T_j \big(U_{Q, \bm{k}}^1, \ldots, U_{Q, \bm{k}}^m \bigr), 
B_{Q, k_{m+1}}^{m+1} \big\rangle \bigg|
\\ \notag
& \leq A_4 \, 2^{-jd \sum_{i \notin \sigma} \frac{1}{r_i^*}} 
\prod_{i \in \sigma} \|B_{Q, \leq j_0}^i\|_{L^{r_i}}
\prod_{i \notin \sigma} \|g_i \mathbf{1}_{5Q}\|_{L^{r_i^*}}
\|B_{Q, \leq j_0}^{m+1}\|_{L^{r'_{\sigma}}}
\\ \notag
& \lesssim A_4 \, |Q|^{- \sum_{i \notin \sigma} \frac{1}{r_i^*}}
\prod_{i \in \sigma} (\beta_{Q, \leq j_0})^{\frac{1}{r_i}}
\prod_{i \notin \sigma} |Q|^{\frac{1}{r_i^*}}
(\beta_{Q, \leq j_0})^{\frac{1}{r'_{\sigma}}}
\prod_{i=1}^{m+1} \langle f_i \rangle_{3Q_0, r_i}
\\ \notag
& = A_4 \, \beta_{Q, \leq j_0}
\prod_{i=1}^{m+1} \langle f_i \rangle_{3Q_0, r_i}.
\end{align}
Besides, for any $L \ge 1$, 
\begin{align}\label{uQL}
\sum_{j=N_1}^{N_2} \sum_{Q \in \D_j} \beta_{Q, \leq L}
= \sum_{I \in \W} |I| \sum_{\substack{N_1 \leq j \leq N_2 \\ 1 \leq j - \log_2 \ell(I) \leq L}}
\sum_{\substack{Q \in \D_j \\ I \subset 5Q}} 1
\lesssim L \sum_{I \in \W} |I|
\lesssim L |Q_0|.                            
\end{align}
Consequently, equations \eqref{Near-1}--\eqref{uQL} prove
\begin{align}\label{eq:2321}
|\mathscr{I}_{\sigma, \mathrm{near}}^{2,3,2}|
\lesssim j_0 A_4 |Q_0| \prod_{i=1}^{m+1} \langle f_i \rangle_{3Q_0, r_i}.
\end{align}

Next let us analyze $\mathscr{I}_{\sigma, \mathrm{far}}^{2,3,2}$. For each vector $\bm{k} = (k_i)_{i \in \sigma^*}$ with $\max_{i \in \sigma^*} k_i >  j_0$, there exists a unique pair $(a, K)$ satisfying 
\begin{align*}
K = \max_{i \in \sigma^*} k_i > j_0
\quad \text{ and } \quad 
a = \min\{ i \in \sigma^* : k_i = K\}.
\end{align*}
Given $a \in \sigma^*$ and $K > j_0$, define a set
\begin{equation*}
\Xi_{K}^a :=
\left\{\bm{k} = (k_i)_{i \in \sigma^*}:
\begin{array}{ll}
k_a = K, &
\\
k_i \leq K-1, & i < a, 
\\
k_i \leq K, & i > a
\end{array}
\right\},
\end{equation*}
which gives
\begin{equation*}
\big\{\bm{k} = (k_i)_{i \in \sigma^*}: \max_{i \in \sigma^*} k_i >  j_0 \big\}
= \bigsqcup_{K > j_0} \bigsqcup_{a \in \sigma^*} \Xi_K^a,
\end{equation*}
where $\bigsqcup$ denotes a disjoint union. Thus,
\begin{equation*}
\mathscr{I}_{\sigma, \mathrm{far}}^{2,3,2}
= \sum_{j=N_1}^{N_2} \sum_{Q \in \D_j}
\sum_{K>j_0} \sum_{a \in \sigma^*}
\sum_{\bm{k} \in \Xi_K^a}
\big\langle T_j \big(U_{Q, \bm{k}}^1, \ldots, U_{Q, \bm{k}}^m \bigr), B_{Q, k_{m+1}}^{m+1} \big\rangle.
\end{equation*}
Given $(a, K)$, define functions
\begin{equation*}
V_{Q, a, K}^i :=
\begin{cases}
B_{Q, K}^i, & i = a \leq m,
\\
B_{Q, \leq K-1}^i, & i \in \sigma, i<a,
\\
B_{Q, \leq K}^i, & i \in \sigma,\ i>a,
\\
g_i \mathbf{1}_{5Q}, & i \notin \sigma,
\end{cases}
\quad \text{ and } \quad
W_{Q, a, K}^{m+1} :=
\begin{cases}
B_{Q, K}^{m+1}, & a = m+1, 
\\
B_{Q, \leq K}^{m+1}, & a \leq m.
\end{cases}
\end{equation*}
Set $B_{Q, \leq 0}^i = 0$. Accordingly,
\begin{equation*}
\sum_{\bm{k} \in \Xi_K^a}
\big\langle T_j \big(U_{Q, \bm{k}}^1, \ldots, U_{Q, \bm{k}}^m \bigr), B_{Q, k_{m+1}}^{m+1} \big\rangle
= \big\langle T_j(V_{Q,a,K}^1, \ldots, V_{Q,a,K}^m), W_{Q,a,K}^{m+1} \big\rangle,
\end{equation*}
which implies 
\begin{align}\label{Far-0}
\mathscr{I}_{\sigma, \mathrm{far}}^{2,3,2}
= \sum_{j=N_1}^{N_2} \sum_{Q \in \D_j}
\sum_{K>j_0} \sum_{a \in \sigma^*}
\big\langle T_j(V_{Q,a,K}^1, \ldots, V_{Q,a,K}^m), W_{Q,a,K}^{m+1} \big\rangle.
\end{align}
Given $I \in \D$ with $\ell(I) = 2^{j-K}$, the cancellation of $b_{a, I}$ indicates $[b_{a, I}]_{Q'} = 0$ for every $Q' \in \D$ with $\ell(Q') = 2^{j-K}$, which gives 
\begin{align}\label{EKBK}
\mathbb{E}_{K-j} B_{Q, K}^a = 0
\quad \text{ and } \quad
B_{Q, K}^a = (\Id - \mathbb{E}_{K-j}) B_{Q, K}^a.
\end{align}
In addition, by \eqref{BQki-1} and \eqref{BQki-2}
\begin{align}\label{VQaK}
\|V_{Q, a, K}^i\|_{L^{r_i}} 
\lesssim (\beta_{Q, \le K})^{\frac{1}{r_i}} \langle f_i \rangle_{3Q_0, r_i}, \quad \forall i \in \sigma.
\end{align}
First treat the case $a \leq m$. If $\sum_{i \in \sigma} \frac{1}{r_i} \ge \frac1s$, we invoke Lemma \ref{lem:reg}, \eqref{eq:ggqp}, \eqref{BQki-2}, \eqref{eq:rsrs} and \eqref{VQaK} to deduce
\begin{align}\label{Far-1}
& \big| \big\langle T_j(V_{Q,a,K}^1, \ldots, V_{Q,a,K}^m), W_{Q,a,K}^{m+1} \big\rangle \big|
\\ \nonumber
& = \big| \big\langle T_j(V_{Q,a,K}^1, \ldots, (\Id - \mathbb{E}_{K-j}) B_{Q, K}^a, \ldots, V_{Q,a,K}^m), 
B_{Q, \le K}^{m+1} \big\rangle \big|
\\ \nonumber
& \lesssim A_5 \, 2^{-\vartheta K} \, 2^{-jd(\frac1r - \frac1s)} 
\prod_{i \in \sigma} \|V_{Q, a, K}^i\|_{L^{r_i}}
\prod_{i \notin \sigma} \|g_i \mathbf{1}_{5Q}\|_{L^{r_i}}
\|B_{Q, \leq K}^{m+1}\|_{L^{s'}}
\\ \nonumber
&\lesssim A_5 \, 2^{-\vartheta K} \, |Q| \prod_{i \in \sigma^*}
\bigg(\frac{\beta_{Q, \leq K}}{|Q|}\bigg)^{\frac{1}{r_i}}
\prod_{i=1}^{m+1} \langle f_i \rangle_{3Q_0, r_i}
\\ \nonumber
& \lesssim A_5 \, 2^{-\vartheta K} \, \beta_{Q, \leq K}
\prod_{i=1}^{m+1} \langle f_i \rangle_{3Q_0, r_i}.
\end{align}
Likewise, if $\frac{1}{r_{\sigma}} := \sum_{i \in \sigma} \frac{1}{r_i} < \frac1s$ (equivalently, $r'_{\sigma} < s' = r_{m+1}$), Lemma \ref{lem:reg}, \eqref{eq:ggqp}, \eqref{eq:bqj}, and \eqref{VQaK} imply
\begin{align}\label{Far-2}
& \big| \big\langle T_j(V_{Q,a,K}^1, \ldots, V_{Q,a,K}^m), W_{Q,a,K}^{m+1} \big\rangle \big|
\\ \nonumber
& = \big| \big\langle T_j(V_{Q,a,K}^1, \ldots, (\Id - \mathbb{E}_{K-j}) B_{Q, K}^a, \ldots, V_{Q,a,K}^m), 
B_{Q, \le K}^{m+1} \big\rangle \big|
\\ \notag
& \leq A_5 \, 2^{-\vartheta K} \, 2^{-jd \sum_{i \notin \sigma} \frac{1}{r_i^*}} 
\prod_{i \in \sigma} \|V_{Q,a,K}^i\|_{L^{r_i}}
\prod_{i \notin \sigma} \|g_i \mathbf{1}_{5Q}\|_{L^{r_i^*}}
\|B_{Q, \leq K}^{m+1}\|_{L^{r'_{\sigma}}}
\\ \notag
& \lesssim A_5 \, 2^{-\vartheta K} \, |Q|^{- \sum_{i \notin \sigma} \frac{1}{r_i^*}}
\prod_{i \in \sigma} (\beta_{Q, \leq K})^{\frac{1}{r_i}}
\prod_{i \notin \sigma} |Q|^{\frac{1}{r_i^*}}
(\beta_{Q, \leq K})^{\frac{1}{r'_{\sigma}}}
\prod_{i=1}^{m+1} \langle f_i \rangle_{3Q_0, r_i}
\\ \notag
& = A_5 \, 2^{-\vartheta K} \, \beta_{Q, \leq K}
\prod_{i=1}^{m+1} \langle f_i \rangle_{3Q_0, r_i}.
\end{align}
In the case $a = m+1$, it follows from \eqref{EKBK} and the self-adjointness of $\mathbb{E}_{K-j}$ that
\begin{align*}
\big \langle T_j(V_{Q,a,K}^1, \ldots, V_{Q,a,K}^m), W_{Q,a,K}^{m+1} \big \rangle
= \big \langle (\Id - \mathbb{E}_{K-j}) T_j(V_{Q,a,K}^1, \ldots, V_{Q,a,K}^m), B_{Q, K}^{m+1} \big \rangle,
\end{align*}
As argued in \eqref{Far-1} and \eqref{Far-2}, there holds
\begin{align}\label{Far-3}
\big| \big\langle T_j(V_{Q,a,K}^1, \ldots, V_{Q,a,K}^m), W_{Q,a,K}^{m+1} \big\rangle \big|
\lesssim A_5 \, 2^{-\vartheta K} \, \beta_{Q, \leq K}
\prod_{i=1}^{m+1} \langle f_i \rangle_{3Q_0, r_i}.
\end{align}
As a consequence of \eqref{uQL}, \eqref{Far-0} and \eqref{Far-1}--\eqref{Far-3}, we conclude 
\begin{align*}
|\mathscr{I}_{\sigma, \mathrm{far}}^{2,3,2}|
\lesssim A_5 \sum_{K>j_0} 2^{-\vartheta K} K \, |Q_0| \prod_{i=1}^{m+1} \langle f_i \rangle_{3Q_0, r_i}
\lesssim A_5 \, 2^{-\vartheta j_0/2} \, |Q_0| \prod_{i=1}^{m+1} \langle f_i \rangle_{3Q_0, r_i},
\end{align*}
which along with \eqref{bg232} and \eqref{eq:2321} leads to
\begin{align}\label{s232}
|\mathscr{I}_{\sigma}^{2,3,2}|
\lesssim (j_0 A_4 + 2^{-\vartheta j_0/2} A_5) 
|Q_0| \prod_{i=1}^{m+1} \langle f_i \rangle_{3Q_0, r_i}.
\end{align}
Therefore, \eqref{bad23} follows from \eqref{bg23}, \eqref{bg231}, and \eqref{s232}.

\section{Applications}\label{sec:app}
This section is devoted to presenting the proof of Theorems \ref{thm:rough}--\ref{thm:BR}.

\subsection{Multilinear rough singular integrals}\label{sec:app-rough}
To demonstrate Theorem \ref{thm:rough}, we first prove the necessity, and then the sufficiency for $d \ge 2$. After that, we treat the sufficiency in the case $d=1$, where the only additional ingredient is a stronger main single-scale improving estimate in \eqref{H4} in a larger one-dimensional exponent region; all other parts of the sufficiency argument are unchanged.

\subsubsection{{\bf The necessity. }} 
First fix $I \subsetneq \{1, \ldots, m+1\}$. We may assume that $k := |I| \ge 2$ because $\sum_{i \in I} \frac{1}{r_i} + \frac{|I| - 1}{q} \le |I|$ holds trivially whenever $|I| = 1$. Take
\begin{align*}
f_i := \mathbf{1}_{F_i^{\varepsilon}} 
\quad \text{ and } \quad
F_i^{\varepsilon} :=
\begin{cases}
B(0, \varepsilon), & i \in I, 
\\
B(e_1, 1/10), & i \notin I,
\end{cases}
\end{align*}
where $e_1 = (1, 0, \ldots, 0)$ and $0 < \varepsilon < 1/10$. Since $B(0, \varepsilon) \cup B(e_1, 1/10) \subset B(0, 2)$ and $\dist \big(B(0, \varepsilon), B(e_1, 1/10) \big) \ge 1 - \varepsilon - \frac{1}{10} \ge \frac45$, Lemma~\ref{lem:sep} gives
\begin{equation}\label{eq:sp-1}
\Lambda_{\vec{r}, s'}^*(\vec{f}, f_{m+1})
\lesssim \varepsilon^{d \sum_{i \in I} \frac{1}{r_i}}.
\end{equation}

Given $(x_1, \ldots, x_{m+1}) \in F_1^{\varepsilon} \times \cdots \times F_{m+1}^{\varepsilon}$, set $w_{m+1} = 0$ and $w_i = x_{m+1} - x_i$ for each $i=1, \ldots, m$. Write $\vec{w} = (w_1, \ldots, w_m)$. Fix $i_0 \in I$ (take $i_0 = m+1$ whenever $m+1 \in I$), and define the linear map
\begin{align*}
 L_I : \R^{md} \longrightarrow \R^{d(k-1)},\qquad
 L_I \vec{w} = (w_i - w_{i_0})_{i \in I \setminus \{i_0\}}.
\end{align*}
We claim that 
\begin{align*}
 N := \rank L_I = d(k-1) < md,
\end{align*}
If $m+1\in I$, then
\begin{align*}
L_I \vec{w} = (w_i)_{i \in I \setminus\{m+1\}} 
\quad \text{and} \quad
\ker L_I = \{\vec{w}: w_i = 0 \text{ for } i \in I \setminus\{m+1\}\}.
\end{align*}
Since $L_I$ is a coordinate projection onto $k-1$ blocks, its rank is $d(k-1)$. If $m+1\notin I$, then
\begin{align*}
\ker L_I = \{\vec{w}: w_i = w_{i_0} \text{ for all } i \in I\},
\end{align*}
which means the common value of the $k$ indicated blocks is free, as are the $(m-k)$ blocks outside $I$. Thus
\begin{align*}
\dim \ker L_I & = d + d(m-k) = d(m-k+1),
\\
\rank L_I & = md - d(m-k+1) = d(k-1).
\end{align*}
Note that $w_i - w_{i_0} = x_{i_0} - x_i$ for each $i =1, \ldots, m+1$. Then for $\theta = \frac{\vec{w}}{|\vec{w}|}$,
\begin{align}\label{eq:Lth-1}
|L_I (\theta)|
= \bigg[\sum_{i \in I \setminus \{i_0\}} |x_i - x_{i_0}|^2 \bigg]^{\frac12} /|\vec{w}|
\le 2\varepsilon \sqrt{k-1}/(4/5)
\le 3 \sqrt{m}.
\end{align}
Thus, for sufficiently small $\varepsilon$, Lemma~\ref{lem:ang} yields, 
\begin{equation*}
\sigma (E_{\varepsilon})
:= \sigma (\{\theta \in \bS^{md-1}: |L_I \theta| \le 3 \sqrt{m} \, \varepsilon\})
\lesssim \varepsilon^{d(k-1)}.
\end{equation*}

Since the band $E_{\varepsilon}$ is invariant under $\theta \mapsto -\theta$, we would like to impose an additional fixed half-space condition that contains every direction arising from points $(x_1, \ldots, x_{m+1}) \in F_1^{\varepsilon} \times \cdots \times F_{m+1}^{\varepsilon}$. If $m+1 \notin I$, fix $i \in I$ and define $\lambda_I(\vec{w}) = (w_i)_1$. Then
\begin{align*}
\lambda_I(\vec{w}) = (x_{m+1})_1 - (x_i)_1
\ge 9/10 - \varepsilon \ge 4/5.
\end{align*}
If $m+1 \in I$, fix $i \notin I \subsetneq \{1, \ldots, m\}$ and define $\lambda_I(\vec{w}) = -(w_i)_1$. Then 
\begin{align*}
\lambda_I(\vec{w}) = (x_i)_1 - (x_{m+1})_1
\ge 9/10 - \varepsilon \ge 4/5.
\end{align*}
Besides, every actual direction satisfies
\begin{align}\label{eq:Lth-2}
\lambda_I(\theta)
= \frac{\lambda_I(\vec{w})}{|\vec{w}|}
\ge \frac{4/5}{2 \sqrt{m}}
> \frac{1}{3 \sqrt{m}}, \qquad \theta = \frac{\vec{w}}{|\vec{w}|}.
\end{align}
Now take
\begin{align*}
E'_{\varepsilon}
:= \big\{\theta \in \bS^{md-1} :
|L_I \theta| \le 3 \sqrt{m} \, \varepsilon, \, \,
\lambda_I(\theta) > 1/(3 \sqrt{m}) \big\}, 
\end{align*}
which satisfies $E'_{\varepsilon} \cap (-E'_{\varepsilon}) = \varnothing$ because $\lambda_I(-\theta) = -\lambda_I(\theta)$ implies $E'_\varepsilon \subset \{\lambda_I > 1/(3 \sqrt{m})\}$ and $-E'_{\varepsilon} \subset \{\lambda_I < -1/(3 \sqrt{m})\}$, Thus, Lemma \ref{lem:ang} gives a mean zero function $\Omega_{\varepsilon} = \varepsilon^{-N/q} (\mathbf{1}_{E'_{\varepsilon}} - \mathbf{1}_{-E'_{\varepsilon}})$ with $\|\Omega_{\varepsilon}\|_{L^q} \lesssim 1$. Then by the assumption and \eqref{eq:sp-1},
\begin{align}\label{eq:sp-2}
|\langle T_{\Omega_{\varepsilon}}(\vec{f}), f_{m+1} \rangle|
\lesssim \|\Omega_{\varepsilon}\|_{L^q} \, \Lambda_{\vec{r}, s'}^*(\vec{f}, f_{m+1})
\lesssim \varepsilon^{d \sum_{i \in I} \frac{1}{r_i}}.
\end{align}
In addition, \eqref{eq:Lth-1} and \eqref{eq:Lth-2} imply $\theta = \vec{w}/|\vec{w}| \in E'_{\varepsilon}$ for all $(x_1, \ldots, x_{m+1}) \in \prod_{i=1}^{m+1} F_i^{\varepsilon}$, hence, 
\begin{align}\label{eq:sp-3}
\langle T_{\Omega_{\varepsilon}}(\vec{f}), f_{m+1} \rangle
& = \int_{\prod_{i=1}^{m+1} F_i^{\varepsilon}} 
\frac{\Omega_{\varepsilon}(\vec{w}/|\vec{w}|)}{|\vec{w}|^{md}} \,dx_1 \ldots dx_{m+1}
\gtrsim \varepsilon^{-N/q} \prod_{i=1}^{m+1} |F_i^{\varepsilon}|.
\end{align}
Combining \eqref{eq:sp-2} and \eqref{eq:sp-3} yields
\begin{align*}
\varepsilon^{d[k - (k-1)/q]} \lesssim \varepsilon^{d \sum_{i \in I} \frac{1}{r_i}}.
\end{align*}
If $\sum_{i \in I} \frac{1}{r_i} > k - \frac{k-1}{q}$, then the above is impossible as $\varepsilon \downarrow 0$. Thus, there must $\sum_{i \in I} \frac{1}{r_i} \le k - \frac{k-1}{q}$.

It remains to show the full-set condition when $d \ge 2$ since the preceding argument holds for all dimensions $d \ge 1$. Given $0 < \varepsilon< 1$, set
\begin{align*}
f_i = \mathbf{1}_{F_i^{\varepsilon}} \quad \text{and} \quad
F_i^{\varepsilon} = [3i, 3i+1] \times B^{d-1}(0, \varepsilon), \quad i=1, \ldots, m+1,
\end{align*}
which gives
\begin{equation}\label{eq:tub-1}
\Lambda_{\vec{r}, s'}^*(\vec{f}, f_{m+1})
\lesssim\varepsilon^{(d-1)(\frac1r + \frac{1}{s'})}.
\end{equation}
Given $(x_1, \ldots, x_{m+1}) \in F_1^{\varepsilon} \times \cdots \times F_{m+1}^{\varepsilon}$, write $((w_i)_1, w_i') = w_i = x_{m+1} - x_i$ for each $i=1, \ldots, m$. Denote $\vec{w} = (w_1, \ldots, w_m)$. Then we see that
\begin{align*}
|w_i'| \le 2 \varepsilon
\quad \text{ and } \quad 
2 \le 3(m+1-i) - 1 \le (w_i)_1 \le 3(m+1-i)+1 \le 4m,
\end{align*}
which implies $2 \le |\vec{w}| < 4 (m+1)^2$. Define the linear map
\begin{align*}
L : \R^{md} \longrightarrow \R^{m(d-1)},\qquad
L \vec{w} = (w'_1, \ldots, w'_m), 
\end{align*}
whose rank is $N := m(d-1) < md$. Then for any $\theta = \frac{\vec{w}}{|\vec{w}|}$, 
\begin{align}\label{eq:Lth-3}
|L(\theta)|
= \frac{1}{|\vec{w}|} \bigg[\sum_{i=1}^m |w'_i|^2 \bigg]^{\frac12}
\le \sqrt{m} \, \varepsilon 
\quad \text{ and } \quad
(\theta_1)_1 = \frac{(w_1)_1}{|\vec{w}|} > \frac{1}{2(m+1)^2}.
\end{align}
The set
\begin{align*}
E'_{\varepsilon} 
:= \big\{\theta \in \bS^{md-1}: 
|L \theta| \le \sqrt{m} \, \varepsilon, \, \, (\theta_1)_1 > (m+1)^{-2}/2 \big\}
\end{align*}
satisfies $E'_{\varepsilon} \cap (-E'_{\varepsilon}) = \varnothing$. Then by Lemma~\ref{lem:ang},
$\Omega_{\varepsilon} = \varepsilon^{-N/q} (\mathbf{1}_{E'_{\varepsilon}} - \mathbf{1}_{-E'_{\varepsilon}})$
verifies $\|\Omega_{\varepsilon}\|_{L^q} \lesssim 1$ for sufficiently small $\varepsilon > 0$.
Since \eqref{eq:Lth-3} states $\theta = \vec{w}/|\vec{w}| \in E'_{\varepsilon}$ for all $(x_1, \ldots, x_{m+1}) \in \prod_{i=1}^{m+1} F_i^{\varepsilon}$, we have
\begin{align}\label{eq:tub-2}
\Lambda_{\vec{r}, s'}^*(\vec{f}, f_{m+1})
\gtrsim \langle T_{\Omega_{\varepsilon}}(\vec{f}), f_{m+1}\rangle
\gtrsim \varepsilon^{-N/q} \prod_{i=1}^{m+1} |F_i^{\varepsilon}|
\simeq \varepsilon^{(d-1)(m+1-m/q)}.
\end{align}
Letting $\varepsilon \downarrow 0$, we use $d-1>0$, \eqref{eq:tub-1}, and \eqref{eq:tub-2} to conclude 
$\frac1r + \frac{1}{s'} \le m+1 - \frac{m}{q}$. This completes the proof.

\begin{lemma}\label{lem:sep}
Let $\kappa := \sum_{i=1}^{m+1} \frac{1}{p_i} > 1$ with $1<p_1, \ldots, p_{m+1} < \infty$. Suppose that measurable sets $F_1, \ldots, F_{m+1} \subset B(0, C_0)$ and $\dist(F_i, F_j) \ge c_0 > 0$ for some $i \ne j$. Then
\begin{equation*}
\sup_{\S: \gamma\text{-sparse}} \sum_{Q \in \S} |Q|
\prod_{i=1}^{m+1} \langle \mathbf{1}_{F_i} \rangle_{Q, p_i}
\lesssim_{c_0, C_0} \prod_{i=1}^{m+1} |F_i|^{\frac{1}{p_i}}.
\end{equation*}
\end{lemma}

\begin{proof}
Fix a $\gamma$-sparse family $\S$. Let $\S^* = \{Q \in \S: Q \cap F_i \ne \varnothing, i=1, \ldots, m+1\}$. The condition $\dist(F_i, F_j) \ge c_0 > 0$ gives $\ell(Q) \ge c_0$ for any $Q \in \S^*$. Then $\S^*$ can be decomposed into $\S^*_j = \{Q \in \S^*: 2^j c_0 \le \ell(Q) < 2^{j+1} c_0\}$, $j=0, 1, \ldots$. By the $\gamma$-sparseness of $\S$, there exists a disjoint family $\{E_Q\}_{Q \in \S^*_j}$ such that $|E_Q| \ge \gamma |Q|$ for all $Q \in \S^*_j$. Thus, 
\begin{align*}
\gamma |2^j c_0|^d \, \# \S^*_j
\le \sum_{Q \in \S^*_j} |E_Q| 
= \bigg|\bigcup_{Q \in \S^*_j} E_Q \bigg| 
\le |B(0, \ell(Q) + C_0)|
\lesssim (2^{j+1}c_0 + C_0)^d,
\end{align*}
which gives $\# \S^*_j \lesssim_{c_0, C_0} 1$ and
\begin{align*}
\sum_{Q \in \S} |Q| \prod_{i=1}^{m+1} \langle \mathbf{1}_{F_i} \rangle_{Q, p_i}
\le \sum_{j \ge 0} \sum_{Q \in \S^*_j} |Q|^{1-\kappa} \prod_{i=1}^{m+1} |F_i|^{\frac{1}{p_i}}
\lesssim \sum_{j \ge 0} 2^{-j (\kappa-1)} \prod_{i=1}^{m+1} |F_i|^{\frac{1}{p_i}}
\lesssim \prod_{i=1}^{m+1} |F_i|^{\frac{1}{p_i}}.
\end{align*}
The proof is complete. 
\end{proof}

\begin{lemma}\label{lem:ang}
Let $L: \R^D \to \R^N$ be a linear map of rank $N$, where $1 \le N < D$. Then for every fixed $C_0>0$ and sufficiently small $\varepsilon > 0$,
\begin{equation*}
\sigma (E_{\varepsilon})
:= \sigma (\{\theta \in \bS^{D-1}: |L \theta|\le C_0 \varepsilon\})
\lesssim \varepsilon^N.
\end{equation*}
In particular, the function $\Omega_{\varepsilon} = \varepsilon^{-N/q} (\mathbf{1}_{E'_{\varepsilon}} - \mathbf{1}_{-E'_{\varepsilon}})$, with $E'_{\varepsilon} \subset E_{\varepsilon}$ and $E'_{\varepsilon} \cap (-E'_{\varepsilon}) = \varnothing$, satisfies $\|\Omega_{\varepsilon}\|_{L^q} \lesssim 1$ and $\int_{\bS^{D-1}} \Omega_{\varepsilon} \, d\sigma = 0$.
\end{lemma}

\begin{proof}
Write $n = \dim (\ker L) = D - N \ge 1$. The restriction of $L$ to $(\ker L)^{\perp}$ is invertible onto its range. Hence, $|L \theta|\le C_0 \varepsilon$ implies $|v| \le C_1 \varepsilon$ for the $(\ker L)^\perp$ component $v$ of $\theta$. Parameterize the sphere there by
\begin{align*}
\theta = \sqrt{1-|v|^2} \, u+v, \qquad u \in \bS^{n-1}\subset \ker L, \quad v \in (\ker L)^{\perp}.
\end{align*}
Then for any $0 < \varepsilon \le \frac{1}{2C_1}$, 
\begin{align*}
\sigma(E_{\varepsilon})
\le \int_{|v| \le C_1 \varepsilon} \int_{\bS^{n-1}}
(1-|v|^2)^{(n-2)/2} d\sigma_{\bS^{n-1}}(u) \,dv
\lesssim \varepsilon^N,
\end{align*}
which implies $\|\Omega_{\varepsilon}\|_{L^q}^q = 2 \varepsilon^{-N} \sigma(E'_{\varepsilon}) \lesssim 1$.
\end{proof}

\subsubsection{\bf{The sufficiency for $d \ge 2$}}
It suffices to prove the case $q \in (1, \infty)$ because
\begin{align*}
L^{\infty}(\bS^{md-1}) \subset \bigcap_{1<q<\infty} L^q(\bS^{md-1}) 
\quad \text{ and } \quad
\H_{m+1}^d(\infty) = \bigcup_{1<q<\infty} \H_{m+1}^d(q), \quad d \ge 1.
\end{align*}
Let $\Phi \in \mathscr{C}_c^{\infty}(\R^{md})$ be a radial function such that $\mathbf{1}_{B(0, 1)} \le \Phi \le \mathbf{1}_{B(0, 2)}$. Denote 
\begin{align*}
& \Psi(\vec{y}) := \Phi(\vec{y}) - \Phi(2 \vec{y}), \qquad
\Psi_j( \vec{y}) := \Psi(2^{-j} \vec{y}),
\\
& K_j := K_{\Omega} \Psi_j, \, \, 
\widehat{K}_j^{\ell} := \Psi_{\ell-j} \widehat{K}_j, \, \text{ and } \, \, 
K^{\ell} := \sum_{j \in \Z} K_j^{\ell},
\end{align*}
Then define
\begin{align*}
T_j^{\ell}(\vec{f})(x) 
& := \int_{\R^{md}} K_j^{\ell}(x-\vec{y}) \prod_{i=1}^m f_i(y_i) \, d\vec{y},
\\
T^{\ell}(\vec{f})(x) 
& := \int_{\R^{md}} K^{\ell}(x-\vec{y}) \prod_{i=1}^m f_i(y_i) \, d\vec{y},
\end{align*}
which allows us to decompose
\begin{align*}
T = \sum_{\ell \in \Z} T^{\ell} 
\quad \text{ and } \quad
T^{\ell} = \sum_{j \in \Z} T_j^{\ell}.
\end{align*}
It was shown in \cite[Proposition 3]{GHH} that
\begin{equation}\label{eq:GHH}
\|T^{\ell}\|_{L^{p_1} \times \cdots \times L^{p_m} \to L^p} 
\lesssim  
\begin{cases}
\|\Omega\|_{L^{\q}} \, 2^{-\delta_1 |\ell|}, & \ell < 0,
\\
\|\Omega\|_{L^{\q}} \, 2^{2n |\ell|}, & \ell \ge 0,
\end{cases}
\end{equation}
for all $\q \in (1, \infty]$ and $\frac1p = \sum_{j=1}^m \frac{1}{p_j} > 0$ with $p_1, \ldots, p_m \in (1, \infty)$, where $\delta_1 > 0$. By \cite[Claim 7]{DS}, there exists an integer $\ell_0 \ge 1$ such that
\begin{align}\label{eq:DS}
\|T^{\ell}\|_{L^{t_1} \times \cdots \times L^{t_m} \to L^t} 
\lesssim \|\Omega\|_{L^{\q}} \, 2^{- \delta_2 |\ell|}, \quad \forall \ell \ge \ell_0,
\end{align}
for all $\q \in (1, \infty]$ and $(\frac{1}{t_1}, \ldots, \frac{1}{t_m}) \in \H_m^d(\q)$, where $\delta_2 > 0$.

Define
\begin{align}\label{def:KKje}
\K_j^{\ell} := K^{\ell} \Psi_j
\quad \text{and} \quad
\T_j^{\ell}(\vec{f})(x) 
:= \int_{\R^{md}} \K_j^{\ell}(x-\vec{y}) \prod_{i=1}^m f_i(y_i) \, d\vec{y}.
\end{align}
Then we rewrite
\begin{align*}
T^{\ell} = \sum_{j \in \Z} \T_j^{\ell}.
\end{align*}
Fix $(\frac{1}{r_1}, \ldots, \frac{1}{r_m}, \frac{1}{s'}) \in \H_{m+1}^d(q)$ and $s \in (1, \infty)$ such that $\frac1s < \frac1r := \sum_{i=1}^m \frac{1}{r_i}$. Then set $s_i := (s/r) r_i > r_i$ so that $\sum_{i=1}^m \frac{1}{s_i} = \frac1s$. We would like to verify the family $\{\T_j^{\ell}\}_{j \in \Z}$ satisfies the hypotheses \eqref{H1}--\eqref{H5} with auxiliary exponents $r_1^* = \cdots = r_m^* = \infty$.

Note that $\supp \K_j^{\ell} \subset \{2^{j-1} \le |x-\vec{y}| \le 2^{j+1}\}$, which shows \eqref{H1}. For all $N_1 \le N_2$,
\begin{align*}
\sum_{j=N_1}^{N_2} \T_j^{\ell} (\vec{f})(x)
& = \int_{\R^{md}} K^{\ell}(x-\vec{y}) \Phi \big(2^{-N_2} (x-\vec{y}) \big) \prod_{i=1}^m f_i(y_i) \, d\vec{y}
\\
&\quad - \int_{\R^{md}} K^{\ell}(x-\vec{y}) \Phi \big(2^{-N_1+1} (x-\vec{y}) \big) \prod_{i=1}^m f_i(y_i) \, d\vec{y},
\end{align*}
which along with \eqref{eq:GHH}, \eqref{eq:DS}, and Lemma \ref{lem:TKfg} part \eqref{TKfg-4} gives
\begin{equation}\label{eq:RTP}
\sup_{N_1 \le N_2} \bigg\| \sum_{j=N_1}^{N_2} \T_j^{\ell} 
\bigg\|_{L^{p_1} \times \cdots \times L^{p_m} \to L^p}
\lesssim  \left.
\begin{cases}
\|\Omega\|_{L^{\q}} \, 2^{-\delta_1 |\ell|}, & \ell < 0
\\
\|\Omega\|_{L^{\q}} \, 2^{2n |\ell|}, & 0 \le \ell < \ell_0
\\
\|\Omega\|_{L^{\q}} \, 2^{- \delta_2 |\ell|}, & \ell \ge \ell_0
\end{cases} 
\right\}
\lesssim 2^{-\delta_0 |\ell|}
\end{equation}
for all $\q \in (1, \infty]$ and $(\frac{1}{p_1}, \ldots, \frac{1}{p_m}) \in \H_m^d(\q)$, where $\delta_0 = \min\{\delta_1, \delta_2\}$. 

Observe that for any $d \ge 2$,
\begin{equation}\label{Hmd2}
\begin{aligned}
\H_m^d(q) 
& = \big\{\vec{a} \in (0, 1)^m: \mathscr{N}_q(\vec{a}) < 1/q \big\}, \qquad \text{and}
\\
\H_{m+1}^d(q) 
& = \big\{(\vec{a}, a_{m+1}) \in (0, 1)^{m+1}:
\mathscr{N}_q(\vec{a}) < \min\{1/q, 1-a_{m+1}\} \big\},
\end{aligned}
\end{equation}
where $\mathscr{N}_q(\vec{a}) := \sum_{i=1}^m \big(a_i + \frac1q - 1 \big)_+$ satisfies $\mathscr{N}_q(\vec{a}) \le \mathscr{N}_q(\vec{b})$ whenever $\vec{a} \le \vec{b}$. Then by \eqref{Hmd2}, the condition $(\frac{1}{r_1}, \ldots, \frac{1}{r_m}, \frac{1}{s'}) \in \H_{m+1}^d(q)$ gives 
\begin{align*}
\mathscr{N}_q(1/\vec{s}) \le \mathscr{N}_q(1/\vec{r}) < \min\{1/q, 1/s\}.
\end{align*}
Hence, $(\frac{1}{r_1}, \ldots, \frac{1}{r_m}), (\frac{1}{s_1}, \ldots, \frac{1}{s_m}) \in \H_m^d(q)$, which together with \eqref{eq:RTP} implies both \eqref{H2} and \eqref{H3} hold for $A_2^{\ell} = A_3^{\ell} \simeq \|\Omega\|_{L^q} 2^{-\delta_0 |\ell|}$. 

To justify the main improving estimates in \eqref{H4}, it suffices to show that there exists some $\varepsilon \in (0, \delta_0)$ such that 
\begin{align}\label{eq:STE}
\|\T_j^{\ell}\|_{L^{t_1} \times \cdots \times L^{t_m} \to L^{t_0}}
\lesssim \|\Omega\|_{L^q} 2^{-\varepsilon |\ell|} 2^{-j d(\frac1t - \frac{1}{t_0})}
=: A_4^{\ell} \, 2^{-j d(\frac1t - \frac{1}{t_0})},
\end{align}
for all $\mathscr{N}_q(1/\vec{t}) < \min\{1/q, 1/t_0\}$ and $0 < \frac{1}{t_0} < \frac1t := \sum_{i=1}^m \frac{1}{t_i}$. Since the function
\begin{align*}
F(s) := \mathscr{N}_s(1/\vec{t}) - \min\{1/s, 1/t_0\}, \quad s \in (1, \infty),
\end{align*}
is continuous, by the fact $F(q) < 0$, there exists some $q_0 \in (1, q)$ such that $F(q_0) < 0$. Pick 
\begin{align*}
& \theta := 1 - \frac{1}{q_0}, \quad 
\frac{1}{p_i} := \frac{\big(\frac{1}{t_i} + \frac{1}{q_0} - 1 \big)_+ + \beta}{1-\theta}, 
\quad \text{and} \quad
\frac{1}{q_i} := \frac{\frac{1}{t_i} - \frac{1-\theta}{p_i}}{\theta},
\end{align*}
where
\begin{align*}
0 < \beta < \frac{1}{2m} \bigg[\min \bigg\{\frac{1}{q_0}, \frac{1}{t_0} \bigg\}  
- \sum_{i=1}^m \bigg(\frac{1}{t_i} + \frac{1}{q_0} -1 \bigg)_+ \bigg].
\end{align*}
Then one can check $p_i, q_i \in (1, \infty)$ and 
\begin{align}\label{thpi}
\sum_{i=1}^m \frac{1-\theta}{p_i}
< \sum_{i=1}^m \bigg(\frac{1}{t_i} + \frac{1}{q_0} -1 \bigg)_+ + 2m \beta
< \min \bigg\{\frac{1}{q_0}, \frac{1}{t_0} \bigg\},
\end{align}
which gives
\begin{align*}
\sum_{i \in I} \frac{1}{p_i} + \frac{|I| - 1}{q(1 - \theta)} 
< \frac{1}{1-\theta} \frac{1}{q_0} + |I| - 1 = |I|, \quad \forall \varnothing \ne I \subset \{1, \ldots, m\}.
\end{align*}
This means $(\frac{1}{p_1}, \ldots, \frac{1}{p_m}) \in \H_m^d(q(1 - \theta))$. Define a $(m+1)$-linear operator $\mathbf{L}_j^{\ell}$ by 
\begin{align*}
\mathbf{L}_j^{\ell} (\omega, \vec{f}) 
:= \T_{j, P \omega}^{\ell}(\vec{f}) 
\quad \text{with} \quad
P \omega := \omega - \fint_{\bS^{md-1}} \omega \, d\sigma,
\end{align*}
where the $m$-linear operator $\T_{j, P \omega}^{\ell}$ is $\T_j^{\ell}$ with the rough kernel $\Omega$ replaced by $P \omega$. Thus, \eqref{eq:RTP} applied to $\Omega = P \omega$ and $\mathfrak{q} = q(1-\theta)$ leads to
\begin{equation}\label{eq:STE-1}
\|\mathbf{L}_j^{\ell}(\omega, \vec{f})\|_{L^p}
\lesssim 2^{- \delta_0 |\ell|} \|P \omega\|_{L^{q(1-\theta)}} \prod_{i=1}^m \|f_i\|_{L^{p_i}}
\lesssim 2^{- \delta_0 |\ell|} \|\omega\|_{L^{q(1-\theta)}} \prod_{i=1}^m \|f_i\|_{L^{p_i}},
\end{equation}
where $\frac1p = \sum_{i=1}^m \frac{1}{p_i}$. On the other hand, by Lemma \ref{lem:CZ} applied to $P \omega \in L^{\infty}$, 
\begin{align}\label{eq:DAk}
|\mathbf{L}_j^{\ell}(\omega, \vec{f})(x)|
\lesssim \|P \omega\|_{L^{\infty}} \prod_{i=1}^m \langle |f_i| \rangle_{B(x, (m+2)2^j)}
=: \|P \omega\|_{L^{\infty}} \prod_{i=1}^m \mathcal{A}_j f_i(x).
\end{align}
Note that by Young's inequality, 
\begin{align}\label{eq:dkh-2}
\|\mathcal{A}_j\|_{L^a \to L^b} 
\lesssim 2^{-jd(\frac1a - \frac1b)}, \quad 1 \le a \le b \le \infty.
\end{align}
Then \eqref{eq:DAk} and \eqref{eq:dkh-2} imply
\begin{equation}\label{eq:STE-2}
\|\mathbf{L}_j^{\ell}(\omega, \vec{f})\|_{L^{\infty}}
\lesssim 2^{- j d \frac1q} \|\omega\|_{L^{\infty}} \prod_{i=1}^m \|f_i\|_{L^{q_i}},
\end{equation}
where $\frac1q = \sum_{i=1}^m \frac{1}{q_i}$. Let $\frac{1}{\tau} := \frac{1-\theta}{p} + \frac{\theta}{\infty} < \frac{1}{t_0} < \frac1t$ due to \eqref{thpi}. Since
\begin{align*}
\frac1q = \frac{1-\theta}{q(1 - \theta)} + \frac{\theta}{\infty}
\quad \text{and} \quad
\frac{1}{t_i} = \frac{1-\theta}{p_i} + \frac{\theta}{q_i}, \quad i = 1, \ldots, m,
\end{align*}
interpolating between \eqref{eq:STE-1} and \eqref{eq:STE-2} yields 
\begin{equation*}
\|\T_{j, P \omega}^{\ell}(\vec{f})\|_{L^{\tau}}
= \|\mathbf{L}_j^{\ell}(\omega, \vec{f})\|_{L^{\tau}}
\lesssim 2^{-\delta_0 |\ell| (1-\theta)} 2^{- j d \frac{\theta}{q}}
\|\omega\|_{L^q} \prod_{i=1}^m \|f_i\|_{L^{t_i}}.
\end{equation*}
Pick $\varepsilon = \delta_0 (1 - \theta)$. Then the above applied to $\omega = \Omega \in L^q$ gives
\begin{equation*}
\|\T_j^{\ell}\|_{L^{t_1} \times \cdots \times L^{t_m} \to L^{\tau}}
\lesssim 2^{- \varepsilon |\ell|} 2^{- j d \frac{\theta}{q}} 
\|\Omega\|_{L^q},
\end{equation*}
which immediately indicates
\begin{align*}
\|\T_j^{\ell}(\vec{f})\|_{L^{t_0}}
& = \bigg(\sum_{Q \in \D: \ell(Q) = 2^j} \| \mathbf{1}_Q \T_j^{\ell}(\vec{f} \mathbf{1}_{7Q})\|_{L^{t_0}}^{t_0} \bigg)^{\frac{1}{t_0}}
\\
& \le \bigg( \sum_{Q \in \D: \ell(Q) = 2^j} |Q|^{(\frac{1}{t_0} - \frac{1}{\tau}) t_0} 
\| \mathbf{1}_Q \T_j^{\ell}(\vec{f} \mathbf{1}_{7Q})\|_{L^{\tau}}^{t_0} \bigg)^{\frac{1}{t_0}}
\\
& \lesssim 2^{- \varepsilon |\ell|} 2^{- j d \frac{\theta}{q}} 2^{j d (\frac{1}{t_0} - \frac{1}{\tau})} \|\Omega\|_{L^q} 
\bigg(\sum_{Q \in \D: \ell(Q) = 2^j} \prod_{i=1}^m \|f_i \mathbf{1}_{7Q}\|_{L^{t_i}}^{t_0} \bigg)^{\frac{1}{t_0}}
\\
& \le 2^{- \varepsilon |\ell|} 2^{-j d (\frac1t - \frac{1}{t_0})} \|\Omega\|_{L^q} 
\bigg(\sum_{Q \in \D: \ell(Q) = 2^j} \prod_{i=1}^m \|f_i \mathbf{1}_{7Q}\|_{L^{t_i}}^t \bigg)^{\frac1t}
\\
& \le 2^{- \varepsilon |\ell|} 2^{-j d (\frac1t - \frac{1}{t_0})} \|\Omega\|_{L^q} 
\prod_{i=1}^m \bigg(\sum_{Q \in \D: \ell(Q) = 2^j} \|f_i \mathbf{1}_{7Q}\|_{L^{t_i}}^{t_i} \bigg)^{\frac{1}{t_i}}
\\
& \lesssim 2^{- \varepsilon |\ell|} 2^{-j d (\frac1t - \frac{1}{t_0})} \|\Omega\|_{L^q} 
\prod_{i=1}^m \|f_i\|_{L^{t_i}},
\end{align*}
where we have used $\frac{\theta}{q} + \frac{1}{\tau} = \frac{\theta}{q} + \frac{1-\theta}{p} = \frac1t$ and $\sum_{Q \in \D: \ell(Q) = 2^k} \mathbf{1}_{7Q} = 7^d$. This proves \eqref{eq:STE}.

In order to show the auxiliary improving estimates in \eqref{H4}, choose $u_i = \max\{r_i, 2q'\}$ if $i \in I$, and $u_i = r_i$ if $i \notin I$. By definition and \eqref{Hmd2},
\begin{align*}
\mathscr{N}_q(1/\vec{u})
= \sum_{i \notin I} \bigg(\frac{1}{r_i} + \frac1q - 1\bigg)_+ 
< \min\bigg\{\mathscr{N}_q(1/\vec{r}), \, \sum_{i \notin I} \frac{1}{r_i}\bigg\}
\le \min\bigg\{\frac1q, \frac{1}{s(I)} \bigg\},
\end{align*}
which, along with \eqref{eq:STE} and the identity $\sum_{i=1}^m \frac{1}{u_i} = \sum_{i \notin I} \frac{1}{r_i} + \sum_{i \in I} \frac{1}{u_i} =: \frac{1}{s(I)} + \frac{1}{u(I)}$, implies
\begin{align*}
\|\T_j^{\ell}\|_{L^{u_1} \times \cdots \times L^{u_m} \to L^{s(I)}} 
\lesssim \|\Omega\|_{L^q} 2^{-\varepsilon |\ell|} 2^{- \frac{jd}{u(I)}}.
\end{align*}
The above further leads to
\begin{align*}
\|\T_j^{\ell}(\vec{f})\|_{L^{s(I)}}
& = \bigg(\sum_{Q \in \D: \ell(Q) = 2^j} \| \mathbf{1}_Q \T_j^{\ell}(\vec{f} \mathbf{1}_{7Q})\|_{L^{s(I)}}^{s(I)} \bigg)^{\frac{1}{s(I)}}
\\
& \lesssim 2^{-\varepsilon |\ell|} 2^{- \frac{jd}{u(I)}} \|\Omega\|_{L^q} 
\bigg(\sum_{Q \in \D: \ell(Q) = 2^j} \prod_{i=1}^m \|f_i \mathbf{1}_{7Q}\|_{L^{u_i}}^{s(I)} \bigg)^{\frac{1}{s(I)}}
\\
& \lesssim 2^{-\varepsilon |\ell|} \|\Omega\|_{L^q} \prod_{i \in I} \|f_i\|_{L^{\infty}}
\bigg(\sum_{Q \in \D: \ell(Q) = 2^j} \prod_{i \notin I} \|f_i \mathbf{1}_{7Q}\|_{L^{r_i}}^{s(I)} \bigg)^{\frac{1}{s(I)}}
\\
& \le 2^{-\varepsilon |\ell|} \|\Omega\|_{L^q} \prod_{i \in I} \|f_i\|_{L^{\infty}}
\prod_{i \not\in I} \bigg(\sum_{Q \in \D: \ell(Q) = 2^j} \|f_i \mathbf{1}_{7Q}\|_{L^{r_i}}^{r_i} \bigg)^{\frac{1}{r_i}}
\\
& \lesssim 2^{-\varepsilon |\ell|} \|\Omega\|_{L^q} 
\prod_{i \in I} \|f_i\|_{L^{\infty}} \prod_{i \notin I} \|f_i\|_{L^{r_i}},
\end{align*}
which proves \eqref{H4}.

Finally, it remains to demonstrate \eqref{H5}. Let $|h| \le 1$. In light of \eqref{def:KKje} and Lemma \ref{lem:CZ}, for each $\T \in \big\{\Delta_{2^j h} \T_j^{\ell}, \T_j^{\ell} \circ_1 \Delta_{2^j h}, \ldots, \T_j^{\ell} \circ_m \Delta_{2^j h} \big\}$, one has
\begin{align}\label{eq:dkh-1}
|\T (\vec{f})(x)|
\lesssim \|\Omega\|_{L^q} \, 2^{(md+1) |\ell|} |h| \prod_{i=1}^m \mathcal{A}_j f_i(x).
\end{align}
Take $v_i = (s/r) r_i$ so that $v_i \ge r_i$ and $\sum_{i=1}^m \frac{1}{v_i} = \frac1s$. Then it follows from \eqref{eq:dkh-2} and \eqref{eq:dkh-1} that 
\begin{align*}
\|\T (\vec{f})\|_{L^s}
& \lesssim \|\Omega\|_{L^q} \, 2^{(md+1) |\ell|} |h| \bigg\|\prod_{i=1}^m \mathcal{A}_j f_i \bigg\|_{L^s}
\\
& \le \|\Omega\|_{L^q} \, 2^{(md+1) |\ell|} |h| \prod_{i=1}^m \| \mathcal{A}_j f_i\|_{L^{v_i}}
\\
& \lesssim \|\Omega\|_{L^q} \, 2^{(md+1) |\ell|} |h| 2^{-jd(\frac1r - \frac1s)} \prod_{i=1}^m \|f_i\|_{L^{r_i}}.
\end{align*}
Recall that $r_i(I) = r_i^* = \infty$ if $i \in I$, and $r_i(I) = r_i$ if $i \notin I$. Then $\frac{1}{s(I)} = \sum_{i=1}^m \frac{1}{r_i(I)}$. As above, there holds
\begin{align*}
\|\T (\vec{f})\|_{L^{s(I)}}
\lesssim \|\Omega\|_{L^q} \, 2^{(md+1) |\ell|} |h| \prod_{i=1}^m \|f_i\|_{L^{r_i(I)}}.
\end{align*}
Thus, \eqref{H5} holds for $A_5 \simeq \|\Omega\|_{L^q} \, 2^{(md+1) |\ell|}$.

Now, by Theorem~\ref{thm:main},
\begin{align}\label{eq:RTL-1}
\sup_{\substack{N_1, N_2 \in \Z \\ N_1 \le N_2}} 
\bigg| \bigg\langle \sum_{j = N_1}^{N_2} \T_j^{\ell}(\vec{f}), \, f_{m+1} \bigg\rangle \bigg|
\lesssim \|\Omega\|_{L^q} (1 + |\ell|) 2^{- \varepsilon |\ell|}
\Lambda_{\vec{r}, s'}^*(\vec{f}, f_{m+1}),
\end{align}
for all $f_1, \ldots, f_{m+1} \in L_c^{\infty}(\Rd)$. It was shown in \cite{HP} that 
\begin{align*}
|\widehat{K}_j(\vec{\xi})| 
\lesssim \|\Omega\|_{L^q} \min\{|2^j \vec{\xi}|, |2^j \vec{\xi}|^{-\delta} \}, 
\end{align*}
for all $q \in (1, \infty]$ and $0 < \delta < \frac{1}{2q'}$. Hence,
\begin{align*}
|\widehat{K^{\ell}}(\vec{\xi})| 
= \bigg|\sum_{j \in \Z} \Psi(2^{j-\ell} \vec{\xi}) \widehat{K}_j(\vec{\xi}) \bigg|
\lesssim \|\Omega\|_{L^q} 2^{-\delta |\ell|},
\end{align*}
which implies
\begin{align}\label{eq:RTL-2}
\sum_{j=N_1}^{N_2} \K_j^{\ell} 
= \big[ \Phi(2^{-N_2} \cdot) - \Phi(2^{-N_1 + 1} \cdot) \big] K^{\ell} \, \to \,
K^{\ell} \quad\text{in } \, \mathscr{S}'(\R^{md}),
\end{align}
as $N_1 \to \infty$ and $N_2 \to \infty$. As a consequence of \eqref{eq:RTL-1} and  \eqref{eq:RTL-2}, we have
\begin{align*}
|\langle T^{\ell}(\vec{f}), \, f_{m+1} \rangle|
\lesssim \|\Omega\|_{L^q} (1 + |\ell|) 2^{- \varepsilon |\ell|}
\Lambda_{\vec{r}, s'}^*(\vec{f}, f_{m+1}),
\end{align*}
which gives
\begin{align*}
|\langle T_{\Omega}(\vec{f}), \, f_{m+1} \rangle|
\lesssim \|\Omega\|_{L^q} \Lambda_{\vec{r}, s'}^*(\vec{f}, f_{m+1}),
\end{align*}
for all $f_1, \ldots, f_{m+1} \in L_c^{\infty}(\Rd)$. This completes the proof of Theorem \ref{thm:rough}.

\subsubsection{{\bf The sufficiency for $d=1$}}
Now let us turn to the case $d=1$. Fix $q \in (1, \infty)$ and $1 < r_1, \ldots, r_m, s < \infty$ with $\frac1s < \frac1r = \sum_{i=1}^m \frac{1}{r_i}$ and $\big(\frac{1}{r_1}, \ldots, \frac{1}{r_m}, \frac{1}{s'} \big) \in \H_{m+1}^1(q)$.
We retain the same decomposition and the auxiliary exponents $r_1^* = \cdots r_m^* = \infty$. Observe that
\begin{equation}\label{Hmd1}
\begin{aligned}
\H_m^1(q) 
& = \big\{\vec{a} \in (0, 1)^m: \mathscr{N}_q(\vec{a}) < 1/q + \mathscr{M}_q(\vec{a}) \big\}, 
\qquad \text{and}
\\
\H_{m+1}^1(q) 
& = \big\{(\vec{a}, a_{m+1}) \in (0, 1)^{m+1}:
\mathscr{N}_q(\vec{a}) < \min\{1/q, \, \mathscr{M}_q(\vec{a}) + 1-a_{m+1} \}\big\},
\end{aligned}
\end{equation}
where $\mathscr{M}_q(\vec{a}) := \min_{1 \le i \le m} \big(a_i + \frac1q - 1\big)_+$. This gives
\begin{align}\label{Nrrr}
\mathscr{N}_q(1/\vec{r}) < \min \big\{1/q, \, \mathscr{M}_q(1/\vec{r}) + 1/s \big\}.
\end{align}
Although the preceding arguments are presented for $d \ge 2$, the proofs of \eqref{H1}--\eqref{H3}, \eqref{H5}, and the auxiliary estimates in \eqref{H4} remain valid for every $d \ge 1$. Their exponent requirements are $\mathscr{N}_q(1/\vec{r}) < 1/q$ and $r<s$, which are satisfied under the assumption \eqref{Nrrr}. Consequently, it suffices to 
establish the main improving estimate in \eqref{H4} in the larger range $\mathscr{N}_q(1/\vec{r}) < \mathscr{M}_q(1/\vec{r}) + 1/s$ than the previous one $\mathscr{N}_q(1/\vec{r}) < 1/s$:
\begin{align*}
\|\T_j^{\ell}\|_{L^{r_1} \times \cdots \times L^{r_m} \to L^s}
\lesssim \|\Omega\|_{L^q} \, 2^{-\varepsilon |\ell|} \, 2^{-j(\frac1r - \frac1s)},
\quad \forall j, \ell \in \Z,
\end{align*}
for some $\varepsilon > 0$. 

In the case $\ell < 0$, by \eqref{def:KKje} and Lemma \ref{lem:CZ},
\begin{align*}
|\T_j^{\ell} (\vec{f})(x)|
\lesssim \|\Omega\|_{L^q} \, 2^{- |\ell|} \prod_{i=1}^m \mathcal{A}_j f_i(x),
\end{align*}
which along with \eqref{eq:dkh-2} gives the desired estimate. To handle the case $\ell \ge 0$, choose $1 < s_- < s < s_+ < \infty$ such that
\begin{align*}
\mathscr{N}_q(1/\vec{r}) - \mathscr{M}_q(1/\vec{r}) < 1/s_+ < 1/s
\quad \text{and} \quad
\mathscr{N}_q(1/\vec{r}) < 1/s_- < 1/r.
\end{align*}
Then, $\mathscr{N}_q(1/\vec{r}) < \min\{1/q, 1/r\}$, which along with \eqref{eq:STE} implies
\begin{align*}
\|\T_j^{\ell}\|_{L^{r_1} \times \cdots \times L^{r_m} \to L^{s_-}}
\lesssim \|\Omega\|_{L^q} \, 2^{-\varepsilon |\ell|} 2^{-j d(\frac1r - \frac{1}{s_-})}.
\end{align*}
By interpolation, it is enough to show
\begin{align}\label{eq:Tjsp}
\|\T_j^{\ell}\|_{L^{r_1} \times \cdots \times L^{r_m} \to L^{s_+}}
\lesssim \|\Omega\|_{L^q} \, 2^{-j d(\frac1r - \frac{1}{s_+})}.
\end{align}
Note that setting $g_i = f_i(2^j \cdot)$ for each $i=1, \ldots, m$, we have 
\begin{align*}
\|g_i\|_{L^{r_i}} = 2^{-\frac{j}{r_i}} \|f_i\|_{L^{r_i}}
\quad \text{ and } \quad 
\T_j^{\ell}(\vec{f})(2^j x) = \T_j^{\ell}(\vec{g})(x),
\end{align*}
provide that $2^{jmd} \K_j^{\ell}(2^j \cdot) = 2^{jmd} (K^{\ell} \Psi_j)(2^j \cdot) = K^{\ell} \Psi_0 = \K_0^{\ell}$. This enables us to reduce \eqref{eq:Tjsp} to the case $j=0$.

Recall that $\K_0^{\ell} = \Psi \sum_{k \in \Z} \widehat{\Psi}_{\ell-k}*K_k$. 
If we let $T_K (\vec{f})(x) := \int_{\R^{md}} K(x-\vec{y}) \prod_{i=1}^m f_i(y_i) \, d\vec{y}$, then by Lemma \ref{lem:TKfg} parts \eqref{TKfg-33}--\eqref{TKfg-4} and Lemma \ref{lem:imp},
\begin{align}\label{eq:TPss-1}
\|T_{\Psi (\widehat{\Psi}_{\ell-k}*K_k)}\|_{L^{r_1} \times \cdots \times L^{r_m} \to L^{s_+}}
\lesssim \|T_{K_k}\|_{L^{r_1} \times \cdots \times L^{r_m} \to L^{s_+}}
\lesssim \|\Omega\|_{L^q} 2^{-k(\frac1r - \frac{1}{s_+})}.
\end{align}
Since $\supp \Psi \subset \{\frac12 \le|\vec{y}| \le 2 \}$ and $\supp K_k \subset \{2^{k-1} \le |\vec{z}| \le 2^{k+1} \}$, we have for any $k \le -3$, $\vec{y} \in \sup \Psi$, and $\vec{z} \in \supp K_k$, $|\vec{y} - \vec{z}| \ge |\vec{y}| - |\vec{z}| \ge 1/2 - 2^{k+1} \ge 1/4$. Thus, for any $k \le -3$ and $N>md$, the Schwartz decay of $\widehat{\Psi}$ yields
\begin{align*}
\big| \Psi(\vec{y}) \big(\widehat{\Psi}_{\ell-k}*K_k \big)(\vec{y}) \big|
& \le |\Psi(\vec{y})| \int_{\R^{md}} 2^{(\ell-k)md} 
\big|\widehat{\Psi} \big(2^{\ell-k}(\vec{y} - \vec{z}) \big) \big| |K_k(\vec{z})| \, d\vec{z}
\\
&\lesssim \|\Omega\|_{L^q} 2^{(k-\ell)(N-md)}
\mathbf{1}_{\{\frac12 \le |\vec{y}| \le 2\}},
\end{align*}
Then 
\begin{align*}
|T_{\Psi (\widehat{\Psi}_{\ell-k}*K_k)}(\vec{f})(x)|
\lesssim \|\Omega\|_{L^q} 2^{k(N-md)} \prod_{i=1}^m \A_0 f_i(x),
\end{align*}
which together with \eqref{eq:dkh-2} gives
\begin{align}\label{eq:TPss-2}
\|T_{\Psi (\widehat{\Psi}_{\ell-k}*K_k)}\|_{L^{r_1} \times \cdots \times L^{r_m} \to L^{s_+}}
\lesssim \|\Omega\|_{L^q} 2^{k(N-md)}, \quad\forall k \le -3.
\end{align}
Now gathering \eqref{eq:TPss-1} and \eqref{eq:TPss-2}, we decude
\begin{align*}
\|\T_0^{\ell}\|_{L^{r_1} \times \cdots \times L^{r_m} \to L^{s_+}}
\le \bigg(\sum_{k >-3} + \sum_{k  \le -3}\bigg) 
\|T_{\Psi (\widehat{\Psi}_{\ell-k}*K_k)}\|_{L^{r_1} \times \cdots \times L^{r_m} \to L^{s_+}}
\lesssim \|\Omega\|_{L^q}.
\end{align*}
This completes the proof of the necessity for $d=1$.

\begin{lemma}\label{lem:CZ}
For every multi-index $\alpha$, 
\begin{align*}
|\partial^{\alpha} K^{\ell}(\vec{y})| 
\lesssim_{\alpha} \mathscr{B}_{\ell}^{\alpha} \|\Omega\|_{L^q} |\vec{y}|^{-md-|\alpha|}, 
\quad \forall \vec{y} \ne 0,
\end{align*}
where $\mathscr{B}_{\alpha}^{\ell} = 2^\ell$ if $\ell < 0$, and $\mathscr{B}_{\alpha}^{\ell} = 2^{\ell(md/q+|\alpha|)}$ if $\ell \ge 0$.
\end{lemma}

\begin{proof}
By homogeneity, $K_j^{\ell}(\vec{y}) = 2^{-jmd} G_{\ell}(2^{-j} \vec{y})$, where $G_{\ell} = K_0* \widehat{\Psi}_{\ell}$.
Since $\Psi$ is radial and $\Omega$ is mean zero, we have
\begin{equation*}
\int_{\R^{md}} K_0 \, d\vec{y} = 0
\quad \text{and} \quad
\|K_0\|_{L^1} + \|K_0\|_{L^q} + \int_{\R^{md}} |\vec{y}| \, |K_0(\vec{y})| \, d\vec{y}
\lesssim \|\Omega\|_{L^q}.
\end{equation*}
Suppose first that $\ell < 0$. The above gives
\begin{align*}
|\partial^{\alpha} G_{\ell}(\vec{y})|
& = \bigg|\int_{\R^{md}} K_0(\vec{z})
\big[\partial^{\alpha} \widehat{\Psi}_{\ell}(\vec{y} - \vec{z})
- \partial^{\alpha} \widehat{\Psi}_{\ell}(\vec{y}) \big] \, d\vec{z} \bigg|
\\
& = \bigg|\int_{\R^{md}} K_0(\vec{z}) \int_0^1 \vec{z} \cdot \nabla \partial^{\alpha}
\widehat{\Psi}_{\ell}(\vec{y} - t \vec{z}) \, dt \, d\vec{z} \bigg|
\\
& \lesssim \int_{\R^{md}} |\vec{z}| |K_0(\vec{z})| 
\int_0^1 \frac{2^{\ell(md+1+|\alpha|)}}{(1+2^{\ell}|\vec{y} - t \vec{z}|)^N} \, dt \, d\vec{z}
\\
& \lesssim \frac{2^{\ell(md+1+|\alpha|)}}{(1+2^{\ell} |\vec{y}|)^N} \|\Omega\|_{L^q}.
\end{align*}
Choosing $N > md+|\alpha|$, we obtain
\begin{align*}
\sum_{j \in \Z} |\partial^{\alpha} K_j^{\ell}(\vec{y})|
& = \sum_{j \in \Z} 2^{-j(md + |\alpha|)} |(\partial^{\alpha} G_{\ell})(2^{-j} \vec{y})|
\\
& \lesssim 2^{\ell} \|\Omega\|_{L^q} \sum_{j \in \Z} 
\frac{2^{(\ell-j)(md+|\alpha|)}}{(1+2^{\ell-j}|\vec{y}|)^N}
\lesssim 2^{\ell} \|\Omega\|_{L^q} |\vec{y}|^{-md-|\alpha|},
\end{align*}
provided that
\begin{equation*}
\sum_{j \in \Z} (2^{-j}\rho)^A(1+2^{-j}\rho)^{-N}
\lesssim_{A, N} 1, \quad \forall \rho>0, \, 0<A<N.
\end{equation*}

Now suppose that $\ell \ge 0$. By H\"older's inequality,
\begin{equation*}
\|\partial^{\alpha} G_{\ell}\|_{L^{\infty}}
\le \|K_0\|_{L^q} \|\partial^{\alpha} \widehat{\Psi}_{\ell}\|_{L^{q'}}
\lesssim \|\Omega\|_{L^q} 2^{\ell(md/q + |\alpha|)}.
\end{equation*}
The fact $\widehat{\Psi}_{\ell} \in \S(\R^{md})$ implies 
\begin{equation*}
|\partial^{\alpha} \widehat{\Psi}_{\ell}(\vec{z})|
= 2^{\ell(md+|\alpha|)} | (\partial^{\alpha} \widehat{\Psi})(2^{\ell} \vec{z})|
\lesssim_N 2^{\ell(md+|\alpha|)} (1+2^{\ell} |\vec{z}|)^{-N}, \quad \forall N>0,
\end{equation*}
If $|\vec{y}| \le1/4$, then for any $1/2 \le |\vec{z}| \le 2$, we have $|\vec{y} - \vec{z}| \ge |\vec{z}| - |\vec{y}| \ge1/4$ and
\begin{equation*}
|\partial^{\alpha} G_{\ell}(\vec{y})|
\lesssim_N \|\Omega\|_{L^q} 2^{\ell(md+|\alpha|-N)}.
\end{equation*}
If $|\vec{y}| \ge 4$, then for any $1/2 \le |\vec{z}| \le 2$, we have $|\vec{y} - \vec{z}| \ge |\vec{y}| - |\vec{z}| \ge |\vec{y}|/2$ and
\begin{equation*}
|\partial^{\alpha} G_{\ell}(\vec{y})|
\lesssim_N \|\Omega\|_{L^q} 2^{\ell(md+|\alpha|-N)} |\vec{y}|^{-N}.
\end{equation*}
Therefore, the estimates above yield
\begin{align*}
\sum_{j \in \Z} |\partial^{\alpha} K_j^{\ell}(\vec{y})|
& = \sum_{j \in \Z} 2^{-j(md + |\alpha|)} |(\partial^{\alpha} G_{\ell})(2^{-j} \vec{y})|
\\
& \lesssim  \|\Omega\|_{L^q} \sum_{2^{-j} |\vec{y}| \le 1/4} 2^{-j(md + |\alpha|)} 2^{\ell(md+|\alpha|-N)}
\\
& \quad + \|\Omega\|_{L^q} \sum_{1/4 < 2^{-j} |\vec{y}| < 4} 2^{-j(md + |\alpha|)} 2^{\ell(md/q+|\alpha|)}
\\
& \quad + \|\Omega\|_{L^q} \sum_{2^{-j} |\vec{y}| \ge 4} 2^{-j(md + |\alpha|)} 
2^{\ell(md+|\alpha|-N)} (2^{-j} |\vec{y}|)^{-N}
\\
& \lesssim \|\Omega\|_{L^q} 2^{\ell(md/q+|\alpha|)} |\vec{y}|^{-md-|\alpha|},
\end{align*}
provided that $N > md + |\alpha|$.
\end{proof}

\subsection{Bilinear pseudo-differential operators}\label{sec:pse}
Next, we prove Theorem \ref{thm:pse}. For this purpose, we begin with showing the necessity, and then justify the sufficiency at the subcritical order, which will be used in the proof of the sufficiency at the critical order. 

\subsubsection{{\bf The necessity. }}
Suppose that there exists some large $N>1$ so that
\begin{equation}\label{eq:psesp}
|\langle T_{\sigma}(f_1, f_2), f_3 \rangle|
\lesssim \|\sigma\|_{BS_{\rho, \delta}^{m, N}} \,
\Lambda^*_{r_1, r_2, s'}(f_1, f_2, f_3)
\end{equation}
for all $\sigma \in BS_{\rho, \delta}^{m, N}$ and for all $f_1, f_2, f_3 \in L_c^{\infty}(\Rd)$. By a standard smooth cutoff argument, \eqref{eq:psesp} can be extended to all $f_1, f_2, f_3 \in \mathscr{S}(\Rd)$, with the same implicit constant. Thus, in what follows, we will construct $\sigma \in BS_{\rho, \delta}^{m, N}$ and $f_1, f_2, f_3 \in \mathscr{S}(\Rd)$.

Denote $r_3 := s'$ and $\bar{r}_i := \min\{r_i, 2\}$ for each $i=1, 2, 3$. Then $\Gamma(r_1, r_2, s) = \sum_{i=1}^3 \frac{1}{\bar{r}_i} - 1$. Let $\psi \in \mathscr{S}(\Rd)$ be a nonnegative function such that $\supp \widehat{\psi} \subset [-\frac{1}{32}, \frac{1}{32}]^d$ and $C_{\psi} := \int_{\Rd} \psi^3 \, dx > 0$. Let $\phi \in \mathscr{C}_c^{\infty}(\Rd)$ with $\supp \phi \subset [-\frac14, \frac14]^d$ and $\phi = 1$ on $[-\frac{1}{16}, \frac{1}{16}]^d$.

For $\lambda \ge 2$, let $K = \lfloor \lambda^{1-\rho} \rfloor$, $E_1 = E_2 = \{K, \ldots, 2K-1\}^d$, and $E_3 = E_1 + E_2 = \{2K, \ldots, 4K-2\}^d$. Thus $|E_1| = |E_2| = K^d \le |E_3| \le (2K)^d$. If $1 < r_i \le 2$, let $\varepsilon_{i, v} = 1$; if $r_i > 2$, choose the signs $\varepsilon_{i, v} \in\{-1, 1\}$ so that 
\begin{align*}
\mathcal J_i(\varepsilon)
:= \bigg\| \langle x \rangle^{\kappa} \psi(x)
\sum_{v \in E_i} \varepsilon_v e^{2\pi i v \cdot x} \bigg\|_{L^{r_i}_x}^{r_i}
\end{align*} 
is minimal among all possible choices of signs, where $\langle x \rangle = (1+|x|^2)^{1/2}$. Such a choice exists due to the finiteness of $E_i$. Then define
\begin{align*}
F_i(x) = \psi(x) P_i(x)  \quad \text{ and } \quad
P_i(x) = \sum_{v \in E_i} \varepsilon_{i, v} \, e^{2 \pi i v \cdot x}.
\end{align*} 
We claim that
\begin{align*}
\|\langle \cdot \rangle^{\kappa} F_i\|_{L^{r_i}}
\lesssim K^{d(1- \frac{1}{\bar{r}_i})}, \qquad i=1, 2, 3.
\end{align*}
Indeed, in the case $1 < r_i \le 2$, since $P_i$ is periodic and $\psi$ is Schwartz,
\begin{align*}
\|\langle \cdot \rangle^{\kappa} F_i\|_{L^{r_i}}^{r_i}
= \int_{\bT^d} |P_i(x)|^{r_i}
\sum_{m \in \Z^d} \br{x+m}^{Ar_i} |\psi(x+m)|^{r_i} \, dx
& \lesssim \|P_i\|_{L^{r_i}(\bT^d)}^{r_i}
\lesssim K^{d(r_i-1)}.
\end{align*}
In the case $r_i > 2$, averaging over all sign vectors and applying Khintchine's inequality, we obtain
\begin{align*}
\|\langle \cdot \rangle^{\kappa} F_i\|_{L^{r_i}}^{r_i}
& = \min_{\varepsilon \in\{-1, 1\}^{E_i}} \mathcal{J}_i(\varepsilon)
\le 2^{-|E_i|} \sum_{\varepsilon \in \{-1, 1\}^{E_i}} \mathcal{J}_i(\varepsilon)
\\
& = \int_{\Rd} \langle x \rangle^{\kappa r_i}|\psi(x)|^{r_i} \, \mathbb{E} \bigg|
\sum_{v \in E_i} \varepsilon_v e^{2 \pi i v \cdot x} \bigg|^{r_i} \, dx
\\
& \lesssim_{r_i} |E_i|^{r_i/2} \|\langle \cdot \rangle^{\kappa} \psi\|_{L^{r_i}}^{r_i}
\lesssim K^{dr_i/2}.
\end{align*}

By Lemma~\ref{lem:trans}, there exist some $t_1, t_2, t_3 \in \bT^d$ such that
\begin{align*}
\Lambda_{\vec r}^{*}
\big(F_1(\cdot - t_1), F_2(\cdot - t_2), F_3(\cdot - t_3)\big)
\lesssim \prod_{i=1}^3 \|\br{\cdot}^{\kappa} F_i\|_{L^{r_i}}
\lesssim K^{d(3- \frac{1}{\bar{r}_1} - \frac{1}{\bar{r}_2} - \frac{1}{\bar{r}_3})}.
\end{align*}
If we take
\begin{align*}
f_i(x) := F_i(\lambda^{\rho} x - t_i), \qquad i=1, 2, 3,
\end{align*}
then
\begin{align}\label{eq:packet}
\Lambda_{\vec r}^{*}(f_1, f_2, f_3)
& = \lambda^{-d \rho} \Lambda_{\vec r}^{*} \big(F_1(\cdot - t_1), F_2(\cdot-t_2), F_3(\cdot-t_3) \big)
\\ \notag
& \lesssim \lambda^{-d \rho} K^{d(3- \frac{1}{\bar{r}_1} - \frac{1}{\bar{r}_2} - \frac{1}{\bar{r}_3})}
= \lambda^{-d \rho} K^{d[2 - \Gamma(r_1, r_2, s)]}.
\end{align}

Define the $x$-independent symbols by
\begin{align*}
\widetilde{\sigma}(\xi, \eta)
& := \lambda^m \sum_{u, v \in E_1}
\varepsilon_{1, u} \, \varepsilon_{2, v} \, 
\varepsilon_{3, u+v} \, \phi(\lambda^{-\rho} \xi - u) \phi(\lambda^{-\rho} \eta - v),
\\
\sigma(\xi, \eta)
& := \exp \big(2\pi i \lambda^{-\rho} [(t_1 - t_3) \cdot \xi + (t_2 - t_3) \cdot \eta] \big) 
\widetilde{\sigma}(\xi, \eta).
\end{align*}
By Leibniz's rule and the fact $\supp \sigma \subset [\lambda/4, \lambda]^d \times [\lambda/4, \lambda]^d$, we have 
\begin{align*}
|\partial_{\xi}^{\beta} \partial_{\eta}^{\gamma} \sigma(\xi, \eta)|
\le C_{\beta, \gamma} \lambda^{m - \rho(|\beta| + |\kappa|)}
\le C_{\beta, \gamma} (1 + |\xi| + |\eta|)^{m - \rho(|\beta| + |\kappa|)},
\end{align*}
for all multi-indices $\beta, \gamma$, which means 
\begin{align}\label{eq:sbs}
\|\sigma\|_{BS_{\rho, \delta}^{m, N}} \le C_N, \quad \forall N \in \N, \lambda \ge 2.
\end{align}

Note that
\begin{align*}
\widehat{f}_i(\xi)
= \lambda^{-d \rho} e^{-2 \pi i \lambda^{-\rho} t_i \cdot \xi} 
\sum_{v \in E_i} \varepsilon_{i, v} \widehat{\psi}(\lambda^{-\rho} \xi - v), \quad i=1, 2,
\end{align*}
which, along with the supports of $\phi$ and $\widehat{\psi}$ and a change of variables, implies
\begin{align*}
T_{\sigma}(f_1, f_2)(x)
= \lambda^m \psi(\lambda^{\rho} x - t_3)^2 
\sum_{u, v \in E_1} \varepsilon_{3, u+v} e^{2 \pi i(u+v) \cdot (\lambda^{\rho} x - t_3)}.
\end{align*}
Thus, 
\begin{align}\label{eq:exa}
\langle T_{\sigma}(f_1, f_2), f_3 \rangle
& = \lambda^{m-d \rho} \sum_{\substack{u, v \in E_1 \\ w \in E_3}}
\varepsilon_{3, u+v} \, \varepsilon_{3, w}
\int_{\Rd} \psi(y)^3 e^{2\pi i(u + v - w) \cdot y}\,dy
\\ \notag
& = C_{\psi} \lambda^{m-d \rho} \sum_{u, v \in E_1} \varepsilon_{3, u+v}^{2}
= C_{\psi} \lambda^{m-d \rho} K^{2d},
\end{align}
provided that
\begin{align*}
\supp \widehat{\psi^3} \subset [-\tfrac{3}{32}, \tfrac{3}{32}] 
\quad \text{and} \quad 
\int_{\Rd} \psi(y)^3 e^{2 \pi i k \cdot y} \, dy
= \widehat{\psi^3}(-k) 
= C_{\psi} \mathbf{1}_{\{k=0\}}, \quad\forall k \in \Z^d.
\end{align*}
Gathering \eqref{eq:packet}--\eqref{eq:exa}, we use sparse domination for $T_{\sigma}$ to conclude
\begin{align*}
\lambda^{m-d \rho} K^{2d}
\lesssim \lambda^{-d \rho} K^{d[2 - \Gamma(r_1, r_2, s)]}, 
\quad \text{hence, } 
\lambda^{m + d(1- \rho) \Gamma(r_1, r_2, s)} \lesssim 1.
\end{align*}
If $m + d(1- \rho) \Gamma(r_1, r_2, s) > 0$, letting $\lambda \to \infty$ leads to a contradiction. This indicates $m + d(1- \rho) \Gamma(r_1, r_2, s) \le 0$.

\subsubsection{{\bf The sufficiency at the subcritical order. }}
Fix $m < -d(1 - \rho) \Gamma(r_1, r_2, s)$ and set $\kappa := - m -d(1 - \rho) \Gamma(r_1, r_2, s)$. By homogeneity, we may assume that $\|\sigma\|_{BS_{\rho, \delta}^{m, N}} \le 1$, where $N$ is sufficiently large. Let $\phi \in \mathscr{C}_c^{\infty}(\R^{2d})$ be a cutoff function such that $\mathbf{1}_{B(0, 1)} \le \phi \le \mathbf{1}_{B(0, 2)}$. Define $\phi_0 := \phi$ and $\phi_{\ell}(\vec{\xi}) := \phi_0(2^{-\ell} \vec{\xi}) - \phi_0(2^{1-\ell} \vec{\xi})$ for each $\ell \ge 1$, which verify the following properties
\begin{align*}
\supp(\phi_0)  \subset \big\{|\vec{\xi}| \leq 2 \big\}, \quad
\supp(\phi_\ell) \subset \big\{2^{\ell-1} \leq |\vec{\xi}| \leq 2^{\ell+1}\big\}, \quad \text{and} \quad
\sum_{\ell=0}^{\infty} \phi_{\ell} = 1.
\end{align*}
Setting 
\begin{align*}
\sigma_{\ell}(x, \vec{\xi}) := \sigma(x, \vec{\xi}) \phi_{\ell}(\vec{\xi}) 
\quad \text{ and } \quad
K_{\sigma_{\ell}}(x, \vec{y}) 
:= \int_{\R^{2d}} \sigma_{\ell}(x, \vec{\xi}) \, e^{2\pi i (x - \vec{y}) \cdot \vec{\xi}} \, d\vec{\xi},
\end{align*}
one has
\begin{align*}
T_{\sigma} = \sum_{\ell \ge 0} T^{\ell}
\quad \text{ and } \quad
T^{\ell}(\vec{f})(x) := T_{\sigma_{\ell}}(\vec{f})(x)
= \int_{\R^{2d}} K_{\sigma_{\ell}}(x, \vec{y}) f_1(y_1) f_2(y_2) \, d\vec{y}.
\end{align*}
Fix $0 < \varepsilon < \rho$ and let $j_{\ell} = \lfloor (\rho - \varepsilon) \ell \rfloor$. We further split $T^{\ell}$ into spatially localized operators $T^{\ell}_j$: 
\begin{align*}
T^{\ell}_j(\vec{f})(x) 
:= \int_{\R^{2d}} K^{\ell}_j(x, y_1, y_2) f_1(y_1) f_2(y_2) \, d\vec{y},
\end{align*}
where
\begin{align*}
K^{\ell}_j(x, y_1, y_2) 
:= K_{\sigma_{\ell}}(x, y_1, y_2) \phi_j \big(2^{j_{\ell}} (x - y_1, x - y_2) \big).
\end{align*}
Set 
\begin{align*}
\T_j^{\ell} := T_{j + j_{\ell}}^{\ell} \, \text{ if } j \ge - j_{\ell}, \quad 
\T_j^{\ell} := 0 \, \text{ if } j < - j_{\ell}.
\end{align*} 
This allows us to write 
\begin{align}\label{Taaj}
T_{\sigma} = \sum_{\ell \ge 0} T^{\ell} 
\quad \text{and} \quad
T^{\ell} = \sum_{j \ge 0} T^{\ell}_j = \sum_{j \in \Z} \T^{\ell}_j.
\end{align}

Denote $\bar{r}_1 := \min\{r_1, 2\}$, $\bar{r}_2 := \min\{r_2, 2\}$, $\bar{s} := \max\{s, 2\}$, and $\frac{1}{\bar{r}} := \frac{1}{\bar{r}_1} + \frac{1}{\bar{r}_2}$, which satisfies $\Gamma := \Gamma(r_1, r_2, s) = \frac{1}{\bar{r}} - \frac{1}{\bar{s}}$. We are going to show the family $\{\T_j^{\ell}\}_{j \ge 0}$ verifies the hypotheses \eqref{H1}--\eqref{H5} for exponents $(\bar{r}_1, \bar{r}_2, \bar{s})$. 

The kernel of $\T_j^{\ell}$ is given by $\K_j^{\ell} = K_{j + j_{\ell}}^{\ell}$. By definition, $\supp \K^{\ell}_j \subset \{(x, y_1, y_2) \in \R^{3d}: |(x-y_1, x-y_2)| \le 2 \cdot 2^j\}$, which gives \eqref{H1}. To proceed, we claim that there exists some $\kappa > 0$ such that for every $j \ge 1$,
\begin{align}
\label{TEL-2}
\| T^{\ell}_j \|_{L^{p_1} \times L^{p_2} \to L^{p_0}}
& \lesssim 2^{-\kappa \ell} 2^{-2j} 
\, 2^{- d(\frac1p - \frac{1}{p_0}) (j- j_{\ell})},
\\
\label{TEL-3}
\| S^{\ell}_j \|_{L^{p_1} \times L^{p_2} \to L^{p_0}}
& \lesssim |h| 2^{-\kappa \ell} 2^{-2j} 2^{j- j_{\ell}}
\, 2^{- d(\frac1p - \frac{1}{p_0}) (j- j_{\ell})},
\end{align}
for all $|h| \le 1$ and $0 < \frac{1}{p_0} \le \frac1p = \frac{1}{p_1} + \frac{1}{p_2}$ with $p_1, p_2 \in [1, \infty]$, where 
\begin{align*}
S_j^{\ell} \in \big\{\Delta_{2^{j - j_{\ell}} h} T_j^{\ell}, 
\, T_j^{\ell}(\Delta_{2^{j - j_{\ell}} h} \otimes \Id), \, T_j^{\ell}(\Id \otimes \Delta_{2^{j - j_{\ell}} h}) \big\}.
\end{align*}
We begin with proving that for every $N \in \N$,
\begin{align}
\label{KSN-1}
|K_{\sigma_{\ell}}(x, \vec{y})| 
& \lesssim_N \frac{2^{\ell(m+2d)}}{(1 + 2^{\rho \ell} |x - \vec{y}|)^{2N}}, 
\\
\label{KSN-2}
|\nabla K_{\sigma_{\ell}}(x, \vec{y})| 
& \lesssim_N \frac{2^{\ell(m+2d+1)}}{(1 + 2^{\rho \ell} |x - \vec{y}|)^{2N}}.
\end{align}
Indeed, by the identity
\begin{align*}
(\Id - 2^{2 \rho \ell} \Delta_{\xi})^k e^{i z \cdot \xi} 
= (1 + 2^{2 \rho \ell} |z|^2)^k e^{i z \cdot \xi}, \quad \forall k \in \N,
\end{align*} 
we use integration by parts to obtain 
\begin{align*}
|K_{\sigma_{\ell}}(x, \vec{y})|
&= \bigg| \frac{1}{(1 + 4 \pi^2 2^{2 \rho \ell} |x - \vec{y}|^2)^N}
\int_{\R^{2d}} \sigma_{\ell}(x, \vec{\xi})
(\Id - 2^{2 \rho \ell} \Delta_{\vec{\xi}})^N e^{2\pi i (x - \vec{y}) \cdot \vec{\xi}} \, d\vec{\xi} \bigg|
\\
&= \bigg| \frac{1}{(1 + 4 \pi^2 2^{2 \rho \ell} |x - \vec{y}|^2)^N}
\int_{\R^{2d}} (\Id - 2^{2 \rho \ell} \Delta_{\vec{\xi}})^N \sigma_{\ell}(x, \vec{\xi}) 
\, e^{2\pi i (x - \vec{y}) \cdot \vec{\xi}} \, d\vec{\xi} \bigg|
\\
& \lesssim \sum_{|\vec{\beta}| \le 2N} \int_{|\vec{\xi}| \simeq 2^{\ell}} 
\frac{(1 + |\vec{\xi}|)^{m - \rho |\vec{\beta}|}}{(1 + 2^{\rho \ell} |x - \vec{y}|)^{2N}} \, d\vec{\xi}
\lesssim \frac{2^{\ell(m + 2d)}}{(1 + 2^{\rho \ell} |x - \vec{y}|)^{2N}},
\end{align*}
which verifies \eqref{KSN-1}. Observe that $\nabla_x K_{\sigma_{\ell}} = K_{\sigma_{\ell}^0}$, $\nabla_{y_1} K_{\sigma_{\ell}} = K_{\sigma_{\ell}^1}$, $\nabla_{y_2} K_{\sigma_{\ell}} = K_{\sigma_{\ell}^2}$,
where 
\begin{align*}
\sigma_{\ell}^0(x, \xi_1, \xi_2) 
& = \nabla_x \sigma_{\ell}(x, \xi_1, \xi_2) + 2 \pi i (\xi_1 + \xi_2) \sigma_{\ell}(x, \xi_1, \xi_2) \in BS_{\rho, \delta}^{m+1}
\\
\sigma_{\ell}^k(x, \xi_1, \xi_2) 
& = -2 \pi i \xi_k \sigma_{\ell}(x, \xi_1, \xi_2) \in BS_{\rho, \delta}^{m+1}, \quad k=1, 2.
\end{align*}
This along with \eqref{KSN-1} implies \eqref{KSN-2}. Then it follows from \eqref{KSN-1} and \eqref{KSN-2} that
\begin{align}
\label{KSN-3}
|K_j^{\ell}(x, \vec{y})| 
& \lesssim 2^{\ell(m+2d)} 2^{-2N(j + \varepsilon \ell)} 
\mathbf{1}_{\{|x-\vec{y}| \le 2^{j+1-j_{\ell}}\}},
\\
\label{KSN-4}
|\nabla K_j^{\ell}(x, \vec{y})| 
& \lesssim 2^{\ell(m+1+2d)} 2^{-2N(j + \varepsilon \ell)} 
\mathbf{1}_{\{|x-\vec{y}| \le 2^{j+1-j_{\ell}}\}}.
\end{align}
Take $q_1 \ge p_1$ and $q_2 \ge p_2$ so that $\frac{1}{p_0} = \frac{1}{q_1} + \frac{1}{q_2} \le \frac{1}{p_1} + \frac{1}{p_2} = \frac1p$. Choose $0 < \varepsilon < \rho$ small enough so that $2 \varepsilon d \Gamma < \kappa$. 
Using \eqref{KSN-3} and Young's inequality, we deduce
\begin{align*}
\|T_j^{\ell}(f_1, f_2)\|_{L^{p_0}}
& \lesssim 2^{\ell(m+2d)} 2^{-2N(j + \varepsilon \ell)} 
\prod_{i=1}^2 \||f_i|* \mathbf{1}_{B(0, 2^{j+1-j_{\ell}})}\|_{L^{q_i}}
\\
& \lesssim 2^{\ell(m+2d)} 2^{-2N(j + \varepsilon \ell)} 
2^{d(j - j_{\ell}) (2 - \frac1p + \frac{1}{p_0})} 
\prod_{i=1}^2 \||f_i\|_{L^{p_i}}
\\
& \lesssim 2^{- \ell [2 \varepsilon (N-d) - m - 2d(1-\rho)]} 2^{-2j(N-d)} 
\, 2^{- d(\frac1p - \frac{1}{p_0}) (j- j_{\ell})}
\prod_{i=1}^2 \||f_i\|_{L^{p_i}}
\\
& \le 2^{-\kappa \ell} 2^{-2j} \prod_{i=1}^2 \||f_i\|_{L^{p_i}},
\end{align*}
where we have chosen $N > d+1$ sufficiently large so that $2 \varepsilon (N-d) > m + 1 + 2d(1-\rho) + \kappa$. This shows \eqref{TEL-2}. Similarly, one has \eqref{TEL-3} by the mean value theorem.

Let $\tau = \min\{1, p_0\}$. Then in light of \eqref{TEL-2} and Lemma \ref{lem:kmt}, 
\begin{align}\label{TEL-0}
\|T_0^{\ell}(f_1, f_2)\|_{L^{p_0}}
& \lesssim \|T^{\ell}(f_1, f_2)\|_{L^{p_0}} 
+ \bigg\|\sum_{j \ge 1} T_j^{\ell}(f_1, f_2) \bigg\|_{L^{p_0}}
\\ \nonumber
& \le \|T^{\ell}(f_1, f_2)\|_{L^{p_0}} 
+ \bigg[\sum_{j \ge 1} \|T_j^{\ell}(f_1, f_2)\|_{L^{p_0}}^{\tau} \bigg]^{\frac{1}{\tau}}
\\ \nonumber
& \lesssim \Big[2^{\ell[ m + d(1-\rho) \Gamma(p_1, p_2, p_0) + \rho d (\frac1p - \frac{1}{p_0})]} 
+ 2^{-\kappa \ell} 2^{d(\frac1p - \frac{1}{p_0}) j_{\ell}} \Big]
\prod_{i=1}^2 \|f_i\|_{L^{p_i}}.
\end{align}
Note that $m + d(1-\rho) \Gamma(\bar{r}_1, \bar{r}_2, \bar{r}) \le m + d(1-\rho) \Gamma = - \kappa$, which along with \eqref{TEL-2} and \eqref{TEL-0} implies 
\begin{align*}
\|\T_j^{\ell}\|_{L^{\bar{r}_1} \times L^{\bar{r}_2} \to L^{\bar{r}}}
\lesssim 2^{-\kappa \ell}, \quad \forall j \in \Z.
\end{align*}
This yields \eqref{H2} with $A_2^{\ell} \simeq 2^{-\kappa \ell}$.

Likewise, \eqref{H3} holds with $A_3^{\ell} \simeq 2^{-\kappa \ell}$ because
\begin{align*}
\|\T_j^{\ell}\|_{L^{2 \bar{s}} \times L^{2 \bar{s}} \to L^{\bar{s}}}
\lesssim 2^{-\kappa \ell}, \quad \forall j \in \Z,
\end{align*}
provided that
\begin{align*}
m + d(1-\rho) \Gamma(2 \bar{s}, 2 \bar{s}, \bar{s}) 
= m + d(1-\rho)(1 - 1/\bar{s})
\le m + d(1-\rho) \Gamma = - \kappa.
\end{align*}
Moreover, 
\begin{align*}
\Gamma(\bar{r}_1, \bar{r}_2, \bar{s}) = \Gamma, \quad
\Gamma(\bar{r}_1, \infty, \bar{r}_1) = 1/\bar{r}_1 \le \Gamma, \quad
\Gamma(\infty, \bar{r}_2, \bar{r}_2) = 1/\bar{r}_2 \le \Gamma,
\end{align*}
and
\begin{align*}
& 2^{\ell[ m + d(1-\rho) \Gamma(\bar{r}_1, \bar{r}_2, \bar{s}) + \rho d (\frac{1}{\bar{r}} - \frac{1}{\bar{s}})]}
+ 2^{-\kappa \ell} 2^{d(\frac{1}{\bar{r}} - \frac{1}{\bar{s}}) j_{\ell}}
\\
& = 2^{\ell[ m + d(1-\rho) \Gamma + \varepsilon d \Gamma]} 2^{(\rho - \varepsilon) \ell d \Gamma}
+ 2^{-\kappa \ell} 2^{j_{\ell} d \Gamma}
\\
& \lesssim 2^{- \kappa \ell/2} 2^{j_{\ell} d \Gamma}
+ 2^{-\kappa \ell} 2^{j_{\ell} d \Gamma}.
\end{align*}
In view of \eqref{TEL-2} and \eqref{TEL-0}, the above indicates
\begin{align*}
\|\T_j^{\ell}\|_{L^{\bar{r}_1} \times L^{\bar{r}_2} \to L^{\bar{s}}}
& \lesssim 2^{-\kappa \ell/2} 2^{- d j \Gamma}, \quad \forall j \in \Z,
\\
\|\T_j^{\ell}\|_{L^{\bar{r}_1} \times L^{\infty} \to L^{\bar{r}_1}}
& \lesssim 2^{-\kappa \ell/2}, \quad \forall j \in \Z,
\\
\|\T_j^{\ell}\|_{L^{\infty} \times L^{\bar{r}_2} \to L^{\bar{r}_2}}
& \lesssim 2^{-\kappa \ell/2}, \quad \forall j \in \Z,
\end{align*}
which justifies \eqref{H4} with $A_4^{\ell} \simeq 2^{-\kappa \ell/2}$.

Finally, to achieve \eqref{H5}, note that
\begin{align*}
|\nabla K_0^{\ell}(x, \vec{y})| 
\lesssim 2^{\ell(m+1+2d)} \mathbf{1}_{\{|x-\vec{y}| \le 2^{1-j_{\ell}}\}},
\end{align*}
which together with the mean value theorem gives
\begin{align}\label{TEL-1}
\| S_0^{\ell} \|_{L^{p_1} \times L^{p_2} \to L^{p_0}}
& \lesssim |h| 2^{\ell [m+1+2d - (2d+1) (\rho - \varepsilon)]} 
\, 2^{j_{\ell} d (\frac1p - \frac{1}{p_0})},
\end{align}
for all $|h| \le 1$ and $0 < \frac{1}{p_0} \le \frac1p = \frac{1}{p_1} + \frac{1}{p_2}$ with $p_1, p_2 \in [1, \infty]$, where 
\begin{align*}
S_0^{\ell} \in \big\{\Delta_{2^{- j_{\ell}} h} T_0^{\ell}, 
\, T_0^{\ell}(\Delta_{2^{- j_{\ell}} h} \otimes \Id), \, T_0^{\ell}(\Id \otimes \Delta_{2^{- j_{\ell}} h}) \big\}.
\end{align*}
Accordingly, \eqref{TEL-3} and \eqref{TEL-1} imply 
\begin{align*}
\|\T_j^{\ell}\|_{L^{\bar{r}_1} \times L^{\bar{r}_2} \to L^{\bar{s}}}
& \lesssim 2^{(2d+1) \ell} 2^{- d j \Gamma}, \quad \forall j \in \Z,
\\
\|\T_j^{\ell}\|_{L^{\bar{r}_1} \times L^{\infty} \to L^{\bar{r}_1}}
& \lesssim 2^{(2d+1) \ell}, \quad \forall j \in \Z,
\\
\|\T_j^{\ell}\|_{L^{\infty} \times L^{\bar{r}_2} \to L^{\bar{r}_2}}
& \lesssim 2^{(2d+1) \ell}, \quad \forall j \in \Z,
\end{align*}
which justifies \eqref{H5} with $A_5^{\ell} \simeq 2^{(2d+1) \ell}$.

Now by Theorem \ref{thm:main}, one has
\begin{align}\label{BJTJ-1}
\sup_{N_1 \le N_2} 
\bigg| \bigg\langle \sum_{j= N_1}^{N_2} \mathbf{T}_j^{\ell}(f_1, f_2), \, f_3 \bigg\rangle \bigg|
\lesssim (\ell + 1) 2^{- \kappa \ell/2} \, \Lambda_{\bar{r}_1, \bar{r}_2, \bar{s}'}^*(f_1, f_2, f_3),
\end{align}
for all $f_1, f_2, f_3 \in L_c^{\infty}(\Rd)$. Observe that for every $J > j_{\ell}$, 
\begin{align*}
\sum_{j=-J}^J \T_j^{\ell} 
= \sum_{j=-j_{\ell}}^J \T_j^{\ell} 
= \sum_{j=0}^{J+j_{\ell}} T_j^{\ell},
\end{align*} 
and
\begin{align*}
\sum_{j=0}^{J+j_{\ell}} K_j^{\ell}(x, y_1, y_2)
= K_{\sigma_{\ell}} (x, y_1, y_2) \phi(2^{-J} (x-y_1, x-y_2)).
\end{align*} 
Suppose that $\supp f_i \subset B(0, 2^{N_0})$ for each $i=1, 2, 3$ and for some $N_0 \ge 1$. Choosing $J \ge N_0 +2$, we see that for all $x, y_1, y_2 \in B(0, 2^{N_0})$, $|2^{-J} (x-y_1, x-y_2)| \le 1$, which gives $\phi(2^{-J} (x-y_1, x-y_2)) = 1$ and 
\begin{align}\label{BJTJ-2}
\bigg\langle \sum_{j=-J}^J \T_j^{\ell}(f_1, f_2), f_3 \bigg\rangle
= \langle T^{\ell} (f_1, f_2), f_3\rangle.
\end{align} 
Consequently, we invoke \eqref{BJTJ-1} and \eqref{BJTJ-2} to conclude 
\begin{align*}
|\langle T_{\sigma}(f_1, f_2), \, f_3 \rangle|
& \le \sum_{\ell \ge 0} |\langle T^{\ell}(f_1, f_2), \, f_3\rangle|
\lesssim \sum_{\ell \ge 0} (\ell + 1) 2^{- \kappa \ell/2} \, \Lambda_{\bar{r}_1, \bar{r}_2, \bar{s}'}^*(f_1, f_2, f_3)
\\
& \lesssim \Lambda_{\bar{r}_1, \bar{r}_2, \bar{s}'}^*(f_1, f_2, f_3)
\le \Lambda_{r_1, r_2, s'}^*(f_1, f_2, f_3),
\end{align*}
which completes the proof of Theorem \ref{thm:pse}.

\subsubsection{{\bf The sufficiency at the critical order. }}
Fix $m = -d(1 - \rho) \Gamma(r_1, r_2, s)$. We use the same notation as those in the preceding subsection. Choose an even function $\psi \in \mathscr{S}(\Rd)$ such that $\mathbf{1}_{B(0, \frac12)} \le \widehat{\psi} \le \mathbf{1}_{B(0, 1)}$. Define 
\begin{align*}
S_{\ell} a(x, \xi, \eta) 
:= \mathcal{F}^{-1}_x \big[\widehat{\psi_{\ell}}(\cdot) \mathcal{F}_x a(\cdot, \xi, \eta) \big](x)
\quad \text{ and } \quad 
\widehat{\psi_{\ell}}(\xi') := \widehat{\psi}(2^{8-\ell} \xi').
\end{align*}
where $\mathcal{F}^{-1}_x$ denotes the Fourier transform with respect to the $x$-variable. Then
\begin{align}\label{eq:psva}
\int_{\Rd} \psi_{\ell}(y) \, dy = 1
\quad \text{ and } \quad 
\int_{\Rd} y^{\alpha} \psi_{\ell}(y) \, dy = 0, \quad \forall |\alpha| \ge 1.
\end{align}
Note that
\begin{align}\label{eq:Sell}
S_{\ell} a(x, \xi, \eta)
= \int_{\Rd} \psi_{\ell}(y) a(x-y, \xi, \eta) \, dy,
\end{align}
and
\begin{align}\label{eq:Taylor}
a(x - y, \xi, \eta)
= \sum_{|\alpha'| < N_0} \frac{(-y)^{\alpha'}}{\alpha'!} \partial_x^{\alpha'} a(x, \xi, \eta)
+ R_{N_0}(x, y, \xi, \eta),
\end{align}
where $N_0 \ge 1$ and
\begin{align*}
R_{N_0}(x, y, \xi, \eta)
:= \sum_{|\alpha'|=N_0} \frac{(-y)^{\alpha'}}{\alpha'!}
\int_0^1 N_0 (1-t)^{N_0-1} \partial_x^{\alpha'} a(x-ty, \xi, \eta) \, dt.
\end{align*}
As a consequence of \eqref{eq:psva}--\eqref{eq:Taylor}, we have for all $\ell \ge 1$,
\begin{align*}
& |\partial_x^{\alpha} \partial_{\xi}^{\beta} \partial_{\eta}^{\gamma} (\Id - S_{\ell}) \sigma_{\ell}(x, \xi, \eta)|
\\
&= |(\Id - S_{\ell}) \partial_x^{\alpha} \partial_{\xi}^{\beta} \partial_{\eta}^{\gamma}\sigma_{\ell}(x, \xi, \eta)|
= \bigg|\int_{\Rd} \psi_{\ell}(y) R_{N_0}(x, y, \xi, \eta) \, dy \bigg|
\\
& \lesssim \max_{|\alpha'| = N_0} \|\partial_x^{\alpha + \alpha'} \partial_{\xi}^{\beta} \partial_{\eta}^{\gamma} 
\sigma_{\ell}(\cdot, \xi, \eta)\|_{L^{\infty}_x}
\int_{\Rd} |y|^{N_0} \psi_{\ell}(y) \, dy
\\
&\lesssim (1+|\xi|+|\eta|)^{m + \delta (|\alpha|+N_0) - \rho (|\beta| + |\gamma|)} 2^{(8-\ell) N_0}
\mathbf{1}_{\{2^{\ell -1} \le |(\xi, \eta)| \le 2^{\ell+1}\}}
\\
&\lesssim (1+|\xi|+|\eta|)^{m - (1-\delta) N_0 + \delta |\alpha| - \rho (|\beta| + |\gamma|)}
\mathbf{1}_{\{2^{\ell -1} \le |(\xi, \eta)| \le 2^{\ell+1}\}},
\end{align*}
which implies
\begin{align}\label{eq:srem}
\sigma_{\mathrm{rem}}
:= \sigma_0 + \sum_{\ell \ge 1}(\Id - S_{\ell}) \sigma_{\ell} 
\in BS_{\rho, \delta}^{m- N_0(1-\delta), N}, \quad \forall N \in \N.
\end{align}
By the sparse bound with the critical order in the preceding subsection, this leads to a sparse domination for $T_{\sigma_{\mathrm{rem}}}$. Thus, it remains to handle the symbol $\sum_{\ell \ge 1} S_{\ell} \sigma_{\ell}$.

Take $\varphi \in \mathscr{C}_c^{\infty}(\Rd)$ such that $\mathbf{1}_{B(0, 1)} \le \varphi \le \mathbf{1}_{B(0, 2)}$. Since $\supp (S_{\ell} \sigma_{\ell}) \subset \{2^{\ell-1} \le |(\xi, \eta)| \le 2^{\ell+1}\}$, we decompose 
\begin{align*}
S_{\ell} \sigma_{\ell} 
= \widetilde{\sigma}_{\ell}^1+ \widetilde{\sigma}_{\ell}^2 + \widetilde{\sigma}_{\ell}^3, 
\end{align*}
where
\begin{align*}
\widetilde{\sigma}_{\ell}^1 
& := S_{\ell} \sigma_{\ell} \, (1 - \varphi(2^{8-\ell} \xi))(1 - \varphi(2^{8-\ell} \eta)) 
\in BS_{\rho, \delta}^{m, N},
\\
\widetilde{\sigma}_{\ell}^2 
& := S_{\ell} \sigma_{\ell} \, (1 - \varphi(2^{8-\ell} \xi)) \varphi(2^{8-\ell} \eta)
\in BS_{\rho, \delta}^{m, N},
\\
\widetilde{\sigma}_{\ell}^3 
& := S_{\ell} \sigma_{\ell} \, \varphi(2^{8-\ell} \xi)(1 - \varphi(2^{8-\ell} \eta))
\in BS_{\rho, \delta}^{m, N},
\end{align*}
for any $N \in \N$. Define 
\begin{align*}
\widehat{P_{\ell} f}(\xi) := p(2^{-\ell} \xi) \widehat{f}(\xi) 
\quad \text{ and } \quad 
p(\xi) := \varphi(\xi/3) [1 - \varphi(2^9 \xi)],
\end{align*}
which satisfies $p(\xi) = 1$ for all $2^{-8} \le |\xi| \le 3$ and $\supp p \subset \{2^{-9} \le |\xi| \le 6\}$. Then
\begin{equation}\label{eq:PEL}
\begin{aligned}
\langle T_{\widetilde{\sigma}_{\ell}^1}(f_1, f_2), f_3 \rangle
& = \langle T_{\widetilde{\sigma}_{\ell}^1}(P_{\ell} f_1, P_{\ell} f_2), f_3 \rangle,
\\
\langle T_{\widetilde{\sigma}_{\ell}^2}(f_1, f_2), f_3 \rangle
& = \langle T_{\widetilde{\sigma}_{\ell}^2}(P_{\ell} f_1, f_2), P_{\ell} f_3) \rangle,
\\
\langle T_{\widetilde{\sigma}_{\ell}^3}(f_1, f_2), f_3 \rangle
& = \langle T_{\widetilde{\sigma}_{\ell}^3}(f_1, P_{\ell} f_2), P_{\ell} f_3 \rangle.
\end{aligned}
\end{equation}
Let $K_{\widetilde{\sigma}_{\ell}^i}$ be the kernel of $T_{\widetilde{\sigma}_{\ell}^i}$, and set 
\begin{align*}
K_{\widetilde{\sigma}_{\ell, j}^i}(x, y_1, y_2) 
:= \phi_j \big(2^{\rho \ell} (x - y_1, x - y_2) \big)
K_{\widetilde{\sigma}_{\ell}^i}(x, y_1, y_2).
\end{align*}
We claim that 
\begin{align}\label{eq:cwds}
\|T_{\widetilde{\sigma}_{\ell, j}^i}\|_{L^{\bar{r}_1} \times L^{\bar{r}_2} \to L^{\bar{s}}}
\lesssim 2^{\ell \rho d \Gamma} 2^{-2Aj}, \quad \forall i=1, 2, 3, \, j \ge 0.
\end{align}
If we let $K_a$ be the kernel of the pseudo-differential operator $T_a$, then
\begin{align*}
T_a^{b, t}(\vec{f})(x) \
:= & \int_{\R^{2d}} b(t(x-y_1, x-y_2)) K_a(x, y_1, y_2) f_1(y_1) f_2(y_2) \, d\vec{y}
\\
= & \int_{\R^{2d}} \widehat b(\xi_1, \xi_2) e^{2\pi i t x \cdot(\xi_1 + \xi_2)} 
T_a \big(e^{-2\pi it\,\cdot\,\xi_1}f_1, e^{-2\pi it\,\cdot\,\xi_2}f_2 \big)(x) \, d\xi_1 \, d\xi_2,
\end{align*}
which along with Minkowski's inequality gives 
\begin{align}\label{eq:Tabt}
\sup_{t>0} \|T_a^{b, t}\|_{L^{p_1} \times L^{p_2} \to L^{p_0}} 
\le \|\widehat{b}\|_{L^1} \|T_a\|_{L^{p_1} \times L^{p_2} \to L^{p_0}}, \quad \forall p_0 \ge 1.
\end{align}
Then it follows from \eqref{eq:Tabt} applied to $b = \phi$ and $t = 2^{\rho \ell}$ and Lemma \ref{lem:kmt} that
\begin{align*}
\|T_{\widetilde{\sigma}_{\ell, 0}^i}\|_{L^{\bar{r}_1} \times L^{\bar{r}_2} \to L^{\bar{s}}}
\lesssim \|T_{\widetilde{\sigma}_{\ell}^i}\|_{L^{\bar{r}_1} \times L^{\bar{r}_2} \to L^{\bar{s}}}
\lesssim 2^{\ell \rho d \Gamma}.
\end{align*}
If $j \ge 1$, we rewrite 
\begin{align*}
K_{\widetilde{\sigma}_{\ell, j}^i}(x, y_1, y_2) 
& = 2^{2A(\rho \ell -j)} b \big(2^{\rho \ell -j} (x-y_1, x-y_2) \big) 
|(x-y_1, x-y_2)|^{2A} K_{\widetilde{\sigma}_{\ell}^i}(x, y_1, y_2)
\\
& = 2^{2A(\rho \ell -j)} b \big(2^{\rho \ell -j} (x-y_1, x-y_2) \big) 
(-4 \pi^2)^{-A} K_{\Delta_{\xi, \eta}^A \widetilde{\sigma}_{\ell}^i}(x, y_1, y_2),
\end{align*}
where $A \ge 1$ will be determined later and $b(z) = |z|^{-2A} [\phi_0(z) - \phi_0(2z)] \in \mathscr{C}_c^{\infty}(\Rd)$. Note that $\Delta_{\xi, \eta}^A \widetilde{\sigma}_{\ell}^i 
\in BS_{\rho, \delta}^{m-2A \rho, N}$. Thus, \eqref{eq:Tabt} and Lemma \ref{lem:kmt} imply
\begin{align*}
\|T_{\widetilde{\sigma}_{\ell, j}^i}\|_{L^{\bar{r}_1} \times L^{\bar{r}_2} \to L^{\bar{s}}}
\lesssim 2^{2A(\rho \ell -j)} 
\|T_{\Delta_{\xi, \eta}^A \widetilde{\sigma}_{\ell}^i}\|_{L^{\bar{r}_1} \times L^{\bar{r}_2} \to L^{\bar{s}}}
\lesssim 2^{-2Aj} \, 2^{\ell \rho d \Gamma}.
\end{align*}

Partition $\Rd$ into cubes of side length $2^{\lfloor j-\rho \ell \rfloor}$, and denote the resulting family by $\mathcal{Q}_{\ell, j}$. Then we invoke \eqref{eq:PEL} and \eqref{eq:cwds} to deduce 
\begin{align*}
\sum_{\ell \ge 1} |\langle T_{\widetilde{\sigma}_{\ell}^1}(f_1, f_2), f_3 \rangle|
& \le \sum_{\ell \ge 1} \sum_{j \ge 0} \sum_{Q \in \Q_{\ell, j}} 
\big\|T_{\widetilde{\sigma}_{\ell, j}^i} \big((P_{\ell} f_1) \mathbf{1}_{5Q}, 
(P_{\ell}f_2) \mathbf{1}_{5Q} \big) \big\|_{L^{\bar{s}}} 
\|f_3\|_{L^{\bar{s}'}(Q)}
\\
& \lesssim \sum_{j \ge 0} 2^{-2Aj} \sum_{\ell \ge 1}  2^{\ell \rho d \Gamma} 
\sum_{Q \in \Q_{\ell, j}} \|P_{\ell} f_1\|_{L^{\bar{r}_1}(5Q)} 
\|P_{\ell} f_2\|_{L^{\bar{r}_2}(5Q)}
\|f_3\|_{L^{\bar{s}'}(Q)}
\\
& \lesssim \sum_{j \ge 0} 2^{-j(2A - d \Gamma)} \sum_{\ell \ge 1}
\sum_{Q \in \Q_{\ell, j}} |Q| \langle P_{\ell} f_1\rangle_{5Q, \bar{r}_1} 
\langle P_{\ell} f_2 \rangle_{5Q, \bar{r}_2} 
\langle f_3 \rangle_{5Q, \bar{s}'}.
\end{align*}
By \cite[Remark 3.2]{LN}, there exist $3^d$ dyadic lattices $\{\mathscr{D}^i\}_{i=1}^{3^d}$ such that for any cube $Q \subset \Rd$, one can find some $i \in \{1, \ldots, 3^d\}$ and a cube $R \in \mathscr{D}^i$ so that $Q \subset R$ and $\ell(R) \le 3 \ell(Q)$. Given $Q \in \Q_{\ell, j}$, there exists some $R \in \mathscr{D}^i$ so that $5Q \subset R$ and $\ell(R) \le 15 \ell(Q)$. Hence,
\begin{align*}
\sum_{\ell \ge 1} |\langle T_{\widetilde{\sigma}_{\ell}^1}(f_1, f_2), f_3 \rangle|
& \lesssim \sum_{i=1}^{3^d} \sum_{j \ge 0} 2^{-j(2A - d \Gamma)} 
\sum_{\ell \ge 1}
\sum_{k_{\ell} = 2^{\lfloor j-\rho \ell \rfloor+2}}^{2^{\lfloor j-\rho \ell \rfloor+4}} 
\sum_{\substack{R \in \mathscr{D}^i \\ \ell(R) = 2^{k_{\ell}}}} 
|R| \prod_{k=1}^2 \langle P_{\ell} f_k \rangle_{R, \bar{r}_k} 
\langle f_3 \rangle_{R, \bar{s}'}.
\end{align*}
which along with Lemma~\ref{lem:pp} gives
\begin{align*}
\sum_{\ell \ge 1} |\langle T_{\widetilde{\sigma}_{\ell}^1}(f_1, f_2), f_3 \rangle|
\lesssim \sum_{j \ge 0} 2^{-j(2A - d \Gamma)} 
\Lambda^*_{\bar{r}_1, \bar{r}_2, \bar{s}'}(\vec{f}) 
\lesssim \Lambda^*_{r_1, r_2, s'}(\vec{f}),
\end{align*}
provide $2A > d \Gamma$. Similarly, one can obtain sparse domination for $\sum_{\ell \ge 1} T_{\widetilde{\sigma}_{\ell}^i}$, $i=2, 3$. This further implies the sparse bound for $T_{\sum_{\ell \ge 1} S_{\ell} \sigma_{\ell}}$ as desired.

\subsection{Bilinear oscillatory Fourier multipliers}\label{sec:osc}
Let us prove Theorem \ref{thm:osc}. Pick $d \max\{1, \frac2r - \frac12 \} < \kappa < b/a$ and $\kappa_1 = \kappa_2 = \kappa/2$ so that  
\begin{align}
\frac{\kappa_i}{d} > \frac12 > \frac1r - \frac12, \quad i=1, 2.
\end{align}
Choose a nonnegative radial function $\omega \in \mathscr{C}_c^{\infty}(\R^{2d})\) satisfying $\supp \omega \subset \{1/2 \le |\zeta| \le 2\}$ and $ \sum_{n \in \Z} \omega(2^{-n} \zeta) = 1$ for all $\zeta \ne 0$. Then
\begin{equation}\label{eq:app-osc-n-decomposition}
m^{\mathbf c}_{a, b} = \sum_{n \ge 0} m_n,
\quad \text{where} \quad
m_n(\zeta) := \sum_{k \in \Z} c_k \, m_{a, b}(2^k \zeta) \omega(2^{k-n} \zeta).
\end{equation}

Let $U \in \mathscr{C}_c^{\infty}(\R^{2d})$ with $\supp U \subset \{|\zeta| \le 1\}$ and $\widehat{U}(0) \ne 0$. Set 
$U_{\delta}(\zeta) := \delta^{-2d} U(\delta^{-1} \zeta)$. Choose $\delta \in (0, 1/2)$ small enough that $\widehat{U}(\delta \zeta)$ does not vanish on $\{1/4 \le |\zeta| \le 4\}$. Let $\varphi := (- \Delta)^N U_{\delta}$, where $N > 2d + \kappa$. Then $\widehat{\varphi}(\zeta) = (4 \pi^2 |\zeta|^2)^N \widehat{U}(\delta \zeta)$. Define, with smooth zero extension, 
\begin{align*}
\widetilde{\sigma}_{n, t}
:= \psi m_n(t\cdot) /\widehat{\varphi}^2
\quad \text{and} \quad
\widetilde{K}_{n, t} := \mathcal{F}^{-1} \widetilde{\sigma}_{n, t}.
\end{align*}
Let $\sum_{\ell \ge 0} \Psi_{\ell} = 1$ be a smooth spatial partition on $\R^{2d}$, where $\Psi_{\ell} \in \mathscr{C}_c^{\infty}(\R^{2d})$, $\supp \Psi_0 \subset \{|(u, v)| \le 1 \}$, $\supp \Psi_{\ell} \subset \{2^{\ell-1} \le |(u, v)| \le 2^{\ell+1} \}$, and $\sup_{\ell \ge 0} (\|\Psi_{\ell}\|_{L^{\infty}} + \|\widehat{\Psi}_{\ell}\|_{L^1}) < \infty$.
Define
\begin{align*}
\widetilde{\sigma}_{n, t}^{\ell}
& := \mathcal{F} \big[\widetilde{K}_{n, t} \Psi_{\ell} \big], \quad
\sigma_{n, t}^{\ell}
:= \widehat{\varphi}^2 \widetilde{\sigma}_{n, t}^{\ell}, \quad \text{and} \quad
m_{n, j}^{\ell}(\zeta)
:= \sigma_{n, 2^j}^{\ell}(2^{-j} \zeta).
\end{align*}
By construction, 
\begin{equation*}
\sum_{\ell \ge 0} m_{n, j}^{\ell}
= \psi(2^{-j} \cdot) m_n \quad\text{in }\mathscr{S}(\R^{2d}).
\end{equation*}
Note that $\mathcal{F}^{-1} (\sigma_{n, t}^{\ell}) = \varphi*\varphi* \big(\widetilde{K}_{n, t} \Psi_{\ell} \big)$ and 
\begin{align*}
\supp \mathcal{F}^{-1}(\sigma_{n, 2^j}^{\ell}) 
\subset \supp \varphi + \supp \varphi + \supp \big(\widetilde{K}_{n, 2^j} \Psi_{\ell} \big) 
\subset \{|\zeta| \le 3 \cdot 2^{\ell}\},
\end{align*}
provide $\delta$ small enough. Then the fact $\mathcal{F}^{-1} (m_{n, j}^{\ell})(\zeta) = 2^{2j d} \mathcal{F}^{-1} (\sigma_{n, 2^j}^{\ell})(2^j \zeta)$ implies
\begin{align}\label{Fmn}
\supp \mathcal{F}^{-1}(m_{n, 2^j}^{\ell}) 
\subset \{|\zeta| \le 3 \cdot 2^{\ell-j}\}.
\end{align}
For a fixed $(n, \ell)$, the map $j \mapsto J = \ell-j$ is a bijective of $\Z$. Set $T_J^{n, \ell} := T_{m_{n, j}^{\ell}}$
We are going to check the family $\{T_J^{n, \ell}\}_{J \in \Z}$ satisfies the hypotheses \eqref{H1}--\eqref{H5} with exponents $(r_1, r_2, s, s_1, s_2, r_1^*, r_2^*) = (r, r, r', 2r', 2r', \infty, \infty)$.

Obviously, \eqref{H1} follows from \eqref{Fmn}. To verify \eqref{H2} and \eqref{H3}, let us recall a Fourier multiplier theorem from \cite[Theorem~1.1]{GMNT}:
\begin{align}\label{product-H2}
\|T_{\sigma}\|_{L^{p_1} \times L^{p_2} \to L^p}
\lesssim \sup_{j \in \Z} \|\psi \sigma(2^j \cdot)\|_{W^{(t_1, t_2)}},
\end{align}
for all $\frac1p = \frac{1}{p_1} + \frac{1}{p_2}$ with $1<p_1, p_2 < \infty$, $t_1, t_2 > \frac{d}{2}$, and $\sum_{i \in I} \big(\frac{t_i}{d} - \frac{1}{p_i}) > -\frac12$ for each $\varnothing \ne I \subset \{1, 2\}$. Here $W^{(t_1, t_2)}$ is a product-type-Sobolev space defined by
\begin{equation*}
\|\sigma\|_{W^{(t_1, t_2)}}
:= \|\widehat{\sigma}(u, v) \langle u \rangle^{t_1} \langle v \rangle^{t_2} \|_{L^2},
\end{equation*}
where $\langle x \rangle = (1+|x|^2)^{1/2}$. In light of Lemma \ref{lem:transfer}, the identity 
\begin{align*}
\psi(\zeta) \sum_{J=J_1}^{J_2} m_{n, \ell-J}^{\ell}(t \zeta)
= \psi(\zeta) \sum_{i = J_1 - \ell}^{J_2 - \ell}
\widehat{\varphi}(2^i t \zeta)^2 \widetilde{\sigma}_{n, 2^{-i}}^{\ell}(2^i t \zeta)
\end{align*}
gives
\begin{align*}
\sup_{t > 0} & \bigg\| \psi \sum_{J=J_1}^{J_2} m_{n, \ell-J}^{\ell}(t \cdot) \bigg\|_{W^{(\kappa_1, \kappa_2)}}
\lesssim \sup_{t>0} \|\widetilde{\sigma}_{n, t}^{\ell}\|_{W^{(\kappa_1, \kappa_2)}}
\\
& \lesssim 2^{\ell(d+\kappa)} \sup_{t>0} \|\widetilde{K}_{n, t} \Psi_{\ell}\|_{L^{\infty}}
:= 2^{\ell(d+\kappa)} \mathscr{G}_{n, \ell}.
\end{align*}
Combining this with \eqref{product-H2}, we obtain
\begin{align*}
& \sup_{J_1 \le J_2} \bigg\| \sum_{J=J_1}^{J_2} 
T_J^{n, \ell} \bigg\|_{L^r \times L^r \to L^{r/2, \infty}}
\lesssim 2^{\ell(d+\kappa)} \mathscr{G}_{n, \ell},
\\
& \sup_{J_1\le J_2} \bigg\| \sum_{J=J_1}^{J_2} 
T_J^{n, \ell} \bigg\|_{L^{2r',1} \times L^{2r', 1} \to L^{r'}}
\lesssim 2^{\ell(d+\kappa)} \mathscr{G}_{n, \ell}.
\end{align*}
This shows both \eqref{H2} and \eqref{H3} hold with constant $A_2^{n, \ell} = A_3^{n, \ell} := 2^{\ell(d+\kappa)} \mathscr{G}_{n, \ell}$.

Let us next turn to the single-scale estimates. We claim that
\begin{align}
\label{eq:osc-11}
\sup_{t>0} \| T_{\widehat{\varphi} \widetilde{\sigma}_{n, t}^{\ell}} \|_{L^1 \times L^1 \to L^{\infty}}
& \lesssim \mathscr{G}_{n, \ell},
\\
\label{eq:osc-22}
\sup_{t>0} \|T_{\widehat{\varphi} \widetilde{\sigma}_{n, t}^{\ell}} \|_{L^2 \times L^2 \to L^2}
& \lesssim 2^{-nb}.
\end{align}
Indeed, \eqref{eq:osc-11} follows from Young's inequality and the fact $\mathcal F^{-1} (\widehat{\varphi} \widetilde{\sigma}_{n, t}^{\ell}) = \varphi*(\widetilde{K}_{n, t} \Psi_{\ell})$. Besides, \eqref{eq:osc-22} is a consequence of Lemma \ref{lem:TKfg} part \eqref{TKfg-3} and the estimates $\widehat{\varphi} \in \mathscr{S}(\R^{2d})$ and $\sup_{t>0} \|\widetilde{\sigma}_{n, t}^{\ell}\|_{L^{\infty}} \lesssim 2^{-nb}$.

Set $\beta_0 := \frac2r - 1 \in (0, 1)$ and $\beta_1 := \frac3r - 1$. Interpolating \eqref{eq:osc-11} and \eqref{eq:osc-22} yields
\begin{equation}\label{eq:TGn}
\sup_{t>0} \| T_{\widehat{\varphi} \widetilde{\sigma}_{n, t}^{\ell}}\|_{L^r \times L^r \to L^{r'}}
\lesssim \mathscr{G}_{n, \ell}^{\beta_0} \, 2^{-nb (1 - \beta_0)}.
\end{equation}
Observe that 
\begin{align*}
T_{\sigma_1 \sigma_2}(f, g)(x)
= \int_{\R^{2d}} (\mathcal{F}^{-1}{\sigma}_1)(y, z) 
T_{\sigma_2}(\tau_{-y}f, \tau_{-z} g)(x) \, dy \, dz,
\end{align*}
which along with Minkowski's inequality gives
\begin{align}\label{Tsisi}
\|T_{\sigma_1 \sigma_2}\|_{L^{p_1} \times L^{p_2} \to L^{p_0}}
\le \|\mathcal{F}^{-1}{\sigma}_1\|_{L^1}
\|T_{\sigma_2}\|_{L^{p_1} \times L^{p_2} \to L^{p_0}},
\end{align}
for all $p_0 \ge 1$. Recall that $\sigma_{n, t}^{\ell} = \widehat{\varphi} (\widehat{\varphi} \widetilde{\sigma}_{n, t}^{\ell})$. Using \eqref{eq:TGn} and \eqref{Tsisi}, we have 
\begin{align*}
\sup_{t>0} \| T_{\sigma_{n, t}^{\ell}}\|_{L^r \times L^r \to L^{r'}}
\le \sup_{t>0} \|\varphi\|_{L^1} 
\| T_{\widehat{\varphi} \widetilde{\sigma}_{n, t}^{\ell}}\|_{L^r \times L^r \to L^{r'}}
\lesssim \mathscr{G}_{n, \ell}^{\beta_0} \, 2^{-nb (1 - \beta_0)}.
\end{align*}
In view of Lemma \ref{lem:TKfg} part \eqref{TKfg-1}, one has
\begin{align*}
& \| T_{\sigma_{n, t}^{\ell}}\|_{L^r \times L^{\infty} \to L^r}
+ \| T_{\sigma_{n, t}^{\ell}}\|_{L^{\infty} \times L^r \to L^r}
\lesssim \|\mathcal{F}^{-1}(\sigma_{n, t}^{\ell})\|_{L^1} 
\\
& = \|\varphi*\varphi* (\widetilde{K}_{n, t} \Psi_{\ell})\|_{L^1}
\le \|\varphi\|_{L^1}^2 \|\widetilde{K}_{n, t} \Psi_{\ell}\|_{L^1}
\lesssim 2^{2\ell d} \mathscr{G}_{n, \ell}.
\end{align*}
Then by the fact $m_{n, j}^{\ell} = \sigma_{n, 2^j}^{\ell}(2^{-j} \cdot)$ and change of variables, the estimates above give
\begin{align*}
\|T_J^{n, \ell}\|_{L^r \times L^r \to L^{r'}}
= 2^{jd (\frac1r + \frac1r - \frac{1}{r'})} 
\| T_{\sigma_{n, 2^j}^{\ell}} \|_{L^r \times L^r \to L^{r'}}
\lesssim 2^{-Jd (\frac1r + \frac1r - \frac{1}{r'})} 
2^{\ell d \beta_1} 2^{-nb (1 - \beta_0)} \mathscr{G}_{n,\ell}^{\beta_0},
\end{align*}
and
\begin{align*}
\| T_J^{n, \ell} \|_{L^r \times L^{\infty} \to L^r}
+ \| T_J^{n, \ell} \|_{L^{\infty} \times L^r \to L^r}
\lesssim 2^{2\ell d} \mathscr{G}_{n, \ell}.
\end{align*}
This proves \eqref{H4} with constant $A_4^{n, \ell} := 2^{\ell d \beta_1} 2^{-nb (1 - \beta_0)} \mathscr{G}_{n, \ell}^{\beta_0} + 2^{2\ell d} \mathscr{G}_{n, \ell}$.  

To verify \eqref{H5}, fix $|h| \le 1$ and let $v = 2^{\ell} h$. Denote 
\begin{align*}
& \rho_0 (\xi, \eta) := \big(e^{2 \pi i v \cdot (\xi + \eta)} - 1 \big) \widehat{\varphi}(\xi, \eta), 
\\ 
& \rho_1 (\xi, \eta) := \big(e^{2 \pi i v \cdot \xi} - 1 \big) \widehat{\varphi}(\xi, \eta), 
\\
& \rho_2 (\xi, \eta) := \big(e^{2 \pi i v \cdot \eta} - 1 \big) \widehat{\varphi}(\xi, \eta),
\\
& S_0 := \Delta_{2^J h}T_J^{n, \ell}, \, \,  
S_k := T_J^{n, \ell} \circ_i \Delta_{2^J h}, \, \, k=1, 2.
\end{align*}
Then for any $0 < \vartheta \le 1$, 
\begin{align}\label{Frho}
\| \mathcal{F}^{-1} \rho_k\|_{L^1}
\lesssim \min\{1, |v|\}
\le 2^{\vartheta \ell} |h|^{\vartheta}, \quad k=0, 1, 2.
\end{align}
The symbols of $S_1, S_2, S_3$ are respectively given by 
\begin{align*}
\big(e^{2 \pi i 2^J h \cdot (\xi + \eta)} - 1 \big) m_{n, j}^{\ell}(\xi, \eta)
& = (\rho_0 \widehat{\varphi} \widetilde{\sigma}_{n, 2^j}^{\ell})(2^{-j}(\xi, \eta)),
\\
\big(e^{2 \pi i 2^J h \cdot \xi} - 1 \big) m_{n, j}^{\ell}(\xi, \eta)
& = (\rho_1 \widehat{\varphi} \widetilde{\sigma}_{n, 2^j}^{\ell})(2^{-j}(\xi, \eta)),
\\
\big(e^{2 \pi i 2^J h \cdot \eta} - 1 \big) m_{n, j}^{\ell}(\xi, \eta)
& = (\rho_2 \widehat{\varphi} \widetilde{\sigma}_{n, 2^j}^{\ell})(2^{-j}(\xi, \eta)),
\end{align*}
which together with \eqref{eq:TGn}, \eqref{Tsisi}, and \eqref{Frho} gives
\begin{align*}
\|S_k\|_{L^r \times L^r \to L^{r'}}
& = 2^{jd (\frac1r + \frac1r - \frac{1}{r'})} 
\| T_{\rho_k \widehat{\varphi} \widetilde{\sigma}_{n, 2^j}^{\ell}} \|_{L^r \times L^r \to L^{r'}}
\\
& \lesssim 2^{jd (\frac1r + \frac1r - \frac{1}{r'})} 2^{\vartheta \ell} |h|^{\vartheta} 
\mathscr{G}_{n, \ell}^{\beta_0} \, 2^{-nb (1 - \beta_0)}
\\
& = |h|^{\vartheta} 2^{-Jd (\frac1r + \frac1r - \frac{1}{r'})} 2^{\vartheta \ell} 2^{\ell d \beta_1} 
2^{-nb (1 - \beta_0)} \mathscr{G}_{n,\ell}^{\beta_0}.
\end{align*}
In addition, by Lemma \ref{lem:TKfg} part \eqref{TKfg-1} and \eqref{Frho},
\begin{align*}
& \|S_k\|_{L^r \times L^{\infty} \to L^r}
+ \|S_k\|_{L^{\infty} \times L^r \to L^r}
\\
& \lesssim \|\mathcal{F}^{-1}(\rho_k \widehat{\varphi} \widetilde{\sigma}_{n, 2^j}^{\ell})\|_{L^1} 
= \|(\mathcal{F}^{-1} \rho_k)*\varphi* (\widetilde{K}_{n, t} \Psi_{\ell})\|_{L^1}
\\
& \le \|\mathcal{F}^{-1} \rho_k\|_{L^1} \|\varphi\|_{L^1} \|\widetilde{K}_{n, t} \Psi_{\ell}\|_{L^1}
\lesssim |h|^{\vartheta} 2^{\vartheta \ell} 2^{2\ell d} \mathscr{G}_{n, \ell}.
\end{align*}
Consequently, all three estimates in \eqref{H5} hold with constant $A_5^{n, \ell} = 2^{\vartheta \ell} A_4^{n, \ell}$.

Invoking Theorem~\ref{thm:main}, we arrive at
\begin{equation}\label{eq:JJne}
\sup_{J_1 \le J_2} \bigg| \bigg\langle \sum_{J_1 \le J \le J_2} T_J^{n, \ell}(f, g), h \bigg\rangle \bigg|
\lesssim A_0^{n, \ell} \, \Lambda_{(r, r, r)}^* (f, g, h),
\end{equation}
for all functions $f, g, h \in L_c^{\infty}(\Rd)$, where
\begin{align*}
A_0^{n, \ell}
:= 2^{\ell(d + \kappa)} \mathscr{G}_{n, \ell}
+ (\ell+1) \big(2^{\ell d \beta_1} 2^{-nb (1 - \beta_0)} \mathscr{G}_{n, \ell}^{\beta_0} 
+ 2^{2\ell d} \mathscr{G}_{n, \ell} \big).
\end{align*}
Recall that $d < \kappa < b/a$. Pick $0 < \varepsilon < \min\{1, b/a - d\}$. Then by Lemma \ref{lem:Knt},
\begin{align}\label{eq:nege}
\sum_{n, \ell \ge 0} A_0^{n, \ell}
& = \sum_{n, \ell \ge 0} 2^{-n b}  2^{\ell(d + \kappa)} \mathscr{H}_{n, \ell}
+ \sum_{n, \ell \ge 0} 2^{-n b}  (\ell+1) 2^{-\ell \varepsilon} 
2^{\ell \varepsilon} 2^{\ell d \beta_1} \mathscr{H}_{n, \ell}^{\beta_0} 
\\ \notag
& \quad + \sum_{n, \ell \ge 0} 2^{-n b}  (\ell+1) 2^{-\ell \varepsilon} 
2^{\ell \varepsilon} 2^{2\ell d} \mathscr{H}_{n, \ell}
\\ \notag
& \lesssim \sum_{n \ge 0} 2^{-n (b - a \kappa)} 
+ \sum_{n \ge 0} 2^{-n [b- a(\varepsilon + d \beta_1 - d \beta_0)]}
+ \sum_{n \ge 0} 2^{-n [b- a(\varepsilon + d)]}
\\ \notag
& \lesssim 1.
\end{align}
Thus, \eqref{eq:JJne} and \eqref{eq:nege} imply 
\begin{equation}\label{eq:JJsp}
\sup_{J_1 \le J_2} \bigg| \bigg\langle \sum_{n, \ell \ge 0} 
\sum_{J_1 \le J \le J_2} T_J^{n, \ell}(f, g), h \bigg\rangle \bigg|
\lesssim \Lambda_{(r, r, r)}^* (f, g, h),
\end{equation}
for all functions $f, g, h \in L_c^{\infty}(\Rd)$.

It remains to transfer \eqref{eq:JJsp} to the original operator $T_{\m^{\c}_{a, b}}$. By the fact $\widehat{\varphi}(0) = 0$ and the fundamental theorem of calculus, 
\begin{equation*}
|\widehat{\varphi} (\zeta)| \lesssim \min\{|\zeta|, |\zeta|^{-1}\}, \qquad \zeta \ne 0,
\end{equation*}
which together with Lemma \ref{lem:Knt} gives
\begin{align*}
& \sum_{n, \ell \ge 0} \sum_{j \in \Z}|m_{n, j}^{\ell}(\zeta)|
\le \sum_{n, \ell \ge 0} \sup_{t>0} \|\widehat{\widetilde{K}_{n, t} \Psi_{\ell}}\|_{L^{\infty}} 
\sum_{j \in \Z} |\widehat{\varphi}(2^{-j} \zeta)|^2
\\
& \lesssim \sum_{n, \ell \ge 0} \sup_{t>0} \|\widetilde{K}_{n, t} \Psi_{\ell}\|_{L^1} 
\lesssim \sum_{n, \ell \ge 0} 2^{2 d \ell} \mathscr{G}_{n, \ell}
\lesssim \sum_{n \ge 0} 2^{-n(b-ad)} 
\lesssim 1.
\end{align*}
This means 
\begin{align*}
\bm{m}_N := \sum_{n, \ell \ge 0} \sum_{|J| \le N} m_{n, \ell-J}^{\ell} 
\to \m_{a, b}^{\bm{c}}, \quad \text{as } \, N \to \infty,
\end{align*}
which along with Lebesgue’s dominated convergence theorem implies
\begin{align}\label{eq:TmN}
& |\langle T_{\bm{m}_N} (f, g), h \rangle - \langle T_{\m_{a, b}^{\bm{c}}} (f, g), h \rangle| 
\le \|T_{a_N} (f, g) - T_{\m_{a, b}^{\bm{c}}} (f, g)\|_{L^{\infty}} \|h\|_{L^1}
\\ \nonumber
& \le \int_{\R^{2d}} |\bm{m}_N(\xi, \eta) - \m_{a, b}^{\bm{c}}(\xi, \eta)| 
|\widehat{f}(\xi)| |\widehat{g}(\eta)| \, d\xi \, d\eta \, \|h\|_{L^1}
\to 0, \quad \text{as } \, N \to \infty.
\end{align}
Therefore, Theorem \ref{thm:osc} follows from \eqref{eq:JJsp} and \eqref{eq:TmN}.

\begin{lemma}\label{lem:Knt}
There exists constants $0 < c_0 < 1 < c_1 < \infty$ such that for every $N > 0$,
\begin{equation*}
\mathscr{H}_{n, \ell} 
:= 2^{nb} \mathscr{G}_{n, \ell} 
\lesssim_N
\begin{cases}
\lambda_n^{-N}, & 2^{\ell} \le c_0 \lambda_n,
\\
\lambda_n^{-d}, & c_0 \lambda_n < 2^{\ell} < c_1 \lambda_n,
\\
2^{-\ell N}, & 2^{\ell} \ge c_1 \lambda_n.
\end{cases}
\end{equation*}
Hence, for all $q_1, q_2 > 0$,
\begin{equation*}
\sum_{\ell \ge 0} 2^{q_1 \ell} \mathscr{H}_{n, \ell}^{q_2}
\lesssim \lambda_n^{q_1 - dq_2}.
\end{equation*}
\end{lemma}

\begin{proof}
The second estimate follows from the first one and that $N > \max\{d, q_1/q_2\}$. To show the first estimate, we rewrite 
\begin{equation}\label{eq:sym}
\widetilde{\sigma}_{n, t}(\zeta)
= 2^{-nb} \sum_{k \in \mathcal{K}_{n, t}} c_k \, F_{k, n, t}(\zeta) 
e^{i \lambda_n \rho_{k, n, t}^a |\zeta|^a},
\end{equation}
where
\begin{equation*}
F_{k, n, t}(\zeta) 
= \psi(\zeta) \widehat{\varphi}(\zeta)^{-2} \phi(2^n\rho_{k,n,t}\zeta)
\rho_{k, n, t}^{-b} |\zeta|^{-b} \w(\rho_{k, n, t} \zeta).
\end{equation*}
Note that for every multi-index $\alpha$,
\begin{equation}\label{eq:amp}
\supp F_{k, n, t} \subset \{1/2 \le |\zeta| \le 2\}
\quad \text{ and } \quad
\sup_{n, t} \sup_{k \in \mathcal{K}_{n, t}}
\|\partial^{\alpha} F_{k, n, t}\|_{L^{\infty}} < \infty.
\end{equation}
If we define 
\begin{equation*}
I_{\lambda, \rho} f(w)
:= \int_{\R^{2d}} e^{i [2 \pi w \cdot \zeta + \lambda A_\rho (\zeta)]} f(\zeta) \, d\zeta,
\end{equation*}
where $A_{\rho}(\zeta) := \rho^a |\zeta|^a$, then by \eqref{eq:sym},
\begin{align*}
\widetilde{K}_{n, t}(w) 
= \mathcal{F}^{-1} \widetilde{\sigma}_{n, t}(w)
= 2^{-nb} \sum_{k\in\mathcal{K}_{n,t}} c_k \, 
I_{\lambda_n, \rho_{k, n, t}} F_{k, n, t}(w).
\end{align*}
Then \eqref{eq:amp} yield 
\begin{align}
\label{eq:Kgl}
|\widetilde{K}_{n, t}(w)| 
& \lesssim 2^{-nb}\lambda_n^{-d}, \quad\forall w \in \R^{2d},
\\
\label{eq:Kta}
|\widetilde K_{n, t}(w)|
& \lesssim_N 2^{-nb} (\lambda_n + |w|)^{-N}, \quad \forall N>0, 
\end{align}
Now take $c_0 := \min\{1, C_0^-\}/2$ and $c_1 := 2 \max\{1, C_0^+\}$. Recall that $\mathscr{H}_{n, \ell} = 2^{nb} \sup_{t>0} \|\widetilde{K}_{n, t} \Psi_{\ell} \|_{L^{\infty}}$. If $c_0 \lambda_n < 2^{\ell} < c_1 \lambda_n$, then $\mathscr{H}_{n, \ell} \lesssim \lambda_n^{-d}$ due to \eqref{eq:Kgl}. If $2^{\ell} \le c_0 \lambda_n$, then for all $w \in \supp \Psi_{\ell}$, $|w| \le 2^{\ell+1} \le 2 c_0 \lambda_n \le C_0^- \lambda_n$. Thus, \eqref{eq:Kta} implies $\mathscr{H}_{n, \ell} \lesssim  \lambda_n^{-N}$. If $2^{\ell} \ge c_1 \lambda_n$, then for all $w \in \supp \Psi_{\ell}$, $|w| \ge 2^{\ell-1} \ge (c_1/2) \lambda_n \ge C_0^+ \lambda_n$, which along with \eqref{eq:Kta} gives $\mathscr{H}_{n, \ell} \lesssim 2^{- \ell N}$.
\end{proof}

\subsection{Bilinear spherical maximal operators}\label{sec:App-sph}
We turn to the proof of Theorems \ref{thm:sph} and \ref{thm:tri}. The following result reduces sparse domination for maximal operators to the uniform verification of \eqref{H1}--\eqref{H5} for finite $\ell^{\infty}$-valued families encoding dyadic scales and normalized radii.

\begin{lemma}\label{lem:finite}
Let $\{U_t\}_{t > 0}$ be a positive bilinear averaging family. 
\begin{list}{\rm (\theenumi)}{\usecounter{enumi}\leftmargin=1.3cm \labelwidth=1cm \itemsep=0.1cm \topsep=0.2cm \renewcommand{\theenumi}{\alph{enumi}}}

\item\label{Uval} Given nonempty finite subsets $I \subset \Z$ and $E \subset [1, 2]$, let $\B_{I, E} := \ell^{\infty}(I \times E)$. If the $\B_{I, E}$-valued family
\begin{align*}
\mathcal{U}_j(f, g)(x) := e_j U_{2^j s}(f, g)(x) \, \, \text{ if } (j, s) \in I \times E, \qquad
\mathcal{U}_j := 0 \, \, \text{ otherwise},
\end{align*}
where $e_j$ is the standard coordinate vector, satisfies \eqref{H1}--\eqref{H5} uniformly $I$ and $E$, then 
\begin{align*}
\Big\langle \sup_{t>0} |U_t (f_1, f_2)|, \, |f_3| \Big\rangle
\lesssim \Lambda_{r_1, r_2, s'}^{*}(f_1, f_2, f_3),
\end{align*}
for all $f_1, f_2, f_3 \in L_c^{\infty}(\Rd)$.

\item\label{Vval} Given a nonempty finite subset $I \subset \Z$, let $\B_I = \ell^{\infty}(I)$. If the $\B_I$-valued family
\begin{align*}
\mathcal{V}_j(f, g)(x) := e_j U_{2^j}(f, g)(x) \, \, \text{ if } j \in I, \qquad
\mathcal{V}_j := 0 \, \, \text{ if } j \notin I,
\end{align*}
satisfies \eqref{H1}--\eqref{H5} uniformly in $I$, then
\begin{align*}
\Big\langle \sup_{j \in \Z} |U_{2^j}(f_1, f_2)|, \, |f_3| \Big\rangle
\lesssim \Lambda_{r_1, r_2, s'}^*(f_1, f_2, f_3),
\end{align*}
for all $f_1, f_2, f_3 \in L_c^{\infty}(\Rd)$.
\end{list}
\end{lemma}

\begin{proof}
We only focus on part \eqref{Uval} since the proof of part \eqref{Vval} uses a linearization argument similar to that in part \eqref{Uval}. It suffices to treat $f_1, f_2, f_3 \ge 0$ since $|U_t(f_1, f_2)| \le U_t(|f_1|, |f_2|)$. By definition, 
\begin{align*}
& \sum_{j \in I} \mathcal{U}_j (f_1, f_2)(x)
= \big(U_{2^j s}(f_1, f_2)(x)\big)_{(j, s) \in I \times E} \qquad \text{ and } \quad
\\
& \bigg|\sum_{j \in I} \mathcal{U}_j(f_1, f_2)(x) \bigg|_{\B_{I, E}}
= \max_{(j, s) \in I \times E} U_{2^j s}(f_1, f_2)(x)
= U_{2^{j(x)} s(x)}(f_1, f_2)(x),
\end{align*}
for some $(j(x), s(x)) \in I \times E$. Let $e_{j, s}$ denote the coordinate vector whose $(k, t)$ entry is $\delta_{jk} \delta_{st}$. By the fact $\B_{I, E}^* = \ell^1(I \times E)$, 
\begin{align*}
F_3(x) := f_3(x) e_{(j(x), s(x))} \quad \text{ satisfies } \quad
|F_3(x)|_{\B_{I, E}^*} = f_3(x).
\end{align*}
Since $F_3 \in L_c^{\infty}(\Rd; \B_{I, E}^*)$ and $\langle |F_3|_{B_{I, E}^*} \rangle_{Q, s'} = \langle f_3 \rangle_{Q, s'}$, Theorem~\ref{thm:main}, applied to $\B_{I, E}$-valued functions, gives
\begin{align*}
\Big\langle \max_{(j, s) \in I \times E} U_{2^j s}(f_1, f_2), f_3 \Big\rangle
= \bigg\langle \sum_{j \in I} \mathcal{U}_j(f_1, f_2), F_3 \bigg\rangle_{(\B_{I, E}, \B_{I, E}^*)}
\lesssim \Lambda^*_{r_1, r_2, s'}(f_1, f_2, f_3),
\end{align*}
where the implicit constant is independent of $I$ and $E$. Let $I_N = \{-N, -N+1, \ldots, N-1, N\}$ and $E_N = \{r_1, \ldots, r_N\}$, where $\{r_i\}_{i=1}^{\infty}$ is the rational numbers in $[1, 2]$. Then
\begin{align*}
\lim_{N \to \infty} \max_{(j, s) \in I_N \times E_N} |U_{2^j s}(f_1, f_2)(x)|
= \sup_{j \in \Z} \sup_{s \in [1, 2]} |U_{2^j s}(f_1, f_2)(x)|
= \sup_{t>0} |U_t(f_1, f_2)(x)|. 
\end{align*}
By the monotone convergence theorem, we obtain the sparse bound as desired. 
\end{proof}

\subsubsection{\bf{Bilinear spherical maximal operators}}
To show Theorem \ref{thm:sph}, we use Lemma \ref{lem:finite} with $U_t = \A_t$. For the full maximal operator, the support condition \eqref{H1} holds with $A_1 = 2$. For any finite subsets $I \subset \Z$ and $E \subset [1, 2]$, 
\begin{align}\label{FUM}
\bigg|\sum_{j=N_1}^{N_2} \mathcal{U}_j(f_1, f_2)(x) \bigg|_{\B_{I, E}}
= \max_{\substack{j \in I \cap [N_1, N_2] \\ s \in E}} |\A_{2^j s}(f_1, f_2)(x)|
\le \M_{\full}(f_1, f_2)(x),
\end{align}
for all $N_1 \le N_2$. The condition $(\frac{1}{r_1}, \frac{1}{r_2}, \frac1s) \in \mathcal{R}_{\full}(d)$ gives $\frac1r < \frac{2d-1}{d}$.  Then by \cite[Theorem~1.1]{JL},
\begin{align*}
\|\M_{\full}\|_{L^{p_1} \times L^{p_2} \to L^q} \lesssim 1
\quad \text{ and } \quad
\|\M_{\full}\|_{L^{2s} \times L^{2s} \to L^s} \lesssim 1,
\end{align*}
which together with \eqref{FUM} implies \eqref{H2} and \eqref{H3}. 

To obtain improving estimates in \eqref{H4}, note that
\begin{align}\label{Uloc-1}
|\mathcal{U}_j (f_1, f_2)(x)|_{\B_{I, E}}
= \max_{s \in E} |\A_{2^j s}(f_1, f_2)(x)|
\le \sup_{t \in [1, 2]} |\A_{2^j t}(f_1, f_2)(x)|.
\end{align}
In addition, the localized maximal estimates in \cite[Proposition~3.2]{JL} assert
\begin{align}\label{Uloc-2}
\Big\|\sup_{t \in [1, 2]} |\A_t| \Big\|_{L^{p_1} \times L^{p_2} \to L^{p_0}} \lesssim 1,
\quad \forall (\tfrac{1}{p_1}, \tfrac{1}{p_2}, \tfrac{1}{p_0}) \in \mathcal{R}_{\full}(d).
\end{align}
By the fact 
\begin{align}\label{Atsca}
\A_{2^j t}(f_1, f_2)(2^j x) = \A_t \big(f_1(2^j \cdot), f_2(2^j \cdot) \big)(x),
\end{align} 
we invoke \eqref{Uloc-1} and \eqref{Uloc-2} to deduce
\begin{align}\label{Uloc-3}
\|\mathcal{U}_j\|_{L^{p_1} \times L^{p_2} \to L^{p_0}(\B_{I, E})} 
\lesssim 2^{-jd(\frac{1}{p_1} + \frac{1}{p_2} - \frac{1}{p_0})},
\quad \forall (\tfrac{1}{p_1}, \tfrac{1}{p_2}, \tfrac{1}{p_0}) \in \mathcal{R}_{\full}(d).
\end{align}
Recall that $(\frac{1}{r_1}, \frac{1}{r_2}, \frac1s) \in \mathcal{R}_{\full}(d)$. Then choose $r_1^*, r_2^* \in (1, \infty)$ sufficiently large so that 
\begin{align}\label{rrRd}
(\tfrac{1}{r_1}, \tfrac{1}{r_2^*}, \tfrac{1}{r_1}) \in \mathcal{R}_{\full}(d)
\quad \text{ and } \quad
(\tfrac{1}{r_1^*}, \tfrac{1}{r_2}, \tfrac{1}{r_2}) \in \mathcal{R}_{\full}(d),
\end{align}
which is possible because $m(d, r_i) > 1 > \frac{1}{r_i}$ for each $i=1, 2$. Hence, \eqref{H4} follows from \eqref{Uloc-3} applied to the triples of exponents $(r_1, r_2, s)$, $(r_1, r_2^*, r_1)$, and $(r_1^*, r_2, r_2)$.

To prove \eqref{H5}, we use the localized continuity estimates from \cite[Proposition~4]{BFOPZ}: for all $(\frac{1}{p_1}, \frac{1}{p_2}, \frac{1}{p_0}) \in \mathcal{R}_{\full}(d)$, there exists some $\eta \in (0, 1)$ such that 
\begin{align}
\label{RFU-1}
\Big\|\sup_{t \in [1, 2]} |\A_t (\Delta_h \otimes \Id)| \Big\|_{L^{p_1} \times L^{p_2} \to L^{p_0}}
& \lesssim |h|^{\eta}, \quad\forall |h| \le 1,
\\
\label{RFU-2}
\Big\|\sup_{t \in [1, 2]} |\A_t (\Id \otimes \Delta_h)| \Big\|_{L^{p_1} \times L^{p_2} \to L^{p_0}}
& \lesssim |h|^{\eta}, \quad\forall |h| \le 1,
\\
\label{RFU-3}
\Big\|\sup_{t \in [1, 2]} |\Delta_h \A_t| \Big\|_{L^{p_1} \times L^{p_2} \to L^{p_0}}
& \lesssim |h|^{\eta}, \quad\forall |h| \le 1.
\end{align}
where \eqref{RFU-3} follows from \eqref{RFU-1}, \eqref{RFU-2} and the identity
\begin{align*}
\Delta_h \A_t(f_1, f_2)
= \A_t (\Delta_h f_1, \tau_h f_2) + \A_t (f_1, \Delta_h f_2).
\end{align*}
Then using \eqref{Uloc-1}, \eqref{Atsca} and \eqref{RFU-1}--\eqref{RFU-3}, we achieve that for all $(\frac{1}{p_1}, \frac{1}{p_2}, \frac{1}{p_0}) \in \mathcal{R}_{\full}(d)$,
\begin{align*}
\|\Delta_h \mathcal{U}_j\|_{L^{p_1} \times L^{p_2} \to L^{p_0}(\B_{I, E})} 
& \lesssim |h|^{\eta} 2^{-jd(\frac{1}{p_1} + \frac{1}{p_2} - \frac{1}{p_0})},
\\
\|\mathcal{U}_j (\Delta_h \otimes \Id)\|_{L^{p_1} \times L^{p_2} \to L^{p_0}(\B_{I, E})} 
& \lesssim |h|^{\eta} 2^{-jd(\frac{1}{p_1} + \frac{1}{p_2} - \frac{1}{p_0})},
\\
\|\mathcal{U}_j (\Id \otimes \Delta_h)\|_{L^{p_1} \times L^{p_2} \to L^{p_0}(\B_{I, E})} 
& \lesssim |h|^{\eta} 2^{-jd(\frac{1}{p_1} + \frac{1}{p_2} - \frac{1}{p_0})},
\end{align*}
which along with \eqref{rrRd} implies \eqref{H5} as desired. 

Consequently, the $\B_{I, E}$-valued family $\{\mathcal{U}_j\}_{j \in \Z}$ satisfies \eqref{H1}--\eqref{H5} uniformly in $I$ and $E$. Lemma~\ref{lem:finite} then gives the sparse bound for $\M_{\full}$.

Next, let us turn to the lacunary maximal operator. The argument is similar to the above. The support condition \eqref{H1} holds trivially. For any finite subset $I \subset \Z$ and all $N_1 \le N_2$,
\begin{align*}
\bigg|\sum_{j=N_1}^{N_2} \mathcal{V}_j(f_1, f_2)(x) \bigg|_{\B_I}
= \max_{j \in I \cap [N_1, N_2]} |\A_{2^j}(f_1, f_2)(x)|
\le \M_{\lac}(f_1, f_2)(x).
\end{align*}
Then \eqref{H2} and \eqref{H3} follows from \cite[Theorem~1]{BF}:
\begin{align*}
\|\M_{\lac}\|_{L^{p_1} \times L^{p_2} \to L^p} \lesssim 1,
\end{align*}
for all $\frac1p = \frac{1}{p_1} + \frac{1}{p_2} > 0$ with $p_1, p_2 \in (1, \infty]$. As abefore, to obtain \eqref{H4} and \eqref{H5}, by rescaling and the fact $|\mathcal{V}_j (f_1, f_2)(x)|_{\B_I} = |\A_{2^j}(f_1, f_2)(x)|$,
we are reduced to proving that for all $(\frac{1}{p_1}, \frac{1}{p_2}, \frac{1}{p_0}) \in \mathcal{R}_{\lac}(d)$ with $p_0 > 1$, there exists some $\eta \in (0, 1)$ such that 
\begin{align}
\label{RFV-1}
\|\A_1\|_{L^{p_1} \times L^{p_2} \to L^{p_0}}
& \lesssim 1,
\\
\label{RFV-2}
\|\A_1(\Delta_h \otimes \Id)\|_{L^{p_1} \times L^{p_2} \to L^{p_0}}
& \lesssim |h|^{\eta},
\\
\label{RFV-3}
\|\A_1 (\Id \otimes \Delta_h)\|_{L^{p_1} \times L^{p_2} \to L^{p_0}}
& \lesssim |h|^{\eta},
\end{align}
for all $|h| \le 1$.

We will see that interpolation plays a crucial role in the proof of \eqref{RFV-1}--\eqref{RFV-3}. Fix $X = (\frac{1}{p_1}, \frac{1}{p_2}, \frac{1}{p_0}) \in \mathcal{R}_{\lac}(d)$ with $p_0 > 1$. Using an orthogonal change of coordinates on $\bS^{2d-1}$, we have
\begin{align}\label{AA11}
\|\A_1(f, g)\|_{L^1}
& \le \int_{\Rd} |f(u)| \bigg[\int_{\bS^{2d-1}} |g(u+y-z)| \, d\sigma(y, z)\bigg] \, du
\\ \nonumber
& = \int_{\Rd} \int_{\Rd} |f(u)| \, |g(u+v)| \, \phi(v) \, du \, dv
\lesssim\|f\|_{L^1} \|g\|_{L^1},
\end{align}
where
\begin{align*}
\phi(v) := 
\frac{\Gamma(d)}{(2\pi)^{d/2} \Gamma(d/2)}
(1- |v|^2/2 )^{(d-2)/2} \mathbf{1}_{\{|v| < \sqrt{2}\}}
\lesssim 1.
\end{align*}
Letting
\begin{align*}
P_2 = (\tfrac12, \tfrac12, \tfrac12) 
\quad \text{and} \quad 
P_t = (1-t) P_1 + t P_2 =: (\tfrac{1}{a_t}, \tfrac{1}{b_t}, \tfrac{1}{c_t}), \qquad 0<t<1, 
\end{align*}
we invoke \cite[Theorem 1.5]{BCSS} to deduce
\begin{align}\label{AA12}
\|\A_1\|_{L^{a_t} \times L^{b_t} \to L^{c_t}}
 \lesssim_t 1, \qquad 0 < t < 1.
\end{align}
Recall $P_0 = (1, 1, 1)$. Since $X$ is an interior point, there exists some $t \in (0, 1)$ small enough so that $Y := \frac{X - t P_2}{1-t} \in \conv \big(\mathcal{R}_{\full}(d) \cup \{P_0, P_1\} \big)$. By the convexity of $\mathcal{R}_{\rm full}(d)$ and that $P_2 \in \mathcal{R}_{\rm full}(d)$, there holds $(1-t) \mathcal{R}_{\full}(d) + t P_2 \subset \mathcal{R}_{\full}(d)$, which along with the fact $P_t = (1-t) P_0 + t P_2 \in \conv \{P_0, P_2\}$ gives 
\begin{align*}
X=(1-t)Y + t P_2 \in \conv \big(\mathcal{R}_{\full}(d)\cup\{P_0, P_t\}\big).
\end{align*}
In light of \eqref{Uloc-2}, \eqref{AA11}, and \eqref{AA12}, this implies that there are finitely many triples $\{X_j = (\frac{1}{a_j}, \frac{1}{b_j}, \frac{1}{c_j})\}_{j=1}^N$ such that
\begin{align*}
& \bigg(\frac{1}{p_1}, \frac{1}{p_2}, \frac{1}{p_0} \bigg)
= X = \sum_{j=1}^N \lambda_j X_j 
= \sum_{j=1}^N \lambda_j \bigg(\frac{1}{a_j}, \frac{1}{b_j}, \frac{1}{c_j} \bigg),
\\
& \lambda_j > 0, \quad \sum_{j=1}^N \lambda_j = 1, \, \, \text{and} \, \,
\|\A_1\|_{L^{a_j} \times L^{b_j} \to L^{c_j}}
\le C_j, \quad j = 1, \ldots, N.
\end{align*}
This immediately yields \eqref{RFV-1} because of H\"{o}lder's inequality and the following
\begin{align*}
|\A_1(f_1, f_2)(x)|
& \le \int_{\bS^{2d-1}} \prod_{j=1}^N
\big[|f_1(x-y)|^{\frac{p_1}{a_j}} |f_2(x-z)|^{\frac{p_2}{b_j}} \big]^{\lambda_j} \, d\sigma 
\le \prod_{j=1}^N \A_1(|f_1|^{\frac{p_1}{a_j}}, |f_2|^{\frac{p_2}{b_j}})(x)^{\lambda_j}.
\end{align*}

Furthermore, choose $\theta > 0$ sufficiently small so that
\begin{align*}
X_1 := \frac{X - \theta P_2}{1 - \theta}
= \bigg(\frac{1}{q_1}, \frac{1}{q_2}, \frac{1}{q_0} \bigg)
\in \mathcal{R}_{\lac}(d) \quad \text{with } \, q_1, q_2, q_0 > 1,
\end{align*}
which means $\frac{1}{p_i} = \frac{1-\theta}{q_i} + \frac{\theta}{2}$, $i=0, 1, 2$. By \eqref{RFU-1} applied to $P_2 = (\frac12, \frac12, \frac12) \in \mathcal{R}_{\full}(d)$ and \eqref{RFV-1} applied to $X_1$, we have
\begin{align*}
& \|\A_1 (\Delta_h \otimes \Id)\|_{L^2 \times L^2 \to L^2} \lesssim |h|^{\eta},
\\ 
& \|\A_1 (\Delta_h \otimes \Id)\|_{L^{q_1} \times L^{q_2} \to L^{q_0}} \lesssim 1,
\end{align*}
for all $|h| \le 1$. This gives \eqref{RFV-2} by interpolation. The inequality \eqref{RFV-3} can be proved identically.

\subsubsection{\bf{Triangle averaging maximal operators}}
To demonstrate Theorem \ref{thm:tri}, we will invoke Lemma \ref{lem:finite} with $U_t = \mathcal{T}_t$. Obviously, \eqref{H1} holds with $A_1=1$. Define
\begin{align*}
\mathcal{R}_{\tri}^{(2)}(d) 
:= & \conv \bigg\{(0, 0), \, (1, 0), (0, 1), \bigg(\frac{d}{d+1}, \frac{d}{d+1} \bigg)\bigg\}.
\end{align*}
Note that the coordinate projection of $\mathcal{P}_d$ onto its first two coordinates is $\mathcal R_{\tri}^{(2)}(d)$, and 
\begin{align*}
\mathcal{R}_{\tri}^{(2)}(d) 
= \bigg\{ \bigg(\frac{1}{p_1}, \frac{1}{p_2}\bigg) \in [0, 1]^2: \frac{1}{p_1} + \frac{d}{p_2} \le d, \frac{d}{p_1}  + \frac{1}{p_1} \le d\bigg\},
\end{align*}
which gives $(\frac{1}{2r}, \frac{1}{2r}) \in \inter \mathcal{R}_{\tri}^{(2)}(d)$. Thus, \eqref{H2} and \eqref{H3} follow from Lemma \ref{lem:Tlac} and that
\begin{align*}
\bigg \|\sum_{j = N_1}^{N_2} \mathcal{U}_j(f_1, f_2)(x) \bigg\|_{\B_I}
= \max_{j \in I \cap [N_1, N_2]} |\mathcal{T}_{2^j}(f_1, f_2)(x)|
\le \mathcal{T}_{\lac}(f_1, f_2)(x),
\end{align*}
for all integers $N_1\le N_2$.

To prove \eqref{H4} and \eqref{H5}, observe that
\begin{equation*}
\begin{aligned}
& \text{for each } y \in (0, 1) \text{ and all sufficiently small } x > 0, 
\\
& \text{one has } (x, y, y) \in \mathcal{R}_{\tri}(d) \text{ and } (y, x, y) \in \mathcal{R}_{\tri}(d).
\end{aligned}
\end{equation*}
Indeed, setting $A = \frac{d}{d+1}$ and $B = \frac{1}{d+1}$, we have for sufficiently small $t > 0$,
\begin{align*}
(At, y, y) 
= \lambda_0 (0, 0, 0) + \lambda_1 (0, 1, 1) + t (A, 0, B) + \frac{t B}{A-B} (0, A, B),
\end{align*}
where
\begin{align*}
\lambda_1 = y - \frac{tAB}{A-B} > 0
\quad \text{and} \quad
\lambda_0 = 1 - \lambda_1 - t - \frac{tB}{A-B} > 0.
\end{align*}
This allows us to find $r_1^*, r_2^*>1$ such that $(\frac{1}{r_1^*}, \frac{1}{r_2}, \frac{1}{r_2}) \in \mathcal{R}_{\tri}(d)$ and $(\frac{1}{r_1^*}, \frac{1}{r_2}, \frac{1}{r_2}) \in \mathcal{R}_{\tri}(d)$. Thus, by rescaling and the fact $|\mathcal{V}_j (f_1, f_2)(x)|_{\B_I} = |T_{2^j}(f_1, f_2)(x)|$, it suffices to show that for all $(\frac{1}{p_1}, \frac{1}{p_2}, \frac{1}{p_0}) \in \mathcal{R}_{\tri}(d)$, there exists some $\eta \in (0, 1)$ such that 
\begin{align}
\label{RTV-1}
\|\mathcal{T}_1\|_{L^{p_1} \times L^{p_2} \to L^{p_0}}
& \lesssim 1,
\\
\label{RTV-2}
\|\mathcal{T}_1(\Delta_h \otimes \Id)\|_{L^{p_1} \times L^{p_2} \to L^{p_0}}
& \lesssim |h|^{\eta},
\\
\label{RTV-3}
\|\mathcal{T}_1 (\Id \otimes \Delta_h)\|_{L^{p_1} \times L^{p_2} \to L^{p_0}}
& \lesssim |h|^{\eta},
\end{align}
for all $|h| \le 1$. Note that \eqref{RTV-1} is contained in \cite[Theorem~1.1]{IPS22}. To verify \eqref{RTV-2} and \eqref{RTV-3}, we use Sobolev estimates from \cite[Theorem~1.3]{BFO} with $E = \{1\}$:
\begin{align}\label{eq:BFO}
\|\mathcal{T}_1\|_{H^{-\sigma} \times L^2 \to L^2} \lesssim 1
\quad \text{ and } \quad
\|\mathcal{T}_1\|_{L^2 \times H^{-\sigma} \to L^2} \lesssim 1,
\end{align}
where $0 < \sigma < \min\{1, (d-4)/2\}$. Note that 
\begin{align*}
\sup_{\xi \in \Rd} |e^{2 \pi i h \cdot \xi} - 1| (1 + |\xi|^2)^{-\sigma/2} \lesssim |h|^{\sigma}, \quad |h| \le 1,
\end{align*}
which implies $\|\Delta_h f\|_{H^{-\sigma}} \lesssim |h|^{\sigma} \|f\|_{L^2}$, 
\begin{align}\label{RTV-4}
\|\mathcal{T}_1(\Delta_h \otimes \Id)\|_{L^2 \times L^2 \to L^2}
\lesssim |h|^{\sigma}, \, \text{ and } \,
\|\mathcal{T}_1 (\Id \otimes \Delta_h)\|_{L^2 \times L^2 \to L^2}
\lesssim |h|^{\sigma}.
\end{align}
Write $X = (\frac{1}{p_1}, \frac{1}{p_2}, \frac{1}{p_0})$ and $X_1 = (\frac12, \frac12, \frac12)$. Since $X$ is interior, there exists $\theta \in(0, 1)$ sufficiently small so that
\begin{align*}
\bigg(\frac{1}{q_1}, \frac{1}{q_2}, \frac{1}{q_0} \bigg) 
=: X_0 = \frac{X - \theta X_1}{1-\theta} \in \mathcal{R}_{\tri}(d).
\end{align*}
Hence, $X = (1-\theta) X_0 + \theta X_1$. The estimate \eqref{RTV-1} gives
\begin{align}\label{RTV-5}
\|\mathcal{T}_1(\Delta_h \otimes \Id)\|_{L^{q_1} \times L^{q_2} \to L^{q_0}}
\lesssim 1 \, \text{ and } \,
\|\mathcal{T}_1(\Id \otimes \Delta_h)\|_{L^{q_1} \times L^{q_2} \to L^{q_0}}
\lesssim 1.
\end{align}
Accordingly, interpolating between \eqref{RTV-4} and \eqref{RTV-5} yields \eqref{RTV-2} and \eqref{RTV-3} with $\eta = \theta \sigma$.

\begin{lemma}\label{lem:Tlac}
Let $d \ge 5$. Then $\mathcal{T}_{\lac}$ is bounded from $L^{p_1}(\Rd) \times L^{p_2}(\Rd)$ to $L^p(\Rd)$ for all $(\frac{1}{p_1}, \frac{1}{p_2}) \in \inter \mathcal{R}_{\tri}^{(2)}(d)$, where $\frac1p = \frac{1}{p_1} + \frac{1}{p_2}$.
\end{lemma}

\begin{proof}
We begin with the following estimates
\begin{align*}
\|\mathcal{T}_1(\Delta_h \otimes \Id)\|_{L^2 \times L^2 \to L^1}
& \lesssim \min\{1, |h|^{\sigma}\}, \quad\forall h \in \Rd.
\end{align*}
Indeed, the case $|h| > 1$ follows from \eqref{RTV-4}; in the case $|h| \le 1$, after partition $\Rd$ into cubes $\D_0$, we use the Cauchy--Schwarz inequality and the fact
\begin{align*}
\mathcal{T}_1(\Delta_h f_1, f_2)(x) 
= \mathcal{T}_1 \big(\Delta_h (f_1 \mathbf{1}_{5Q}), f_2 \mathbf{1}_{5Q} \big)(x), 
\quad x \in Q \in \D_0, |h| \le 1.
\end{align*}

Choose $\varphi \in \mathscr{S}(\Rd)$ such that $\mathbf{1}_{\{|\xi| \le 1\}} \le \widehat{\varphi} \le \mathbf{1}_{\{|\xi| \le 2\}}$, and set $\widehat{\psi} (\xi) = \widehat{\varphi}(\xi) - \widehat{\varphi}(2\xi)$, $\psi_m (x) = 2^{md} \psi(2^m x)$, and $P_m f = \psi_m*f$. 
Then, $\int_{\Rd} \psi \, dx = 0$ and
\begin{align*}
\mathcal{T}_1(P_m f_1, P_n f_2)
= \int_{\Rd} \psi_m(x) T_1(\Delta_{-v} f_1, P_n f_2) \,dv, \quad \forall m, n \ge 1,
\end{align*}
which along with Minkowski's inequality implies
\begin{align*}
\|\mathcal{T}_1(P_m \otimes P_n)\|_{L^2 \times L^2 \to L^1}
\lesssim \int_{\Rd} |\psi_m(v) |\min\{|v|^{\sigma}, 1\} \, dv
 \lesssim 2^{-m \sigma}.
\end{align*}
Similarly, $\|\mathcal{T}_1(P_m \otimes P_n)\|_{L^2 \times L^2 \to L^1} \lesssim 2^{-n \sigma}$, which indicates 
\begin{align*}
\|\mathcal{T}_1(P_m \otimes P_n)\|_{L^2 \times L^2 \to L^1}
 \lesssim 2^{- \sigma(m+n)/2}.
\end{align*}
In view of the fact $(\frac{1}{p_1}, \frac{1}{p_2}) \in \inter \mathcal{R}_{\tri}^{(2)}(d)$, one can find some $\theta \in (0, 1)$ small enough so that
\begin{align*}
(1/q_1, 1/q_2) := \frac{(1/p_1, 1/p_2) - \theta(1/2, 1/2)}{1-\theta}
\in \inter \mathcal{R}_{\tri}^{(2)}(d).
\end{align*}
Let $\frac{1}{q} = \frac{1}{q_1} + \frac{1}{q_2}$. Note that $\sup_{m \ge 1} \|P_m\|_{L^t \to L^t} \lesssim 1$ for all $t \in [1, \infty]$. Then by \cite[Theorem~1.1]{IPS22},
\begin{align*}
\|\mathcal{T}_1(P_m \otimes P_n)\|_{L^{q_1} \times L^{q_2} \to L^q}
 \lesssim 1,
\end{align*}
which together with bilinear interpolation gives
\begin{align}\label{TPP-1}
\|\mathcal{T}_1(P_m \otimes P_n)\|_{L^{p_1} \times L^{p_2} \to L^p}
\lesssim 2^{- \kappa(m+n)}, \quad \kappa > 0.
\end{align}

Let $\Phi_k f = \varphi_k*f$, where $\varphi_k (x) = 2^{-kd} \varphi(2^{-k} x)$. Then
\begin{align}\label{TP2k-1}
\mathcal{T}_{2^k}(f, g)
& = \mathcal{T}_{2^k}(\Phi_k f, g) + \mathcal{T}_{2^k}(f, \Phi_k g) 
- \mathcal{T}_{2^k}(\Phi_k f, \Phi_k g) 
+ \sum_{m, n \ge 1} \mathcal{T}_{2^k}(P_{m-k}f, P_{n-k}g),
\end{align}
where we used the identity $1 = \widehat{\varphi}(2^k \xi) + \sum_{m \ge 1}\widehat{\psi}(2^{k-m} \xi)$. It is not hard to check 
\begin{align}
& |\mathcal{T}_{2^k}(\Phi_k f, g)| \lesssim Mf \, M_{\lac}(|g|),
\\
& |\mathcal{T}_{2^k}(f, \Phi_k g)| \lesssim M_{\lac}(|f|) \, Mg,
\\
\label{TP2k-2}
& |\mathcal{T}_{2^k}(\Phi_k f, \Phi_k g)| \lesssim Mf \, Mg.
\end{align}
where $M$ is the Hardy--Littlewood maximal operator, and $M_{\mathrm{lac}}$ is the linear lacunary spherical maximal
operator. 

For the remaining terms, define the bilinear sequence operator
\begin{align*}
\mathscr T_N^{m,n}(\{f_k\}, \{g_k\})
:= \big\{\mathcal{T}_{2^k}(P_{m-k}f_k, P_{n-k}g_k) \big\}_{|k| \le N}.
\end{align*}
Note that 
\begin{align*}
|P_{m-k}f_k| \lesssim MF_N
\quad \text{ and } \quad
|P_{n-k} g_k| \lesssim MG_N,
\end{align*} 
where $F_N = \sup_{|k| \le N} |f_k|$ and $G_N = \sup_{|k| \le N} |g_k|$. Besides, by dilation and \cite[Theorem~1.1]{IPS22},
\begin{align*}
A_N := \bigg\|\sup_{|k| \le N} |\mathcal{T}_{2^k}| \bigg\|_{L^{p_1} \times L^{p_2} \to L^p}
\lesssim (2N+1)^{\frac1p} < \infty, \quad \forall N \ge 1,
\end{align*}
which gives
\begin{align*}
& \|\mathscr{T}_N^{m, n}(\{f_k\}, \{g_k\})\|_{L^p(\ell^{\infty})}
\lesssim \bigg\|\sup_{|k| \le N} |\mathcal{T}_{2^k}(M F_N, M G_N)| \bigg\|_{L^p}
\\
& \lesssim A_N \|MF_N\|_{L^{p_1}} \|MG_N\|_{L^{p_2}}
\lesssim A_N \|\{f_k\}\|_{L^{p_1}(\ell^{\infty})}
 \|\{g_k\}\|_{L^{p_2}(\ell^{\infty})}.
\end{align*}
On the other hand, letting $\widetilde{f}_k := f_k(2^{-k} \cdot)$ and $\widetilde{g}_k := g_k(2^{-k} \cdot)$, we use \eqref{TPP-1} and H\"older's inequality to arrive at
\begin{align*}
\|\mathscr{T}_N^{m, n}(\{f_k\}, \{g_k\})\|_{L^p(\ell^p)}
& = \bigg[\sum_{|k| \le N} \|\mathcal{T}_{2^k}(P_{m-k} f_k, P_{n-k} g_k)\|_{L^p}^p \bigg]^{\frac1p}
\\
& = \bigg[\sum_{|k| \le N} \|\mathcal{T}_1(P_{m-k} \widetilde{f}_k, 
P_{n-k} \widetilde{g}_k)(2^{-k} \cdot)\|_{L^p}^p \bigg]^{\frac1p}
\\
& \lesssim 2^{- \kappa(m+n)} 
\bigg[\sum_{|k|\le N}\|f_k\|_{L^{p_1}}^p \|g_k\|_{L^{p_2}}^p \bigg]^{\frac1p}
\\
& \le 2^{- \kappa(m+n)} \bigg[\sum_k \|f_k\|_{L^{p_1}}^{p_1} \bigg]^{\frac{1}{p_1}}
\bigg[\sum_k \|g_k\|_{L^{p_2}}^{p_2} \bigg]^{\frac{1}{p_2}}
\\
& \le 2^{- \kappa(m+n)} \|\{f_k\}\|_{L^{p_1}(\ell^1)} \|\{g_k\}\|_{L^{p_2}(\ell^1)}.
\end{align*}
Then interpolating yields
\begin{align}\label{TPP-2}
\|\mathscr{T}_N^{m, n}(\{f_k\}, \{g_k\})\|_{L^p(\ell^{2p})}
\lesssim 2^{- \kappa(m+n)/2} A_N^{1/2} 
\|\{f_k\}\|_{L^{p_1}(\ell^2)} \|\{g_k\}\|_{L^{p_2}(\ell^2)}.
\end{align}
Define $\widehat{Q_j  f}(\xi) := \widehat{\Phi}(2^{-j} \xi) \widehat{f}(\xi)$, where $\widehat{\Phi}(\xi) := \widehat{\varphi}(\xi/2) - \widehat{\varphi}(4\xi)$. Thus, $P_j Q_j = P_j$, which along with \eqref{TPP-2} applied to 
$f_k = Q_{m-k} f$ and $g_k =Q_{n-k}g$ gives
\begin{align}\label{TPP-3}
& \bigg\|\sup_{|k| \le N} |\mathcal{T}_{2^k}(P_{m-k}f, P_{n-k}g)| \bigg\|_{L^p}
= \bigg\|\sup_{|k| \le N} |\mathcal{T}_{2^k}(P_{m-k}f_k, P_{n-k}g_k)| \bigg\|_{L^p}
\\ \nonumber
& \le \|\mathscr{T}_N^{m, n}(\{f_k\}, \{g_k\})\|_{L^p(\ell^{2p})}
\lesssim 2^{-\kappa(m+n)/2} A_N^{1/2} \|f\|_{L^{p_1}} \|g\|_{L^{p_2}},
\end{align}
where the Littlewood--Paley theory was used in the last step. 
Setting $\eta = \min\{1, p\}$ and invoking \eqref{TP2k-1}--\eqref{TP2k-2} and \eqref{TPP-3}, we obtain
\begin{align*}
A_N^\eta
\le C_0 + C_1 A_N^{\eta/2} \sum_{m, n \ge 1} 2^{-\kappa\eta(m+n)/2}
\le C_0 + C_2 A_N^{\eta/2},
\end{align*}
which together with the finiteness of $A_N$ implies $\sup_N A_N \lesssim 1$. Therefore, the monotone convergence theorem proves $\|\mathcal{T}_{\lac}\|_{L^{p_1} \times L^{p_2} \to L^p} \lesssim 1$.
\end{proof}

\subsection{Bilinear Bochner--Riesz means}\label{sec:App-BR}
Let us justify Theorem \ref{thm:BR}. For each $d \ge 2$, set $a_d = \frac{d+1}{2d}$ and $b_d = \frac{d+1}{2} + \frac{d-1}{d(d+1)}$. Given $q_1, q_2 \in [1, \frac{2d}{d+1})$, define
\begin{equation}\label{def:aqq}
\alpha(q_1, q_2) = 
\begin{cases}
\frac{4}{d+1}, & \text{if } (\tfrac{1}{q_1}, \tfrac{1}{q_2}) \in (a_d, b_d) \times (a_d, b_d),
\vspace{4pt}
\\
\frac{2}{d+1} - \frac{d-1}{2d} + \frac{1}{q_2}, & \text{if } (\tfrac{1}{q_1}, \tfrac{1}{q_2}) \in (a_d, b_d) \times [b_d, 1],
\vspace{4pt}
\\
\frac{2}{d+1} - \frac{d-1}{2d} + \frac{1}{q_1}, & \text{if } (\tfrac{1}{q_1}, \tfrac{1}{q_2}) \in [b_d, 1] \times (a_d, b_d),
\vspace{4pt}
\\
\frac{1}{q_1} + \frac{1}{q_2} - \frac{d-1}{d}, & \text{if } (\tfrac{1}{q_1}, \tfrac{1}{q_2}) \in [b_d, 1] \times [b_d, 1].
\end{cases}
\end{equation}

Choose an even function $\Psi \in \mathscr{C}_c^{\infty}(\R)$ such that $\supp \Psi \subset \{1/4 \le |t| \le1\}$ and $\sum_{j \in \Z} \Psi(2^{-j} t) =1$ for any $t>0$. Set $\Psi_0(t) := 1 - \sum_{j \ge 1} \Psi(2^{-j} t)$ and $\Psi_j(t) := \Psi(2^{-j} t)$ for each $j \ge 1$. Then define
\begin{align*}
& \m_0(s, t) := (1 - s^2 - t^2)_+^{\delta}, 
\\
& \widehat{\m}_{0, j}(s, t) := \Psi_j(|(s, t)|) \widehat{\m}_0(s, t),
\\
& \widehat{K}_j(\xi, \eta) := \m_j(\xi, \eta) := \m_{0, j}(|\xi|, |\eta|), \quad \xi, \eta \in \Rd.
\end{align*}
Hence, we write
\begin{align}\label{BdTm}
B^{\delta} = \sum_{j \ge 0} T_{\m_j},
\end{align}
where
\begin{align*}
T_{\m_j}(f, g)(x) 
& = \int_{\R^{2d}} \m_j(\xi, \eta) \,
\widehat{f}(\xi) \widehat{g}(\eta) \, e^{2 \pi i x \cdot (\xi + \eta)} \, d\xi \, d\eta
\\
& = \int_{\R^{2d}} K_j(x-y, x-z) f(y) g(z) \, dy \, dz.
\end{align*}

Fix $1 \le q_1, q_2 < \frac{2d}{d+1}$ and denote $\mu := \delta + 1 - d \alpha(q_1, q_2) > 0$. Let $(\frac{1}{r_1}, \frac{1}{r_2}), (\frac{1}{s_1}, \frac{1}{s_2}) \in \mathcal{R}(\vec{q})$ be such that $\mu > \max\{\frac1r - \frac1s, \frac{1}{r_1}, \frac{1}{r_2}\} > 0$, where $\frac1r = \frac{1}{r_1} + \frac{1}{r_2}$ and $\frac1s = \frac{1}{s_1} + \frac{1}{s_2}$. We are going to verify that the hypotheses \eqref{H1}--\eqref{H5} are satisfied by the family $\{T_{\m_j}\}_{j \ge 0}$ with exponents $(r_1^*, r_2^*) = (r_1, r_2)$.

It was shown in \cite[Lemma~3.1]{BGSY} that
\begin{equation*}
\begin{aligned}
& \text{for any even function $a_0$ on $\R^2$ with $\supp \widehat{a}_0 \subset [-L, L]^2$,} 
\\
& \text{the function $a(x, y) := a_0(|x|, |y|)$ satisfies $\supp \widehat{a} \subset [-L, L]^{2d}$.} 
\end{aligned}
\end{equation*}
Since $K_j(y, z) = \widehat{\m}_j(-y, -z)$, there holds $\supp K_j(x-y, x-z) \subset \{|x-y| \le 2^j, |x-z| \le 2^j\}$, which implies the support condition \eqref{H1}.

To prove \eqref{H2}, let us record a result from \cite[Lemma~3.9]{BGSY} that for every $m_0 \in L_c^{\infty}(\R^2)$ with $\supp m_0 \subset [-2, 2]^2$,
\begin{equation}\label{Tmws}
\|T_m\|_{L^{p_1} \times L^{p_2} \to L^p}
\lesssim \|m_0\|_{W^{\beta,1}}, \quad \beta > d \alpha(q_1, q_2),
\end{equation}
for all $\big(\frac{1}{p_1}, \frac{1}{p_2} \big) \in \mathcal{R}(\vec{q})$, where $m(\xi, \eta) = m_0(|\xi|, |\eta|)$. Let $\rho_0$ be a smooth cutoff function on $\R^2$ such that $\mathbf{1}_{B(0, \frac32)} \le \rho_0 \le \mathbf{1}_{B(0, 2)}$. Set $\rho(\xi, \eta) := \rho_0(|\xi|, |\eta|)$ for any $\xi, \eta \in \Rd$. Then $\supp(\rho_0 \m_{0, j}) \subset \{|(s, t)| \le 2\}$ and 
\begin{align}\label{rmwb}
\|\rho \m_{0, j}\|_{W^{\beta, 1}}
\lesssim \|\m_{0, j}\|_{W^{\beta, 1}}
\lesssim_{\beta, \eta} \, 2^{-j(\delta + 1 - \beta - \theta)}, \quad \forall \theta > 0, \, 0 < \beta < \delta + 1,
\end{align}
since $\m_0 \in W^{\beta, 1}(\R^2)$ (cf. \cite[Lemma~4.4]{BGSY}) and $\m_{0, j}$ is the $j$-th Littlewood--Paley piece of $\m_0$. Pick 
\begin{align*}
0 < \mu_0 < \mu - \max\bigg\{\frac1r - \frac1s, \frac{1}{r_1}, \frac{1}{r_2} \bigg\}.
\end{align*} 
Applying \eqref{Tmws} and \eqref{rmwb} to $m_0 = \rho \m_{0, j}$ with $\beta = d\alpha(q_1, q_2) + \mu_0/2$ and $\theta = \mu_0/2$, we obtain 
\begin{align}\label{Trm-1}
\|T_{\rho \m_j}\|_{L^{p_1} \times L^{p_2} \to L^p}
\lesssim 2^{-j (\mu - \mu_0)}, \quad \forall (p_1, p_2) \in \Xi(q_1, q_2).
\end{align}
In addition, by Lemmas \ref{lem:NALj} and \ref{lem:TKfg},
\begin{align}\label{Trm-2}
\|T_{(1-\rho) \m_j}\|_{L^{t_1} \times L^{t_2} \to L^{t_0}} \lesssim_N \, 2^{-jN}, \quad\forall N > 0,
\end{align}
for all $\frac{1}{t_0} \le \frac{1}{t_1} + \frac{1}{t_2}$ with $t_1, t_2 \in [1, \infty]$. Thus, \eqref{Trm-1} and \eqref{Trm-2} yield
\begin{align}\label{Trm-3}
\|T_{\m_j}\|_{L^{p_1} \times L^{p_2} \to L^p} \lesssim 2^{-j (\mu - \mu_0)}, 
\quad\forall (p_1, p_2) \in \Xi(q_1, q_2).
\end{align}
Then for all $\big(\frac{1}{p_1}, \frac{1}{p_2} \big) \in \mathcal{R}(\vec{q})$ with $\frac1p = \frac{1}{p_1} + \frac{1}{p_2} \le 1$, \eqref{Trm-3} gives
\begin{align*}
\sup_{N_1 \le N_2} \bigg\|\sum_{j=N_1}^{N_2} T_{\m_j} \bigg\|_{L^{p_1} \times L^{p_2} \to L^p}
\lesssim 1,
\end{align*}
where we used Minkowski's inequality in the case $p \ge 1$ and $(\sum_j |a_j|)^p \le \sum_j |a_j|^p$ in the case $0<p<1$. This implies \eqref{H2} and \eqref{H3}.

To check \eqref{H4}, take $\psi \in \mathscr{C}_c^{\infty}(\Rd)$ such that $\mathbf{1}_{B(0, 2)} \le \psi \le \mathbf{1}_{B(0, 4)}$. Define $\widehat{P f}(\xi) := \psi(\xi) \widehat{f}(\xi)$. Since $\supp(\rho \m_j) \subset B(0, 2)$ and $\widehat{(\Id - P) f} = 0$ on $B(0, 2)$, one has
\begin{align*}
T_{\rho \m_j}((\Id - P)f, g) = 0
\quad \text{and} \quad 
T_{\rho \m_j}(Pf, (\Id - P)g) = 0.
\end{align*}
Using the fact $r < s$, Bernstein's inequality, and \eqref{Trm-1}, we deduce
\begin{align}\label{Trm-41}
\|T_{\rho \m_j}(f, g)\|_{L^s}
& = \|T_{\rho \m_j}(Pf, Pg)\|_{L^s}
\lesssim \|T_{\rho \m_j}(Pf, Pg)\|_{L^r}
\\ \notag
& \lesssim 2^{-j (\mu - \mu_0)} \|Pf\|_{L^{r_1}} \|Pg\|_{L^{r_2}}
\lesssim 2^{-j (\frac1r - \frac1s)} \|f\|_{L^{r_1}} \|g\|_{L^{r_2}}.
\end{align}
Analogously, 
\begin{align}
\label{Trm-42}
& \|T_{\rho \m_j}\|_{L^{r_1} \times L^{r_2^*} \to L^{r_1}} 
\lesssim 2^{-j (\mu - \mu_0)}
\le 2^{-j \frac{1}{r_2}} 
= 2^{-j \frac{1}{r_2^*}},
\\
\label{Trm-43}
& \|T_{\rho \m_j}\|_{L^{r_1^*} \times L^{r_2} \to L^{r_2}}
\lesssim 2^{-j (\mu - \mu_0)}
\le 2^{-j \frac{1}{r_1}} 
= 2^{-j \frac{1}{r_1^*}}.
\end{align}
Thus, \eqref{H4} follows from \eqref{Trm-2} and \eqref{Trm-41}--\eqref{Trm-43}. 

It remains to justify \eqref{H5} holds. Let $|h| \le 1$ and $v := 2^j h$. By Lemma \ref{lem:NALj}, the kernel of $T_{(1-\rho) \m_j}(\Delta_v f, g)(x)$ satisfies
\begin{align*}
|L_j(x-y+v, x-z) - L_j(x-y, x-z)|
\lesssim \frac{\min\{1, |v|\} \, 2^{-jN}}{(1+|x-y|+|x-z|)^M},
\end{align*}
which together with Lemma \ref{lem:TKfg} yields
\begin{align}\label{Trm-51}
& \|T_{(1-\rho) \m_j} \circ (\Delta_v \otimes \Id)\|_{L^{t_1} \times L^{t_2} \to L^{t_0}}
\lesssim_N \min\{1, |v|\} \, 2^{-jN}, \quad \forall N > 0,
\end{align}
for all $\frac{1}{t_0} \le \frac{1}{t_1} + \frac{1}{t_2}$ with $t_1, t_2 \in [1, \infty]$.

Define 
\begin{align*}
\widehat{A_v f}(\xi) := \psi(\xi) \big(e^{2\pi i v \cdot \xi} - 1\big) \widehat{f}(\xi).
\end{align*}
Then by Young's inequality,
\begin{align*}
\|A_v\|_{L^p \to L^p} \lesssim \min\{1, |v|\}, \quad \forall p \in [1, \infty],
\end{align*}
which along with \eqref{Trm-41} implies
\begin{align}\label{Trm-52}
\|T_{\rho \m_j} (\Delta_v f, g)\|_{L^s}
& = \|T_{\rho \m_j} (A_v f, P g)\|_{L^s}
\lesssim 2^{-j (\mu - \mu_0)} \|A_v f\|_{L^{r_1}} \|Pg\|_{L^{r_2}}
\\ \notag
& \lesssim \min\{1, |v|\}\, 2^{-j (\mu - \mu_0)} \|f\|_{L^{r_1}} \|g\|_{L^{r_2}}
\\ \notag
& \le |h|^{\vartheta} 2^{-j (\mu - \mu_0 - \vartheta)} \|f\|_{L^{r_1}} \|g\|_{L^{r_2}}
\\ \notag
& \le |h|^{\vartheta} 2^{-j (\frac1r - \frac1s)} \|f\|_{L^{r_1}} \|g\|_{L^{r_2}},
\end{align}
provided $0 < \vartheta < \min \big\{1, \mu - \mu_0 - \max \{\frac1r - \frac1s, \frac{1}{r_1}, \frac{1}{r_2} \} \big\}$. Similarly, in light of \eqref{Trm-42} and \eqref{Trm-43}, one has
\begin{align*}
& \|T_{\rho \m_j} (\Delta_{2^j h} \otimes \Id)\|_{L^{r_1} \times L^{r_2^*} \to L^{r_1}} 
\lesssim |h|^{\vartheta} 2^{-j \frac{1}{r_2^*}},
\\
& \|T_{\rho \m_j} (\Delta_{2^j h} \otimes \Id) \|_{L^{r_1^*} \times L^{r_2} \to L^{r_2}} 
\lesssim |h|^{\vartheta} 2^{-j \frac{1}{r_1^*}}.
\end{align*}
Thus, \eqref{Trm-51} and \eqref{Trm-52} imply
\begin{align*}
\|T_{\m_j} (\Delta_{2^j h} \otimes \Id)\|_{L^{r_1} \times L^{r_2} \to L^s}
\lesssim |h|^{\vartheta} 2^{-j (\frac1r - \frac1s)}.
\end{align*}
Symmetrically, 
\begin{align*}
\|T_{\m_j} (\Id \otimes \Delta_{2^j h})\|_{L^{r_1} \times L^{r_2} \to L^s}
\lesssim |h|^{\vartheta} 2^{-j (\frac1r - \frac1s)}.
\end{align*}
Note that
\begin{align*}
\Delta_h T_{\m_j}(f_1, f_2) = T_{\m_j}(\Delta_h f_1, f_2) + T_{\m_j}(\tau_h f_1, \Delta_h f_2),
\end{align*}
which gives
\begin{align*}
\|\Delta_{2^j h} T_{\m_j}\|_{L^{r_1} \times L^{r_2} \to L^s}
\lesssim |h|^{\vartheta} 2^{-j (\frac1r - \frac1s)}.
\end{align*}
In the same way, we obtain
\begin{align*}
\|T_{\m_j} \circ_i \Delta_{2^j h}\|_{L^{r_1} \times L^{r_2^*} \to L^{r_1}}
& \lesssim |h|^{\vartheta} 2^{-j \frac{1}{r_2^*}}, \quad i=1, 2,
\\
\|T_{\m_j} \circ_i \Delta_{2^j h}\|_{L^{r_1^*} \times L^{r_2} \to L^{r_2}}
& \lesssim |h|^{\vartheta} 2^{-j \frac{1}{r_1^*}}, \quad i=1, 2,
\\
\|\Delta_{2^j h} T_{\m_j}\|_{L^{r_1} \times L^{r_2^*} \to L^{r_1}} 
& \lesssim |h|^{\vartheta} 2^{-j \frac{1}{r_2^*}},
\\
\|\Delta_{2^j h} T_{\m_j}\|_{L^{r_1^*} \times L^{r_2} \to L^{r_2}} 
& \lesssim |h|^{\vartheta} 2^{-j \frac{1}{r_1^*}}.
\end{align*}
The above justifies \eqref{H5}.

Hence, it follows from Theorem~\ref{thm:main} that
\begin{align*}
\sup_{\substack{N_1, N_2 \in \Z \\ N_1 \le N_2}} 
\bigg| \bigg\langle \sum_{j = N_1}^{N_2} T_{\m_j}(f_1, f_2), \, f_3 \bigg\rangle \bigg|
\lesssim \Lambda_{\vec{r}, s'}^*(f_1, f_2, f_3),
\end{align*}
for all $f_1, f_2, f_3 \in L_c^{\infty}(\Rd)$. In view of \eqref{Trm-3}, this implies Theorem \ref{thm:BR}.

\begin{lemma}\label{lem:NALj}
Let $L_j(y, z) := \mathcal{F}^{-1}_{\xi,\eta} [(1 - \rho) \m_j](y, z)$ for $y, z \in \Rd$. Then for all $N, M > 0$, 
\begin{align*}
|\nabla^s L_j(y, z)|
\lesssim_{N, M} 2^{-jN} (1 + |y| + |z|)^{-M}, \quad \forall j \ge 1, \, s=0, 1.
\end{align*}
\end{lemma}

\begin{proof}
Set $\Phi(s, t) := \Psi (|(s, t)|)$ and $\Phi_j(s, t) := \Phi(2^j s, 2^j t)$. Then $\m_{0, j} = \m_0 * \check{\Phi}_j$, where $\supp(\m_0) \subset \{|\xi| \le 1\}$ and $\check{\Phi}_j(s, t) = 2^{2j} \check{\Phi}(2^j s, 2^j t)$. By the fact $\check{\Phi} \in \mathcal{S}(\R^2)$, for every multi-index $\gamma$ and every $A > 0$,
\begin{align*}
|\partial^{\gamma} \check{\Phi}_j(s, t)|
\lesssim_{\gamma, A} 2^{j(2 + |\gamma|)} (1 + 2^j |(s, t)| )^{-A}, \quad (s, t) \in \R^2,
\end{align*}
which gives that for all for every $N>0$ and $(s, t) \in \R^2$ with $|(s, t)| \ge \frac32$, 
\begin{align*}
|\partial^{\gamma} \m_{0, j}(s, t)|
&= \bigg| \int_{B(0,1)} \m_0(s', t') \partial^{\gamma} \check{\Phi}_j \big((s, t) - (s', t')\big) \, ds' \, dt' \bigg|
\\
&\lesssim 2^{j(2 + |\gamma|)} \int_{B(0,1)} |\m_0(s', t')| [1 + 2^j |(s, t) - (s', t')|]^{-A} \, ds' \, dt'
\\
&\lesssim 2^{j(2 + |\gamma|)} [1 + 2^j (|s| + |t|)]^{-A}
\le 2^{-jN} (1 + |s| + |t|)^{-N},
\end{align*}
provided $A \ge 2N + 2 + |\gamma|$. Recall $\m_j(\xi, \eta) = \m_{0, j}(|\xi|, |\eta|)$, and set $a_j(\xi, \eta) := (1 - \rho(\xi, \eta)) \m_j(\xi, \eta)$. Then for all multi-indices $(\alpha, \beta)$ and $N>0$,
\begin{align*}
|\partial_{\xi}^{\alpha} \partial_{\eta}^{\beta} a_j(\xi, \eta)|
& \lesssim 2^{-j N} (1 + |\xi| + |\eta|)^{-N},
\\
|\partial_{\xi}^{\alpha} \partial_{\eta}^{\beta} [(\xi, \eta) a_j(\xi, \eta)]|
& \lesssim 2^{-j N} (1 + |\xi| + |\eta|)^{-N}.
\end{align*}
This, together with integration by parts and the identity
\begin{align*}
(\Id - \Delta_{\xi})^k e^{i z \cdot \xi} = (1 + |z|^2)^k e^{i z \cdot \xi}, \quad \forall k \in \N,
\end{align*} 
implies
\begin{align*}
|L_j(y, z)|
& = \bigg| \int_{\R^{2d}} 
\frac{a_j(\xi, \eta) \, (\Id - \Delta_{\xi, \eta})^k e^{2\pi i (y \cdot \xi + z \cdot \eta)}}{(1 + 4 \pi^2 |y|^2 + 4 \pi^2 |z|^2)^k} \, d\xi \, d\eta \bigg|
\\
& = \bigg| \int_{\R^{2d}} 
\frac{e^{2\pi i (y \cdot \xi + z \cdot \eta)} (\Id - \Delta_{\xi, \eta})^k a_j(\xi, \eta)}{(1 + 4 \pi^2 |y|^2 + 4 \pi^2 |z|^2)^k} \, d\xi \, d\eta \bigg|
\\
& \lesssim \int_{\R^{2d}} \frac{2^{-j (N+2d)}(1 + |\xi| + |\eta|)^{-(N+2d)}}{(1 + |y| + |z|)^k} \, d\xi \, d\eta
\\
& \lesssim \frac{2^{-j N}}{(1 + |y| + |z|)^k}
\le \frac{2^{-j N}}{(1 + |y| + |z|)^M},
\end{align*}
for any $M>0$, where $k \ge M$ is an integer. The gradient estimate is similar.
\end{proof}


\begin{thebibliography}{99}


\bibitem{Bar}A. Barron, 
\emph{Weighted estimates for rough bilinear singular integrals via sparse domination}, 
N. Y. J. Math. \textbf{23} (2017), 779--811.


\bibitem{BC}D. Beltran and L. Cladek,
\emph{Sparse bounds for pseudodifferential operators},
J. Anal. Math. \textbf{140} (2020), 89--116.


\bibitem{BRS}D. Beltran, J. Roos, and A. Seeger,
\emph{Multi-scale sparse domination},
Mem. Amer. Math. Soc. \textbf{298} (2024), no. 1491.


\bibitem{BBL}C. Benea, F. Bernicot, and T. Luque,
\emph{Sparse bilinear forms for Bochner Riesz multipliers and applications},
Trans. London Math. Soc. \textbf{4} (2017), 110--128.


\bibitem{BFP}F. Bernicot, D. Frey, and S. Petermichl, 
\emph{Sharp weighted norm estimates beyond Calder\'{o}n--Zygmund theory},  
Anal. PDE \textbf{9} (2016), 1079--1113.


\bibitem{BGSY}F. Bernicot, L. Grafakos, L. Song, and L. Yan,
\emph{The bilinear Bochner--Riesz problem},
J. Anal. Math. \textbf{127} (2015), 179--217.


\bibitem{BCS}A. Bhojak, S.S. Choudhary, and S. Shrivastava,
\emph{Bilinear Bochner--Riesz means for convex domains and Kakeya maximal function},
Math. Ann. \textbf{391} (2025), 2281--2318.


\bibitem{BCSS}A. Bhojak, S.S. Choudhary, and S. Shrivastava, and K. Shuin,
\emph{Sharp endpoint $L^p$-estimates for bilinear spherical maximal functions},
https://arxiv.org/abs/2310.00425.


\bibitem{BF}T. Borges and B. Foster,
\emph{Bounds for lacunary bilinear spherical and triangle maximal functions},
J. Fourier Ana. Appl. \textbf{30} (2024), No. 55.


\bibitem{BFO}T. Borges, B. Foster, and Y. Ou,
\emph{Sobolev smoothing estimates for bilinear maximal operators with fractal dilation sets},
J. Funct. Anal. \textbf{288} (2025), 110694.


\bibitem{BFOPZ}T. Borges, B. Foster, Y. Ou, J. Pipher, and Z. Zhou,
\emph{Sparse bounds for the bilinear spherical maximal function},
J. London. Math. Soc. (2) \textbf{107} (2023), 1409--1449.


\bibitem{CIRXY}M. Cao, G. Iba\~{n}ez-Firnkorn, I.P. Rivera-R\'{i}os, Q. Xue, and K. Yabuta,
\emph{A class of multilinear bounded oscillation operators on measure spaces and applications},
Math. Ann. \textbf{388} (2024), 3627--3755.


\bibitem{CXY}M. Cao, Q. Xue, and K. Yabuta,
\emph{Weak and strong type estimates for the multilinear pseudo-differential operators},
J. Funct. Anal. \textbf{278} (2020), 108454. 


\bibitem{P.C}A.P. Calder\'{o}n, 
\emph{Commutators of singular integral operators}, 
Proc. Natl. Acad. Sci. USA \textbf{53} (1965), 1092--1099. 


\bibitem{C.C}C. Calder\'{o}n, 
\emph{On commutators of singular integrals},  
Studia Math. \textbf{53} (1975), 139--174. 


\bibitem{CO}L. Cladek and Y. Ou,
\emph{Sparse domination of Hilbert transforms along curves},
Math. Res. Lett. \textbf{25} (2018), 415--436.


\bibitem{CCDO}J.M. Conde-Alonso, A. Culiuc, F. Di Plinio, and Y. Ou,
\emph{A sparse domination principle for rough singular integrals}, 
Anal. PDE \textbf{10} (2017), 1255--1284.


\bibitem{CDO}A. Culiuc, F. Di Plinio, and Y. Ou,
\emph{Domination of multilinear singular integrals by positive sparse forms},
J. Lond. Math. Soc. \textbf{98} (2018), 369--392.


\bibitem{DS}G. Dosidis and L. Slav\'{i}kov\'{a},
\emph{Multilinear singular integrals with homogeneous kernels near $L^1$},
Math. Ann. \textbf{389} (2024), 2259--2271.


\bibitem{DGHST}G. Diestel, L. Grafakos, P. Honz\'{i}k, Z. Si, and E. Terwilleger,
\emph{Method of rotations for bilinear singular integrals}, 
Commun. Math. Anal. 2011, Conference 3, 99--107.


\bibitem{FS}C. Fefferman and E.M. Stein,
\emph{Some maximal inequalities},
Amer. J. Math. \textbf{93} (1971), 107--115.


\bibitem{Gra1}L. Grafakos, 
\emph{Classical Fourier Analysis}, 
Third edition. GTM, 249. Springer, New York, 2014. 


\bibitem{GHH}L. Grafakos, D. He, and P. Honz\'{i}k, 
\emph{Rough bilinear singular integrals},
Adv. Math. \textbf{326} (2018), 54--78.


\bibitem{GHHP23}L. Grafakos, D. He, P. Honz\'{i}k, and B.J. Park,
\emph{Initial $L^2 \times \cdots \times L^2$ bounds for multilinear operators},
Trans. Amer. Math. Soc. \textbf{376} (2023), 3445--3472.


\bibitem{GHHP24}L. Grafakos, D. He, P. Honz\'{i}k, and B.J. Park,
\emph{Multilinear rough singular integral operators},
J. London Math. Soc. \textbf{109} (2024), No. e12867. 


\bibitem{GMNT}L. Grafakos, A. Miyachi, H.V. Nguyen, and N. Tomita,
\emph{Multilinear Fourier multipliers with minimal Sobolev regularity, II},
J. Math. Soc. Japan \textbf{69} (2017), 529--562.


\bibitem{GWX}L. Grafakos, Z. Wang, and Q. Xue,
\emph{Sparse domination and weighted estimates for rough bilinear singular integrals},
J. Fourier Anal. Appl. \textbf{28} (2022), No. 80.


\bibitem{HP}D. He and B.J. Park,
\emph{Improved estimates for bilinear rough singular integrals},
Math. Ann. \textbf{386} (2023), 1951--1978.


\bibitem{Hu}B. Hu,
\emph{Sparse domination of singular Radon transform},
J. Math. Pures Appl. (9) \textbf{139} (2020), 235--316.


\bibitem{IPS20}A. Iosevich, E.A. Palsson, and S.R. Sovine,
\emph{The triangle averaging operator},
J. Funct. Anal. \textbf{279} (2020), 108671.


\bibitem{IPS22} A. Iosevich, E.A. Palsson, and S.R. Sovine,
\emph{Simplex averaging operators: quasi-Banach and $L^p$-improving bounds in lower dimensions},
J. Geom. Anal. \textbf{32} (2022), 87.


\bibitem{JL}E. Jeong and S. Lee,
\emph{Maximal estimates for the bilinear spherical averages and the bilinear Bochner--Riesz operators},
J. Funct. Anal. \textbf{279} (2020), 108629.


\bibitem{JS}K. Jotsaroop and S. Shrivastava,
\emph{Maximal estimates for bilinear Bochner--Riesz means},
Adv. Math. \textbf{395} (2022), No. 108100.


\bibitem{KMT}T. Kato, A. Miyachi, and N. Tomita,
\emph{Boundedness of bilinear pseudo-differential operators of
$S_{0,0}$-type in Wiener amalgam spaces and in Lebesgue spaces},
J. Math. Anal. Appl. \textbf{515} (2022), No. 126382.


\bibitem{KS}T. Kato and N. Shida, 
\emph{A remark on the condition $1/p = 1/p_1 + 1/p_2$ for boundedness of bilinear pseudo-differential operators with exotic symbols}, 
Hokkaido Math. J. \textbf{52} (2023), 285--300.


\bibitem{KL}R. Kesler and M.T. Lacey,
\emph{Sparse endpoint estimates for Bochner--Riesz multipliers on the plane},
Collect. Math. \textbf{69} (2018), 427--435.


\bibitem{KraL}B. Krause and M.T. Lacey,
\emph{Sparse bounds for maximally truncated oscillatory singular integrals},
Ann. Sc. Norm. Super. Pisa Cl. Sci. (5) \textbf{20} (2020), 415--435.


\bibitem{Lac}M.T. Lacey,
\emph{Sparse bounds for spherical maximal functions},
J. Anal. Math. \textbf{139} (2019), 613--635.


\bibitem{LMR}M.T. Lacey, D. Mena, and M.C. Reguera,
\emph{Sparse bounds for Bochner–Riesz multipliers},
J. Fourier Anal. Appl. \textbf{25} (2019), 523--537.


\bibitem{LS}M.T. Lacey and S. Spencer,
\emph{Sparse bounds for oscillatory and random singular integrals},
New York J. Math. \textbf{23} (2017), 119--131.


\bibitem{Ler}A.K. Lerner, 
\emph{On pointwise estimates involving sparse operators},
New York J. Math. \textbf{22} (2016), 341--349.


\bibitem{LN}A.K. Lerner and F. Nazarov,
\emph{Intuitive dyadic calculus: the basics},
Expo. Math. \textbf{37} (2019), 225--265.


\bibitem{LMO}K. Li, J.M. Martell, and S. Ombrosi,
\emph{Extrapolation for multilinear Muckenhoupt classes and applications to the bilinear Hilbert transform},
Adv. Math. \textbf{373} (2020), 107286.


\bibitem{MT13}A. Miyachi and N. Tomita,
\emph{Calder\'{o}n--Vaillancourt-type theorem for bilinear operators},
Indiana Univ. Math. J. \textbf{62} (2013), 1165--1201.


\bibitem{MT20}A. Miyachi and N. Tomita,
\emph{Bilinear pseudo-differential operators with exotic symbols},
Ann. Inst. Fourier (Grenoble) \textbf{70} (2020), 2737--2769.


\bibitem{Nie}Z. Nieraeth, 
\emph{Quantitative estimates and extrapolation for multilinear weight classes}, 
Math. Ann. \textbf{375} (2019), 453--507.


\bibitem{Obe}R. Oberlin,
\emph{Sparse bounds for a prototypical singular Radon transform},
Canad. Math. Bull. \textbf{62} (2019), 405--415.


\bibitem{PS}E.A. Palsson and S.R. Sovine,
\emph{Sparse bounds for maximal triangle and bilinear spherical averaging operators},
J. Fourier Anal. Appl. \textbf{32} (2026), No. 2.


\bibitem{RS}S. Rodr\'{i}guez-L\'{o}pez and W. Staubach, 
\emph{Estimates for rough Fourier integral and pseudodifferential operators and applications to the boundedness of multilinear operators},
J. Funct. Anal. \textbf{264} (2013), 2356--2385.


\bibitem{Stein}E.M. Stein, 
\emph{Singular Integrals and Differentiability Properties of Functions},
Princeton Mathematical Series, No. 30, 
Princeton University Press, Princeton, NJ, 1970.

\end{thebibliography}
\end{document}